\documentclass[10pt]{article}       
\usepackage{geometry}               
\usepackage{graphicx}
\usepackage{caption}
\usepackage{subcaption}
\usepackage{mathrsfs}
\usepackage{comment}
\usepackage{amsmath} 
\usepackage{amssymb}  
\usepackage{amsthm}  
\usepackage{bm}
\usepackage{color}
\usepackage{array}
\usepackage{hyperref}

\makeatletter
\def\namedlabel#1#2{\begingroup
    #2%
    \def\@currentlabel{#2}%
    \phantomsection\label{#1}\endgroup
}

\usepackage{enumitem}

\usepackage[utf8]{inputenc}
\usepackage{booktabs} 

\usepackage{float}

\usepackage{pgfplots}
\pgfplotsset{compat=1.18}
\usepgfplotslibrary{fillbetween}

\newtheorem{theorem}{Theorem}[section]
\newtheorem{lemma}[theorem]{Lemma}

\newtheorem{proposition}[theorem]{Proposition}

\numberwithin{equation}{section}

\newcommand{\R}{\mathbb{R}}

\newcommand{\cA}{\mathcal{A}}

\newcommand{\cC}{\mathcal{C}}
\newcommand{\cF}{\mathcal{F}}
\newcommand{\cH}{\mathcal{H}}

\newcommand{\cP}{\mathcal{P}}

\newcommand{\cE}{\mathcal{E}}

\def\le{\leqslant}
\def\leq{\leqslant}

\def\geq{\geqslant}

\def\pa{\partial}
\def\na{\nabla}
\def\eps{\varepsilon}
\def\div{\mathrm{div}\, }

\title{Cell-cell adhesion and multiphase Hele--Shaw problem as the singular limit of a Keller--Segel system}

\newcommand{\footremember}[2]{%
    \footnote{#2}
    \newcounter{#1}
    \setcounter{#1}{\value{footnote}}%
}

\author{
  Jiwoong Jang \footremember{trailer}{Department of Mathematics, University of Maryland-College Park (jjang124@umd.edu).},
  Antoine Mellet \footremember{alley}{Department of Mathematics, University of Maryland-College Park (mellet@umd.edu). Partially supported by NSF Grant DMS-2307342.
}
  }

\date{}
\providecommand{\keywords}[1]{\textbf{Key words.} #1}
\providecommand{\MSC}[1]{\textbf{MSC codes.} #1}

\begin{document}
\maketitle

\pagestyle{myheadings}
\thispagestyle{plain}

\begin{abstract}

We investigate a singular limit of a system of Patlak--Keller--Segel (PKS) equations modeling the evolution of multiple interacting species. Our primary motivation is the Differential Adhesion Hypothesis (DAH), introduced by Malcolm Steinberg in 1962, which posits that cell populations self-organize by minimizing adhesion energy, in a manner analogous to fluids minimizing surface tension.
Our starting point is a continuum model describing the evolution of the density distributions of $N$ distinct species, representing different cell types. These species interact through attractive nonlocal forces governed by interaction kernels of similar form but different strengths (the $N\times N$ matrix of interaction coefficients encodes the key properties of the system).
These attractive forces are balanced by a nonlinear pressure, depending on the total density and strong enough to prevent concentration. 

In the limit of short-range interactions, we establish sufficient conditions on the interaction matrix for  cell sorting  to take place (i.e., the spontaneous separation of the different species). We then prove a general $\Gamma$-convergence result for the associated energy functional, showing convergence to an interfacial energy where the surface tension coefficients are determined by a geodesic problem. 
A detailed analysis of this problem allows us to identify regimes in which engulfment occurs (more adhesive species cluster together and are surrounded by less adhesive ones) which is a key feature of the DAH.

In the second part of the paper, we analyze the asymptotic behavior of solutions to the PKS system in the combined long-time and short-range interaction limit. Under a standard energy convergence assumption, we prove the convergence to a multiphase Hele--Shaw problem with surface tension (and contact angle conditions at triple junctions).
\end{abstract}


\vspace{0.2cm}

\MSC{35A15, 35K57, 53E10, 76T30.}

\section{Introduction}\label{sec:introduction}

\subsection{The multispecies PKS model}\label{subsec:summary-intro}
Several recent works  have investigated the limit $\eps\to0$ of the rescaled Patlak--Keller--Segel (PKS) model for chemotaxis:
\begin{equation}\label{eq:general-PKS}
\begin{cases}
\eps \partial_t \rho - \Delta \rho^m + \alpha \operatorname{div} (\rho  \nabla \phi ) = 0, & \text{in }\Omega\times(0,\infty), \\
 \sigma\phi - \eps^2 \Delta \phi = \rho . & \text{in }\Omega .
\end{cases}
\end{equation}
When $m>2$, the limit $\eps\to 0$ leads to phase separation, that is, $\rho^\eps (x,t) \to \theta \chi_{E(t)}$ for some constant $\theta>0$ (depending only on $m$ and the coefficients $\sigma>0$ and $\alpha>0$) and the limiting set $E(t)$ can be proved to evolve according to the Hele--Shaw free boundary problem with surface tension (see \cite{D18,GL98,KMW24,M24,KMW242}).

Our goal   is to extend this analysis to a system of coupled Patlak--Keller--Segel equations describing multiple species interacting through chemotaxis and cross-diffusion.
For each $i=1,\dots,N$, we denote by $\rho_i(x,t)$ the density distribution of the $i$-th species of living organisms (hereafter referred to as “cells”). Its evolution is described by:
\begin{equation}
 \varepsilon \partial_t \rho_i^{\varepsilon} - \operatorname{div}\left(\rho_i^{\varepsilon}\nabla f_m'\left(\rho_1^{\varepsilon}+\dots +\rho_N^\eps\right)\right) +  \operatorname{div} \left(\rho_i^{\varepsilon} \left(\sum_{j=1}^N\alpha_{ij}\nabla \phi^{\varepsilon}_j\right)\right) = 0  \qquad \text{in } \Omega \times (0, \infty),  \label{eq:epsilon-problem-1} 
\end{equation}
where $\phi^\eps_j(\cdot,t)$ denotes the chemotaxis potential associated to the $j$-th species, solution of 
\begin{equation}
\sigma\phi^{\varepsilon}_j - \varepsilon^2 \Delta \phi^{\varepsilon}_j = \rho^{\varepsilon}_j  \qquad \text{in } \Omega , \quad j=1,\dots N .\label{eq:epsilon-problem-2}
    \end{equation} 
The evolution of the density of the $i$-th species is thus coupled to that of the other species through two mechanisms. First, the cross-diffusion term models the tendency of cells to move away from crowded regions via Darcy’s law. It depends on the total density 
    $\rho^{\varepsilon}:=\rho^{\varepsilon}_1 + \cdots + \rho_N^{\varepsilon}$ through the simple power-law pressure:
\begin{equation}\label{eq:f-s}
f_m(s) = \frac{s^m}{m-1}, \qquad m>2,
\end{equation}
so that when $N=1$, this model reduces to \eqref{eq:general-PKS}.
Second, 
the terms $+ \alpha_{ij}\operatorname{div} (\rho_i^{\varepsilon}  \nabla \phi_j^{\varepsilon})$ describe the interactions between the $i$-th and $j$-th species. In the context of chemotaxis, the potentials $\phi_j^{\varepsilon}$ are often interpreted as the density of some chemoattractant (produced by the $j$-th cells), solution of the elliptic equation \eqref{eq:epsilon-problem-2}.
We can write $\phi_j^{\varepsilon}(x)=\int_\Omega G_\eps(x,y)\rho_j^{\varepsilon}(y)\, dy$
for some Green function $G_\eps$ independent of $j$. In particular the chemotaxis mechanisms are identical and the properties of the system depend on the coefficients  $\alpha_{ij}\in\R$ which represent the strength of the interactions between the $i$-th and $j$-th cells. These interactions can be attractive ($\alpha_{ij}>0$) or repulsive ($\alpha_{ij}<0$) 
but we will always assume that self-interactions are attractive ($\alpha_{ii}>0$, see Assumptions \ref{assumption:A1} and \ref{assumption:A2} below).
A key role in our analysis is played by the interaction coefficients matrix 
$$A:=(\alpha_{ij})_{1\leq i,j\leq N}.$$
Finally,  the system is set on a fixed bounded domain $\Omega$ and supplemented with Neumann boundary conditions and initial conditions:
$$
\left(-\rho_i^{\varepsilon}\nabla f_m'\left(\rho^{\varepsilon}_1 + \cdots + \rho_N^{\varepsilon}\right) + \rho_i^{\varepsilon} \left(\sum_{j=1}^N\alpha_{ij}\nabla \phi^{\varepsilon}_j\right)\right) \cdot n = 0\qquad  \text{on } \partial \Omega \times (0, \infty),
$$ 
$$
\nabla \phi^{\varepsilon}_i \cdot n = 0 \qquad  \text{on } \partial \Omega,
$$
and
$$
\rho^{\varepsilon}_i(x, 0) = \rho^{\varepsilon}_{i,in}(x)\qquad  \text{in } \Omega\qquad\text{for each }i=1,\cdots,N.$$
These conditions imply in particular that \eqref{eq:epsilon-problem-1} preserves the mass of each species so we have 
\begin{equation}\label{eq:mass}
\int_\Omega \rho^{\varepsilon}_i(x, t)\, dx = \int_\Omega\rho^{\varepsilon}_{i,in}(x) \, dx =:m_i \qquad \forall t>0.
\end{equation}

The primary objective of this paper is to show that differences in the interaction coefficients $\alpha_{ij}$ lead to cell sorting and to emergent collective behavior (large-scale patterns arising from simple pairwise interactions).
This will be done by studying the singular limit $\eps\to0$ of the system \eqref{eq:epsilon-problem-1}-\eqref{eq:epsilon-problem-2}, which can be interpreted as a long-time/large-scale asymptotic limit: the system \eqref{eq:epsilon-problem-1}-\eqref{eq:epsilon-problem-2} can be obtained via a rescaling $(x,t)\mapsto(\varepsilon x,\varepsilon^3t)$ assuming that the microscopic domain is large (so that $\Omega$ has size of order $1$ after the rescaling) and that the cell populations are large (so that the $m_i$ are of order $1$ after this rescaling). 
We refer to Section \ref{subsec:PKS-system} for further details about the role of $\eps$  in \eqref{eq:epsilon-problem-1}-\eqref{eq:epsilon-problem-2}.

Our first step is to identify sufficient conditions on the interaction  matrix $A=(\alpha_{ij})$ under which
the system evolves from overlapping density distributions $\rho_i$ to segregated regions $E_i(t)$ where $\rho_i>0$ and $\rho_j=0$ for $j\neq i$ (see Assumptions \ref{assumption:A1}-\ref{assumption:A2} below). 
We then perform an asymptotic analysis of the associated energy functional and prove  that it $\Gamma$- converges to an interfacial energy functional where the surface tension coefficients are determined by the length of some geodesics in the phase plane (for an appropriate metric).

This result validates the Differential Adhesion Hypothesis (DAH), introduced by Malcolm Steinberg in 1962 \cite{S62a,S62b,S62c,S62d}, which states that 
cells in a tissue behave like molecules in a fluid: they move and adhere to each other with different strengths and these  differences drive the system toward a configuration that minimizes an interfacial surface energy.
Our analysis rigorously connects the system \eqref{eq:epsilon-problem-1}-\eqref{eq:epsilon-problem-2}
to the regimes predicted by the DAH: mixing, complete sorting, partial engulfment, and full engulfment. In particular, a careful analysis of the auxiliary geodesic problem allows us to identify conditions on the interaction matrix $A$ that lead to engulfment (which is a state where   more adhesive species cluster together and are surrounded by less adhesive ones).
 
Finally, we derive an effective evolution equation for the segregated regions $E_i(t)$:
Building on recent results for the one-species model \cite{KMW23,KMW24,M24}, we prove that, in the sharp-interface limit $\eps\to0$, the evolution of the sets $E_i(t)$ is governed by the multiphase Hele--Shaw free boundary  problem with surface tension, thereby providing a continuum description of the dynamics driven by adhesion.

The emergence of collective behavior in multicellular systems as a result of cell-cell interactions is an actively studied topic. Our work is particularly motivated by \cite{FBC22}, where a Cahn-Hilliard approximation is used to derive a local continuum description of adhesion-driven dynamics. In contrast, we retain here the nonlocal structure of the underlying PKS-type model.

\medskip

\subsection{Main Assumptions}
We recall that $f_m$ is given by \eqref{eq:f-s} with $m>2$ and we make the following assumptions on the matrix $A$:

\medskip
\medskip

\namedlabel{assumption:A1}{(A1)} The matrix $A=(\alpha_{ij})_{1\leq i,j\leq N}$ is symmetric and positive semidefinite.
\medskip
\medskip

\namedlabel{assumption:A2}{(A2)} No two columns of the matrix $A$ are identical (i.e., 
for all $i,j$, there exists $k$ such that $\alpha_{ki}\neq\alpha_{kj}$) and no columns of $A$ are identically $0$ (i.e., 
for all $i$, there exists $k$ such that $\alpha_{ki}\neq0$).

\medskip

Assumption \ref{assumption:A1} plays a crucial role in establishing that the energy associated to our system is bounded from below and admits a $\Gamma$-limit. 
It implies in particular that cross-species interactions are symmetric and that self-interactions dominate the dynamics, as expressed by the inequalities
\begin{equation}\label{eq:detpositive}
\alpha_{ij}^2 \leq \alpha_{ii} \alpha_{jj}
\end{equation}
(which are obtained by writing the determinant condition for the positive semidefinite 2 by 2 submatrix formed by the $i$-th and $j$-th rows and columns of $A$).
It also implies that $\overrightarrow{e_i}^T A \overrightarrow{e_i} = \alpha_{ii}\geq 0$.  Using the nonzero column condition from Assumption \ref{assumption:A2} together with \eqref{eq:detpositive}, we deduce
$$\alpha_{ii} >0 \qquad\forall i=1,\dots ,N,$$
which indicates that all self-interactions are of attractive nature.

Assumption \ref{assumption:A2} rules out indistinguishable species in terms of their interaction coefficients and ensures that each species exhibits some nontrivial interactions, either self-interactions or interactions with other species.
This is not a restrictive requirement: if two species $i$ and $j$ were identical in this sense, they could be merged into a single species by replacing $\rho_i$ and $\rho_j$ with the combined density $\rho_i + \rho_j$. 
Indeed, if the first two columns of $A$ are equal, then the functions $(\rho_1+\rho_2, \rho_3 , \dots ,\rho_N)$ solve the reduced system of equations\footnote{Note in particular that when the first two columns of $A$ are equal, then we have  $\alpha_{11}=\alpha_{12}$ and $\alpha_{21}=\alpha_{22}$. The symmetry assumption then implies that $\alpha_{11}=\alpha_{12}=\alpha_{21}=\alpha_{22}$.} with $N-1$ species and the associated coefficient matrix obtained by removing the first row and column from $A$.

Under these assumptions, we will prove that in the limit $\eps\to0$, the system \eqref{eq:epsilon-problem-1}-\eqref{eq:epsilon-problem-2} leads to the formation of disjoint cell aggregates: Each species density function will satisfy
$$ \rho_i^\eps(x,t)\to  \theta_i \chi_{E_i(t)} (x) \qquad \mbox{ as } \eps\to0$$ 
for some family of disjoint sets $E_i(t)$ and with $\theta_i$ given by
$$ \theta_i:=\left(\frac{1}{2\sigma}\alpha_{ii}\right)^{\frac{1}{m-2}}>0  \qquad i=1,\cdots N,$$
the volume of these sets $E_i(t)$ being then prescribed by the mass condition \eqref{eq:mass}.
Of course, this requires that the domain $\Omega$ be large enough to accommodate all the cell aggregates: We will always assume that the domain $\Omega$ satisfies
\begin{align}\label{assumption:vacuum}
\sum_{i=1}^N\frac{m_i}{\theta_i}<|\Omega|.
\end{align}

\medskip

\subsection{Main results}\label{sec:settings-main-results}
The PKS system   \eqref{eq:epsilon-problem-1}-\eqref{eq:epsilon-problem-2} is classically the gradient flow of the energy
\begin{align}\label{eq:energy-J}
\mathcal{J}^{\varepsilon}(\rho_1^{\varepsilon},\dots , \rho_N^{\varepsilon})
&:=\frac{1}{\varepsilon}\int_{\Omega}f_m(\rho_1^{\varepsilon}+\cdots +\rho_N^{\varepsilon} )  -\frac12\sum_{i,j=1}^N\alpha_{ij}\rho_i^{\varepsilon}\phi_{j}^{\varepsilon}\,dx 
\end{align}
with respect to the Wasserstein distance. 
In what follows, it will be convenient to use the vector notation
$$
\overrightarrow{\rho}  := ({\rho_1} , \dots ,{\rho_N} )^T\,:\,\Omega\to\R^N
$$
to describe the state of the system.
We also introduce the function $f:\R^N\to\R\cup\{+\infty\}$ defined by 
\begin{equation}\label{eq:f}
f(\overrightarrow{\rho})= 
\begin{cases}
    f_m(\rho_1+\cdots +\rho_N) & \mbox{ if } \overrightarrow{\rho}\in\R^N_{\geq0}\\
    +\infty & \mbox{  otherwise.}
\end{cases}
\end{equation}
With these notations, the energy \eqref{eq:energy-J} can be written as
\begin{align*}
\mathcal{J}^{\varepsilon}(\overrightarrow{\rho})
=\frac{1}{\varepsilon}\int_{\Omega}f(\overrightarrow{\rho})-\frac12(\overrightarrow{\rho})^TA\overrightarrow{\phi}^\eps\,dx
\end{align*}
where $\overrightarrow{\phi}^\eps:=(\phi_1^\eps,\cdots,\phi_N^\eps)^T$ and $\phi_i^\eps$ solution of \eqref{eq:epsilon-problem-2}, that is,
\begin{equation}\label{eq:phi}
\begin{cases}
    \sigma\overrightarrow{\phi}^\eps - \eps^2 \Delta \overrightarrow{\phi}^\eps = \overrightarrow{\rho} & \mbox{ in } \Omega\\
    \na \overrightarrow{\phi}^\eps \cdot n  = 0 & \mbox{ on } \pa\Omega.
\end{cases}
\end{equation}

Next, we recall that the mass of each species  is fixed (it is preserved by the equation) and denoted by $m_i$. We then write
$$
\overrightarrow{m} = (m_1,\dots , m_N)^T.
$$
Before passing to the limit $\eps\to 0$, we need to make sure that the minimum of the energy (with the mass constraint) is zero. As in the one-species model (see  \cite{KMW24,M24}), this  requires adding a constant $\overrightarrow{a}\cdot\overrightarrow{m} = \int_\Omega \overrightarrow{a}\cdot\overrightarrow{\rho}(x)\, dx $ to the energy, with 
$$ \overrightarrow{a} = (a_1,\dots a_N)^T, \qquad 
a_i:=
\frac{m-2}{m-1}\left(\frac{1}{2\sigma}\alpha_{ii}\right)^{\frac{m-1}{m-2}}.   
$$
We thus  set
\begin{equation}\label{eq:definition-cE}
\cE^{\varepsilon}(\overrightarrow{\rho}):= 
\begin{cases} 
\displaystyle \frac{1}{\varepsilon} \int_{\Omega} f(\overrightarrow{\rho}(x))+ \overrightarrow{a}\cdot\overrightarrow{\rho}(x) - \frac12 \overrightarrow{\rho}^T(x)A \overrightarrow{\phi}^{\varepsilon}(x) \, dx \qquad & \text{if }  \int_{\Omega} \overrightarrow{\rho}\,dx=\overrightarrow{m}, \\ 
+\infty & \text{otherwise}.
\end{cases}
\end{equation}

\medskip

Next, we define the limiting energy: When $\eps\to0$, the energy \eqref{eq:definition-cE} leads to phase separation.
We thus introduce the admissible set:
\begin{align*}
\cF(A,\overrightarrow{m})&:=\left\{\overrightarrow{\rho} =(\theta_1\chi_{E_1},\dots ,\theta_N\chi_{E_N})^T \, \Big|\, 
\chi_{E_i}\in BV(\Omega,\{0,1\})\text{ for }i=1,\dots,N,\right. \\
&\qquad\qquad\qquad\qquad\qquad\qquad\qquad\qquad \left.\{E_i\}_{i=1}^N\text{ disjoint, and}\int_{\Omega}\overrightarrow{\rho}(x)\,dx=\overrightarrow{m}\right\}
\end{align*}
where we recall that $\theta_i=\left(\frac{1}{2\sigma}\alpha_{ii}\right)^{\frac{1}{m-2}}$.
This set describes all configurations in which the $i$-th species aggregates on a set $E_i$ (with fixed density $\theta_i$ and total mass $m_i$) without mixing with the other species ($E_i\cap E_j=\emptyset$ for all $i\neq j$). 

Note that we work here within the framework of sets of finite perimeter: the assumption $\chi_{E_i}\in BV(\Omega,\{0,1\})$ implies that the perimeter of $E_i$ in $\Omega$ is finite:
$$P(E_i,\Omega) = \int_\Omega |\na \chi_{E_i}|<\infty.$$
We  denote by   $\partial^*E_i$ the reduced boundary of $E_i$ in $\Omega$ and by  $\Sigma_{ij} :=\partial^*E_i \cap \partial^*E_j$ the interface where  the $i$-th and $j$-th species are in contact. 
To simplify the notations below, we also introduce the vacuum set $E_0=\Omega \setminus \cup_{i=1}^N E_i$ and denote by
$\Sigma_{0i} := (\partial^*E_0 \cap \partial^*E_i)$ the free surface of the $i$-th species (see Figure \ref{fig:phases}). 

\begin{figure}[htbp]
	\begin{center}
            \includegraphics[height=6cm]{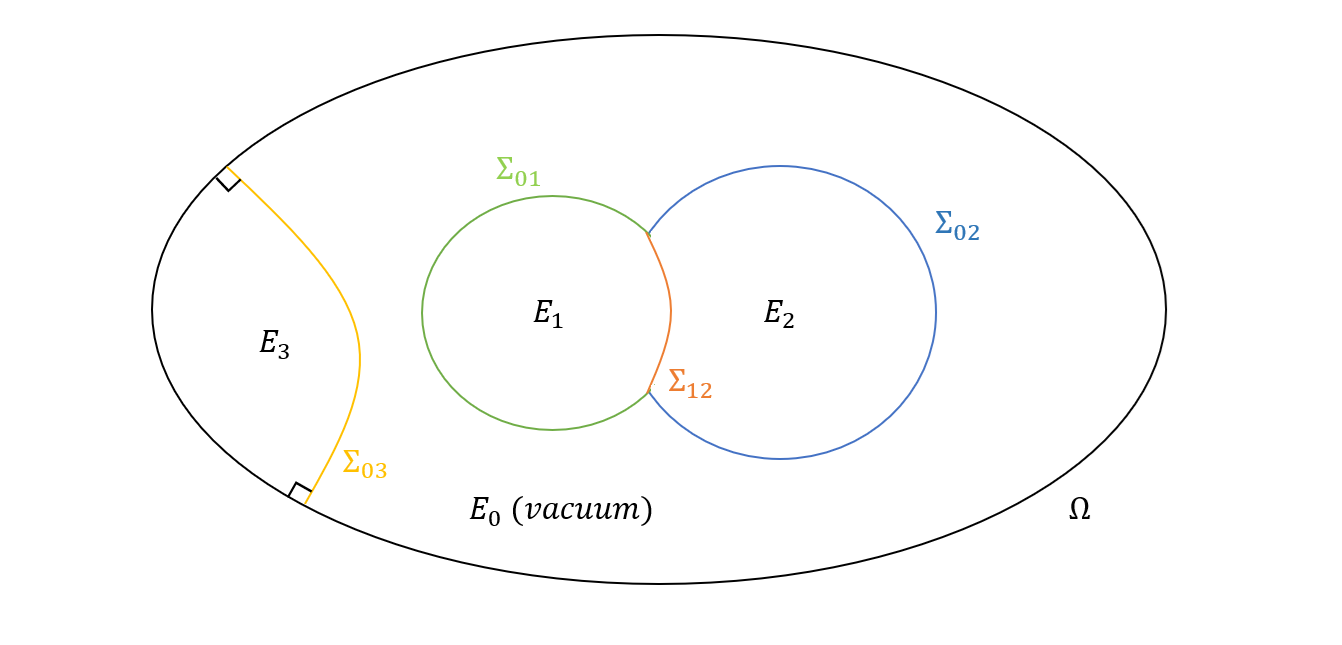}
            \captionsetup{width=0.8\textwidth} 
		\vskip 0pt
		\caption{Phases $E_i$ and their interfaces $\Sigma_{ij}$. The ninety-degree contact angle condition to $\partial\Omega$ (see $\Sigma_{03}$) comes from the Neumann boundary condition.}
        \label{fig:phases}
	\end{center}
\end{figure}

Our first main result, Theorem \ref{thm:gamma-convergence} below, identifies the $\Gamma$-limit of $\cE^\eps$ as the following interfacial energy functional for multiphases: 
\begin{equation}\label{eq:limit-energy-perimeter}
\cE^{0}(\overrightarrow{\rho}):= 
\begin{cases} 
\displaystyle \sum_{i,j=0}^N\sigma_{ij}\cH^{d-1}(\Sigma_{ij}) \qquad & \text{if } \overrightarrow{\rho}\in\cF(A,\overrightarrow{m}), \\ 
+\infty & \text{otherwise.}
\end{cases}
\end{equation}

The relative adhesion coefficients $\sigma_{ij}$ play a key role in the properties of $\cE^0$. Their definition involves a distance function $d_A:\R^N\times \R^N\to [0,\infty)$ as follows:
\begin{equation}\label{eq:metric-coefficient}
    \sigma_{ij}:=d_A(\theta_i\overrightarrow{e_i},\theta_j\overrightarrow{e_j}), \qquad i,j = 0,\cdots, N
\end{equation}
where $\overrightarrow{e_i}$ is the usual basis vector with $1$ in the $i$-th coordinate and
with the convention that $\theta_0\overrightarrow{e_0}:=\overrightarrow{0}$.
The distance $d_A$ is defined by
\begin{align}\label{eq:distance-definition}
d_A(\overrightarrow{\zeta_0},\overrightarrow{\zeta_1}):=\inf\left\{\frac{1}{\sqrt{2}}\int_0^1\sqrt{H(\gamma(t))}|\gamma'(t)|_A\,dt\,\Big|\,\text{$\gamma:[0,1]\to\R^N$ is piecewise $C^1$} ,\,\gamma(0)=\overrightarrow{\zeta_0},\,\gamma(1)=\overrightarrow{\zeta_1}\right\},
\end{align}
where we used the notation $|\overrightarrow{\phi}|^2_A = \overrightarrow{\phi}^TA\overrightarrow{\phi}$
and where $H:\R^N\to\R$ is the function
\begin{align*}
H(\overrightarrow{\phi}):=\frac{\sigma}{2}|\overrightarrow{\phi}|_A^2-f^*(A\overrightarrow{\phi}-\overrightarrow{a}).
\end{align*}   
with 
\begin{equation}\label{eq:f^*}
f^*(v_1, \dots, v_N) = \left( \frac{m-1}{m} \max \{ v_1, \dots, v_N, 0 \} \right)^{\frac{m}{m-1}}
\end{equation}
(which is the Legendre transform of the function $f$ defined by \eqref{eq:f}).

\medskip

We can now state our first main result:
\begin{theorem}\label{thm:gamma-convergence}
Let $m>2$ and assume \ref{assumption:A1}, \ref{assumption:A2}, and \eqref{assumption:vacuum}. Then, with respect to the strong $L^1(\Omega)^N$-topology, the energy $\cE^{\varepsilon}$ $\Gamma$-converges as $\varepsilon\to0$ to $\cE^0$ defined by \eqref{eq:limit-energy-perimeter}-\eqref{eq:metric-coefficient}.
More precisely, we have:
\begin{enumerate}[label=(\roman*)]
\item The liminf property: For any sequence $\overrightarrow{\rho}^{\varepsilon}\in L^1(\Omega)^N$ converging to $\overrightarrow{\rho}^0$ strongly in $L^1(\Omega)^N$, it holds
\begin{align*}
\liminf_{\varepsilon\to0}\cE^{\varepsilon}(\overrightarrow{\rho}^{\varepsilon})\geq\cE^0(\overrightarrow{\rho}^0).
\end{align*}

\item The limsup property: For any $\overrightarrow{\rho}^0\in L^1(\Omega)^N$, there exists a sequence $\overrightarrow{\rho}^{\varepsilon}\in L^1(\Omega)^N$ converging to $\overrightarrow{\rho}^0$ strongly in $L^1(\Omega)^N$ such that
\begin{align*}
\limsup_{\varepsilon\to0}\cE^{\varepsilon}(\overrightarrow{\rho}^{\varepsilon})\leq\cE^0(\overrightarrow{\rho}^0).
\end{align*}
\end{enumerate}
\end{theorem}

This theorem rigorously establishes the connection between the classical nonlocal interaction energy \eqref{eq:energy-J} and the interfacial energy underlying the Differential Adhesion Hypothesis \eqref{eq:limit-energy-perimeter}. 
Energies of the form \eqref{eq:limit-energy-perimeter} commonly appear in the study of fluid-fluid interfaces and  give rise to constant mean curvature surfaces, together with prescribed contact angle conditions determined by the surface tension coefficients
$(\sigma_{ij})_{0\leq i,j\leq N}$.
An important contribution of this work is thus the derivation of the formula \eqref{eq:metric-coefficient} which relates the properties of the limiting energy \eqref{eq:limit-energy-perimeter} to the relative strength of the interaction between the different species (and the pressure function $f$).
This formula is not explicit but the fact that it involves the length of the shortest path in the phase space implies the following triangle inequality:
\begin{align}\label{eq:triangle-ineq}
\sigma_{ij} \leq \sigma_{ik} + \sigma_{kj} \qquad\text{for any }i,j,k=0,\cdots,N,
\end{align}
which in turns guarantees the lower-semicontinuity of $\cE^0$  \cite{M97}.

An  important question is to determine when these triangle inequalities are strict and when equality holds. Indeed, equality in \eqref{eq:triangle-ineq} corresponds to particular configurations of the energy minimizers, which can be directly related to the different regimes identified by Steinberg \cite{S62a,S62b,S62c,S62d} in the context of the Differential Adhesion Hypothesis for interacting cell populations: complete sorting, partial engulfment and full engulfment (see Figure~\ref{fig:steinberg}). 
We refer to Section \ref{subsec:steinberg} for a detailed discussion of these regimes.

In the case of two species, we prove the following theorem, which provides a partial classification of the corresponding regimes.

\medskip

\begin{theorem}\label{thm:DAH}
Assume that $N=2$ (only two types of cells are present) and that the interaction coefficient matrix  $A=\begin{pmatrix}
\alpha_{11} & \alpha_{12}\\
\alpha_{12} & \alpha_{22}
\end{pmatrix}$ satisfies Assumptions \ref{assumption:A1}, \ref{assumption:A2}, and \eqref{assumption:vacuum}. Let $(\sigma_{ij})_{i,j=0,1,2}$ be the coefficients appearing in the limiting energy $\cE^0$, given by \eqref{eq:metric-coefficient}.

Then, the following hold:
\begin{enumerate}[label=(\roman*)]
\item (Complete sorting) When $\alpha_{12}=0$, it holds that $$\sigma_{12} = \sigma_{01} + \sigma_{02}.$$
\item[(ii)] (Partial engulfment) There exists a number $\gamma_0 = \gamma_0(\alpha_{11},\alpha_{22},m,\sigma)>0$ such that whenever $\alpha_{12}\in(0,\gamma_0)$, we have the strict inequalities
$$
\sigma_{ij} < \sigma_{ik} + \sigma_{kj} \qquad\text{for every permutation $(i,j,k)$ of }(0,1,2).
$$
\item[(iii)] (Full engulfment) If $\alpha_{11} > \alpha_{22}>0$ and $\alpha_{12} \geq \alpha_{22}$, then it holds
$$
\sigma_{01} = \sigma_{02} + \sigma_{21}.
$$
\end{enumerate}
\end{theorem}
We note that in the case $N=2$, Assumptions \ref{assumption:A1} is equivalent to $\alpha_{11}, \alpha_{22}\geq 0$, and $\alpha_{12}^2\leq  \alpha_{11} \alpha_{22}$, so (iii) identifies a non-trivial range of values of $\alpha_{12}$ for which full engulfment is predicted.

\medskip

This theorem and the numerical simulations presented in Appendix \ref{subsec:numerics} suggest that given self-attraction strengths $\alpha_{11}$ and $\alpha_{22}$, there exists a critical value $\gamma_c$ of the cross-interaction strength such that
partial engulfment occurs when $\alpha_{12} <\gamma_c$ and full engulfment occurs when $\alpha_{12}\geq \gamma_c$.
This critical value must satisfy $\gamma_c \in [\gamma_0,\min\{\alpha_{11},\alpha_{22}\}]$, and numerical simulations suggest that $\gamma_c$ depends on $m$ and satisfies $\gamma_c <\min\{\alpha_{11},\alpha_{22}\}$.

\medskip

For the final result of this paper, we go back to the evolution equations \eqref{eq:epsilon-problem-1}-\eqref{eq:epsilon-problem-2}.
The existence of solutions to this system is a delicate matter. The system is a
Wasserstein gradient flow, so approximate solutions can be constructed by a standard
JKO-type time discretization. The difficulty lies in passing to the limit in the terms
$\rho_i\nabla f_m'(\rho)$: the scheme provides no regularity, and thus no strong
convergence, for the individual densities $\rho_i$, only for the total density $\rho$.

To the best of our knowledge, weak solutions of such cross-diffusion systems have been
obtained in the presence of a zeroth-order growth term but without advection
\cite{david,Price-Xu}, or with a single advection field shared by all species
\cite{Jacobs}. When the advection fields are independent, existence is known for
$m\leq1$ \cite{MeszarosParker25,MeszarosParker26} and, in dimension one, for every
$m\geq1$ \cite{Skr}.

Interestingly, we do not need to address this issue here because the problematic terms enter our analysis
only through their sum, and
$\sum_{i=1}^N \rho_i\nabla f_m'(\rho)=\rho\nabla f_m'(\rho)$ is meaningful under the
compactness we do have. 
In fact, we establish the convergence  as $\eps\to 0$ for
\emph{any} family $(\rho_1^\varepsilon,\dots,\rho_N^\varepsilon)$ satisfying the energy
inequality together with the individual continuity equations
\begin{equation}\label{eq:jkosol1}
 \varepsilon \partial_t \rho_i^{\varepsilon} - \operatorname{div} j_i^\varepsilon = 0,
 \qquad i=1,\dots,N,
\end{equation}
and the total flux condition
\begin{equation}\label{eq:jkosol2}
j_1^\varepsilon + \cdots + j_N^\varepsilon
 = \rho^{\varepsilon}\nabla f_m'(\rho^{\varepsilon})
 + \sum_{i,j=1}^N \alpha_{ij}\,\rho_i^{\varepsilon}\nabla \phi^{\varepsilon}_j,
 \qquad \rho^\varepsilon = \rho_1^{\varepsilon}+\dots +\rho_N^\varepsilon .
\end{equation}
We check in Appendix
\ref{sec:appendix-A} that the limits of the JKO scheme satisfy
\eqref{eq:jkosol1}--\eqref{eq:jkosol2}.

It might seem surprising that we do not need to characterize the individual fluxes to identify the limiting behavior, but the reason is that the species do not overlap in the limit $\eps\to0$. This phase separation regime is precisely the regime in which the difficulty mentioned above disappears.

Because the limiting behavior as $\eps\to 0$ will be characterized with the help of the energy functional $\cE^\eps$, we assume that the initial condition satisfies
\begin{equation}\label{eq:initenergy}
\sup_{\varepsilon\in(0,1)}\cE^{\varepsilon}(\overrightarrow{\rho}^{\varepsilon}_{in})\leq M
\end{equation}
for some constant $M>0$. 
This implies in particular that there exists a subsequence along which 
$\overrightarrow{\rho}^{\varepsilon}_{in}\to\overrightarrow{\rho}_{in}$ with $\cE^0(\overrightarrow{\rho}_{in})\leq M$. We thus have (along that subsequence):
\begin{equation}\label{eq:initenergylimit}
\lim_{\eps\to0} \overrightarrow{\rho}^{\varepsilon}_{in}= \overrightarrow{\rho}_{in}=\sum_{i=1}^N\theta_i\chi_{E_{i,in}}\overrightarrow{e_i}
\end{equation}
for some partition $\{E_{i,in}\}_{i=0}^N$ of $\Omega$ into $BV(\Omega)$-sets.

\begin{theorem}\label{thm:phase-separation-convergence}
Let $m>2$ and assume \ref{assumption:A1}, \ref{assumption:A2}, and \eqref{assumption:vacuum}.
Let $\overrightarrow{\rho}^{\varepsilon}=(\rho^{\varepsilon}_1,\cdots,\rho^{\varepsilon}_N)^T$ be a solution to \eqref{eq:epsilon-problem-1}-\eqref{eq:epsilon-problem-2} with initial condition satisfying \eqref{eq:initenergy} and \eqref{eq:initenergylimit}.
Then, the following hold:

\begin{enumerate}[label=(\roman*)]
\item There exists a subsequence (still denoted by $\varepsilon\to0$) such that
\begin{align*}
\overrightarrow{\rho}^{\varepsilon}\to\overrightarrow{\rho}\left(=:(\rho_1,\cdots,\rho_N)\right)\quad\text{strongly in }L^{\infty}_{loc}(0,\infty;L^1(\Omega)^N).
\end{align*}
Furthermore, we have $\cE^0(\overrightarrow{\rho}(t))\leq M$ for every $t\geq0$ so that we can write $$\overrightarrow{\rho}(t)=\sum_{i=1}^N\theta_i\chi_{E_i(t)}\overrightarrow{e_i}\in\cF(A,\overrightarrow{m})$$ for some partition $\{E_i(t)\}_{i=0}^N$ of $\Omega$ into $BV(\Omega)$-sets for every $t\geq0$.

\item For each $i=1,\cdots,N$, there is a velocity field $v_i\in L^2(\Omega\times(0,\infty),\chi_{E_i(t)}(x)\,dxdt;\R^d)$ such that the function $\chi_{E_i(t)}$ solves the continuity equation
\begin{align}\label{eq:continuity}
\begin{cases}
\partial_t \chi_{E_i(t)} + \operatorname{div} (\chi_{E_i(t)} v_i) = 0 & \textit{in } \Omega\times(0,\infty), \\
\chi_{E_i} v_i \cdot n = 0 & \textit{on } \partial \Omega\times(0,\infty), \\
\chi_{E_i(0)}=\chi_{E_{i,in}} & \textit{in } \Omega.
\end{cases}
\end{align}
Moreover, the energy dissipation holds:
\begin{align*}
\cE^{0}(\overrightarrow{\rho}(t))+\sum_{i=1}^N\frac{1}{2}\int_0^t\int_{\Omega}\rho_i|v_i|^2\,dxdt\leq\cE^0(\overrightarrow{\rho}_{in})\qquad\text{for any }t>0.
\end{align*}

\item The pressure
\begin{align*}
p^{\varepsilon}:=\frac{1}{\varepsilon}\left(\rho^{\varepsilon}f_m'(\rho^{\varepsilon})+\overrightarrow{a}\cdot\overrightarrow{\rho}^{\varepsilon}-(\overrightarrow{\rho}^{\varepsilon})^TA\overrightarrow{\phi}^{\varepsilon}\right)
\end{align*}
converges to a function $p(x,t)$ weakly* in $L^2_{loc}(0,\infty;(C^s(\Omega))^*)$ for any $s\in(0,1)$.

\item If the following energy convergence is satisfied:
\begin{align}\label{assumption:energy-convergence}
\lim_{\varepsilon\to0}\int_0^T\cE^{\varepsilon}(\overrightarrow{\rho}^{\varepsilon}(t))\,dt = \int_0^T\cE^0(\overrightarrow{\rho}(t))\,dt, 
\end{align}
then the evolution of the system is described by the multiphase Hele--Shaw free boundary problem with surface tension in the following sense:
The total flux $j := \rho_1v_1+\cdots+\rho_Nv_N $ satisfies
\begin{align}\label{eq:mullins-sekerka}
\int_{0}^{T} \int_{\Omega} j \cdot \xi - p \operatorname{div}(\xi) \, dx \, dt = - \int_{0}^{T} \sum_{i, j=0}^{N} \sigma_{ij} \int_{\Omega} (\operatorname{div}(\xi) - \nu_{ij} \otimes \nu_{ij} : D\xi) \, d\mu_{ij} \, dt,
\end{align}
for every test vector field $\xi \in C_c^\infty (\overline{\Omega} \times [0, T); \mathbb{R}^d)$ satisfying $\xi \cdot n \equiv 0 $ on $  \partial \Omega$
where
\begin{align*}
\left\{
\begin{aligned}
\sigma_{ij} &\text{ is the surface tension coefficient given by \eqref{eq:metric-coefficient}}, \\
\mu_{ij} &\text{ denotes the surface measure } d\cH^{d-1}\lfloor_{\partial^* E_i \cap \partial^* E_j \cap \Omega}, \\
\nu_{ij} & \text{ denotes the measure theoretic unit normal vector of } \partial^* E_i \cap \partial^* E_j \cap \Omega \text{ pointing outward of $E_i$.}\\
\end{aligned}
\right.
\end{align*}
\end{enumerate}
\end{theorem}

We point out that the identification of the free boundary condition $(iv)$ requires the assumption of convergence of the energy \eqref{assumption:energy-convergence}. This is a classical way to address the lack of control on the free boundaries as $\eps \to 0$, see for instance
 \cite{EO15,LO16,JKM16,LS18}.

Assuming that we have a smooth solution of \eqref{eq:continuity}-\eqref{eq:mullins-sekerka}, we can describe more explicitly the evolution of the interfaces $\pa E_i$:
First, the continuity equation leads to the conditions
\begin{align}\label{eq:evolution-law-from-weak-continuity}
\left\{
\begin{aligned}
\div( v_i) &= 0 \text{ in } E_i(t),\\
v_i\cdot n&= 0 \text{ on } E_i(t)\cap\partial\Omega
\end{aligned}
\right.
\end{align}
together with the condition
$$V_{ij} = v_i\cdot\nu_{ij} = v_j\cdot\nu_{ij},$$
where $V_{ij}$ is the normal speed of $\Sigma_{ij}$ in the direction of $\nu_{ij}$.

Next, we turn to Equation \eqref{eq:mullins-sekerka}: We first note that
\begin{align*}
\int_{\Omega}p\,\div(\xi)\,dx=\sum_{i=0}^N\int_{E_i}p\,\div(\xi)\,dx
&=\frac12\sum_{i,j=0}^N\int_{\Sigma_{ij}}[[p]]_{{ij}}\xi\cdot\nu_{ij}\,d\cH^{d-1} - \int_{\Omega}\nabla p\cdot\xi\,dx,
\end{align*}
where 
$$[[p]]_{{ij}}:=\lim_{h\to0^+}(p(x-h\nu_{ij})-p(x+h\nu_{ij}))$$
 denotes the jump of $p$ across $\Sigma_{ij}$.
Equation \eqref{eq:mullins-sekerka} then gives
\begin{align}\label{eq:mullins-sekerka-LHS}
& \int_{0}^{T} \int_{\Omega} (j+\nabla p) \cdot \xi \, dx \, dt - \frac12\sum_{i,j=0}^N\int_{0}^{T} \int_{\Sigma_{ij}} [[p]]_{{ij}}\xi\cdot\nu_{ij}\,d\cH^{d-1} \, dt\nonumber \\
& \qquad \qquad\qquad  =- \int_{0}^{T} \sum_{i, j=0}^{N} \sigma_{ij} \int_{\Omega} (\operatorname{div}(\xi) - \nu_{ij} \otimes \nu_{ij} : D\xi) \, d\mu_{ij} \, dt
\end{align}
This implies that $j=\theta_i v_i =- \na p$ in $E_i(t)$ and 
using the classical formula
\begin{align*}
\int_{\Sigma_{ij}}\div(\xi)-\nu_{ij}\otimes\nu_{ij}:D\xi\,d\cH^{d-1}=\int_{\Sigma_{ij}}\kappa_{ij}(\xi\cdot\nu_{ij})\,d\cH^{d-1}+\int_{\partial\Sigma_{ij}}\xi\cdot\widetilde{n}_{ij}\,d\cH^{d-1}
\end{align*}
where $\kappa_{ij}:=\div_{\Gamma_{ij}}(\nu_{ij})$ is the mean curvature and $\widetilde{n}_{ij}$ is the unit conormal vector on $\pa\Gamma_{ij}$,
we obtain the following multiphase Hele-Shaw free boundary problem with surface tension:
\begin{align}\label{eq:evolution-law-from-weak-MS}
\left\{
\begin{aligned}
\Delta p &= 0 \qquad\qquad\,\,\,\,\,\,\,\,\,\text{ in } E_i(t)\quad\text{and}\quad\frac{\partial p}{\partial n} = 0\qquad\text{ on } \partial \Omega\cap E_i(t), \\
[[p]]_{{ij}} &= 2\sigma_{ij}\kappa_{ij}\,\,\,\,\,\,\qquad\text{ on } \Sigma_{ij},\\
V_{ij} = -\theta_i^{-1}\left[\frac{\partial p}{\partial \nu_{ij}}\right]^{-} &= -\theta_j^{-1}\left[\frac{\partial p}{\partial \nu_{ij}}\right]^{+}\text{on }\Sigma_{ij}, \\
\sigma_{ij}\widetilde{n}_{ij} + \sigma_{jk}\widetilde{n}_{jk} + \sigma_{ki}\widetilde{n}_{ki}&= 0 \qquad\qquad\,\,\,\,\,\,\,\,\,\,\text{ on } \Sigma_{ijk}, \\
\widetilde{n}_{ij} &= n \qquad\qquad\,\,\,\,\,\,\,\,\,\,\text{ in } \partial \Sigma_{ij}\cap\pa\Omega,
\end{aligned}
\right.
\end{align}
where $\Sigma_{ijk}:=\pa E_i\cap \pa E_j\cap \pa E_k\cap\Omega$ denotes the triple junction 
and 
$ \left[\frac{\partial p}{\partial \nu_{ij}}\right]^{\pm}(x) : = \lim_{h\to 0^+} \na p(x\pm h) \cdot \nu_{ij}$ for $x\in \Sigma_{ij}$.

We point out that the fourth equation in \eqref{eq:evolution-law-from-weak-MS} is a well-known formulation of the Herring contact angle condition at triple junctions (see \eqref{eq:herring}), while the last equation of \eqref{eq:evolution-law-from-weak-MS} indicates that the interfaces meet $\partial\Omega$ orthogonally.

\section{Motivations  and related works}
\subsection{A system of interacting species}\label{subsec:PKS-system}
The system \eqref{eq:epsilon-problem-1}-\eqref{eq:epsilon-problem-2} is a  particular system for interacting populations. 
In a more general context, it might be natural to associate to each species a different volume density $\mu_i>0$ and a different mobility coefficient $d_i>0$ leading to
\begin{align*}
\begin{cases}
\partial_t \rho_i - d_i\operatorname{div}(\rho_i \nabla f'(\rho )) +  \operatorname{div} \left(\rho_i \sum_{j=1}^N \nabla\left( K _{ij}*\rho_{j} \right)\right) = 0,\\
\sigma_{ij}K_{ij} - \Delta K_{ij} = \chi_{ij} \delta_0,
\end{cases}
\end{align*}
where $\sigma_{ij},\chi_{ij}>0$ are possibly different adhesion coefficients and the pressure depends on the total volume 
$$\rho :=\mu_1\rho_1+\cdots+\mu_N\rho_N.$$
The change of variables, $\overline{\rho}_i(x,t) = \mu_i\rho_i(x,d_it)$ and $\overline{K}_{ij} = \frac{1}{d_i\mu_j\chi_{ij}}K_{ij}$ leads back to a system with $\mu_i=1$, $d_i=1$, and $\chi_{ij}=1$.
On the other hand, we stress the fact that our analysis requires that 
$${K}_{ij} = \alpha_{ij} G  \qquad \mbox{ where $G$ solves }  \sigma G -\Delta G = \delta,$$
meaning that all interactions, while of different strength, result from the same underlying mechanism.
This allows us to compare the self/cross and attraction/repulsion effects simply by comparing the interaction coefficients $\alpha_{ij}$. 

Our particular choice of kernel $G$ is classical in the context of chemotaxis. While our results are expected to hold for more general kernel $G$, this choice plays a fundamental role in this paper. 
Rigorously justifying the result (especially Theorem \ref{thm:phase-separation-convergence}) for more general $G$ is an open and challenging problem.

\medskip

The introduction of the small parameter $\eps$ in the system can then be justified as follows:
The formation of sharp interfaces in such system is observed  when a large population is observed at appropriate scale. 
More precisely, we consider a very large total number of cells, which we parametrize by
\[
\int_\Omega \rho_i\,dx=m_i \varepsilon^{-d}
\qquad\text{for some small parameter }\varepsilon\ll 1.
\]
The rescalings $x\mapsto \varepsilon x$ and $t\mapsto \varepsilon^{3} t$ lead to $\int_\Omega \rho_i^\eps\,dx=m_i$ and
\begin{align}\label{eq:rescaled-multicells}
\varepsilon \partial_t \rho_i^{\varepsilon} - \operatorname{div}(\rho_i^{\varepsilon}\nabla f_m'(\rho^{\varepsilon})) +  \operatorname{div} \left(\rho_i^{\varepsilon} \sum_{j=1}^N \nabla\left( K^{\varepsilon}_{ij}*\rho_{j}^{\varepsilon}\right)\right) = 0,
\end{align}
where  
$$
K_{ij}^\varepsilon(x)=\varepsilon^{-d}K_{ij}(\varepsilon^{-1}x).
$$

\medskip

With this rescaling and particular choice of $G$, we are able to rigorously identify the asymptotic dynamics and interpret our results in the context of the celebrated Differential Adhesion Hypothesis (DAH) which we discuss in further details below.

\subsection{The Differential Adhesion Hypothesis}\label{subsec:steinberg}
The Differential Adhesion Hypothesis (DAH), first formulated by Steinberg \cite{S62a,S62b,S62c,S62d}, suggests that equilibrium configurations are solely determined by the relative strength of adhesion between the different species of cells via the interfacial energy given by (for two species):
$$
\sigma_{01}\cH^{d-1}(\Sigma_{01}) + \sigma_{02}\cH^{d-1}(\Sigma_{02})+\sigma_{12}\cH^{d-1}(\Sigma_{12}).
$$
As a result, the DAH leads to different types of equilibria: mixing, complete sorting, partial and full engulfment (see Figure \ref{fig:steinberg}).

\medskip

A major contribution of this work is the rigorous derivation of this interfacial energy from the nonlocal interaction model \eqref{eq:epsilon-problem-1}-\eqref{eq:epsilon-problem-2}.
This derivation leads in particular to the formula \eqref{eq:metric-coefficient}-\eqref{eq:distance-definition} that relates the adhesion coefficient $\sigma_{ij}$ to the strength of the attraction between the $i$-th and $j$-th species $\alpha_{ij}$.
Although this formula is hardly explicit, the coefficients $\sigma_{ij}$  can be computed numerically and we establish (see Theorem \ref{thm:DAH}) rigorously the appearance of the different regimes predicted by the DAH depending on the cross-interaction coefficient.
In Appendix \ref{subsec:numerics} we briefly present some numerical results that further explore the transition between these regimes depending on the strength of the cross-interactions.

\begin{figure}[htbp]
	\begin{center}
            \includegraphics[height=8cm]{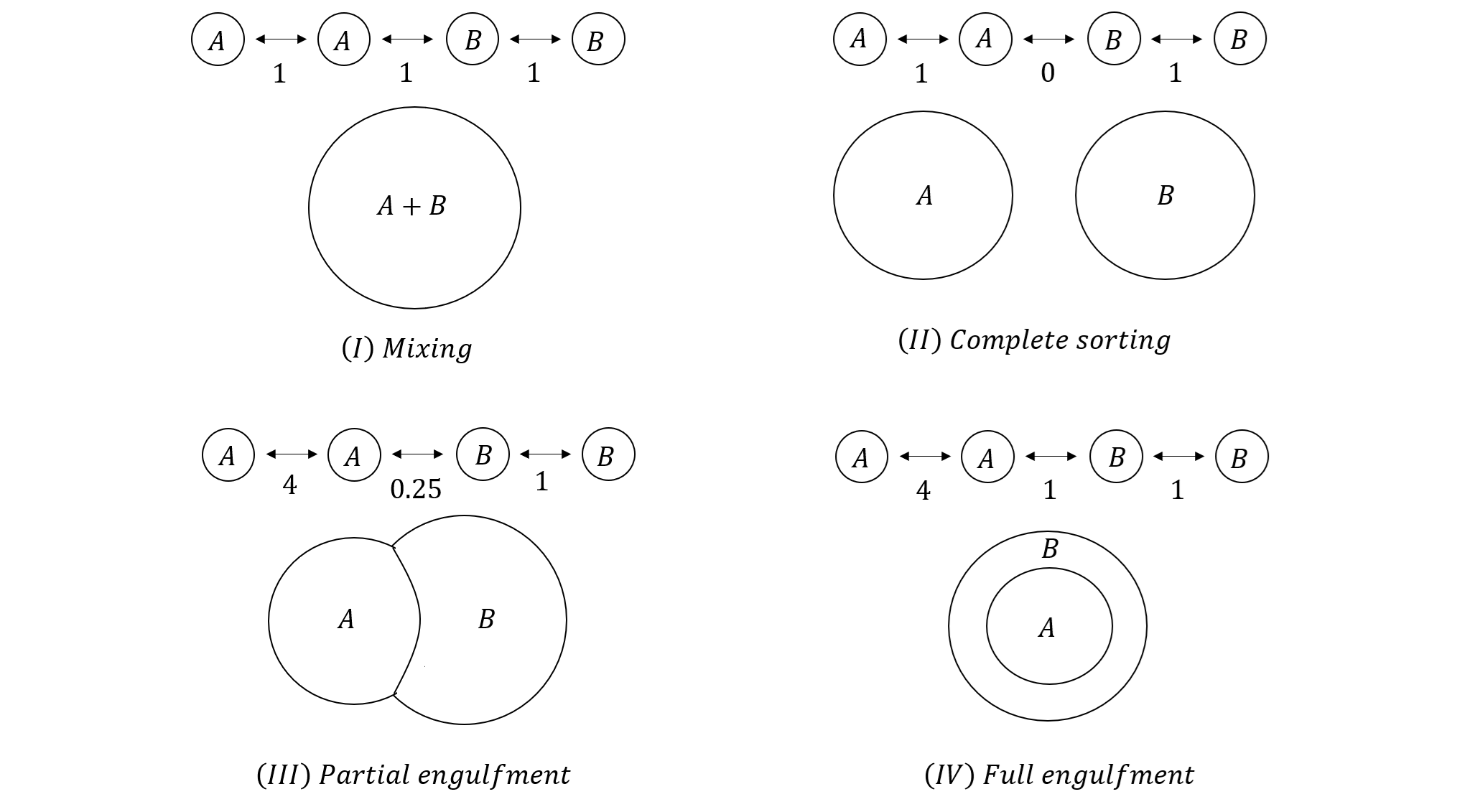}
            \captionsetup{width=0.8\textwidth} 
		\vskip 0pt
		\caption{The four configurations of cell-cell adhesion from the DAH.}
        \label{fig:steinberg}
	\end{center}
\end{figure}

\medskip

(I) Mixing corresponds to a situation where the two types of cell are not distinguishable energetically (same interaction coefficients) leading in particular to $\sigma_{12} = 0$. In our analysis, this case is ruled out by Assumption \ref{assumption:A2} (two such cell populations would be treated as one).

\medskip

(II) In the complete sorting case, different species avoid interfacial contact with each other (and thus evolve to balls with prescribed volume). 
This occurs when 
$$ \sigma_{12}=\sigma_{01}+\sigma_{02}$$
and the interfacial energy reduces to
\begin{align*}
2 ( \sigma_{01}\cH^{d-1}(\partial^*E_1 \cap \Omega) + \sigma_{02}\cH^{d-1}(\partial^*E_2 \cap \Omega)).
\end{align*}
We show that this energy appears when considering non-interacting species (i.e., with  cross-interaction coefficient $\alpha_{12}=0$), see Theorem \ref{thm:DAH}-(i).

\medskip

(III) Partial engulfment results from cross-interaction that is not exceedingly high, leading cell B (the less cohesive cell) to engulf cell A (the more cohesive one) partially.
It occurs when the adhesion coefficients $\sigma_{ij}$ are positive and satisfy strict triangle inequalities
$$
\sigma_{ij} < \sigma_{ik} + \sigma_{kj} \qquad\text{for any permutation }(i,j,k)=(0,1,2).
$$
Equilibrium states are  characterized by nontangential contact angles at any triple junctions. 
The angles are prescribed by the coefficients $\sigma_{ij}$'s appearing in the energy via the Herring contact angle condition
\begin{align}\label{eq:herring}
\frac{\sin(\theta_0)}{\sigma_{12}} = \frac{\sin(\theta_1)}{\sigma_{02}} = \frac{\sin(\theta_2)}{\sigma_{01}}.
\end{align}
Partial engulfment can be viewed as an intermediate state between complete sorting and full engulfment. 
Theorem \ref{thm:DAH}-(ii) proves that this regime appears in the limit $\eps\to0$ when 
the cross-interaction coefficient $\alpha_{12}$ is positive but small.

\medskip

(IV) Finally, full engulfment happens when cell B envelopes cell A completely. 
It occurs when we have equality in the triangle inequality:
$$
\sigma_{01}=\sigma_{02}+\sigma_{21}
$$
and we will establish this equality for large enough cross-interaction coefficient
$\alpha_{22}\leq \alpha_{12}\leq \sqrt{\alpha_{11}\alpha_{22}}$ - see Theorem \ref{thm:DAH}-(iii). This results coincides with the results of \cite{FBC22,MT15}.

We mention that the interfacial energy does not favor a particular placement of the engulfed cells A, although 
numerical simulation results \cite{FBC22,MT15} suggest that cells A tend to stay at the center as in Figure \ref{fig:steinberg}(IV). This symmetry is also expected to happen in our framework for a fixed (but small) $\varepsilon>0$ as the nonlocal effects should lead to a unique and radially symmetric minimizer.

\subsection{Brief review of the literature}\label{subsec:literature}

This paper is at the intersection of several important topics: the PKS model and free boundary problems, multiphase flows, and cell-cell adhesion. 
There is a rich body of literature devoted to single phase problems (the classical Cahn--Hilliard and Allen--Cahn equations,  Muskat, Hele--Shaw and Mullins--Sekerka flows with surface tension, etc.).
Since our focus is on the multi-phase framework, we will not review these work extensively here.

\medskip

In the case $N=1$, connections between  PKS type models and free boundary problems 
have been explored in various settings. It was first shown in \cite{KMW24,KMW23} that the PKS model with the incompressible constraint $\rho\leq1$ leads to the formation of sharp interfaces that evolve according to the (one-phase) Hele--Shaw free boundary problems with surface tension when $\varepsilon$ tends to zero. 
This result was extended in \cite{M24} to  nonlinear diffusion with $m>2$. 
We also refer to \cite{MR25,R25,JM25} for further work in this direction on the parabolic-parabolic and elliptic-parabolic PKS model (this last one leads to volume-preserving mean curvature flow).
These results rely on the $\Gamma$-convergence of the associated energy to the perimeter functional, in a way that generalizes the classical Modica--Mortola \cite{MM77} convergence result for the Allen--Cahn energy.

\medskip

In  the  multiphase setting $N>1$, \cite{B90} generalized the Modica--Mortola   result to multi-well potentials and 
proved the $\Gamma$-convergence to the multiphase perimeter functional.
In fact, \cite{B90} relies on a suitable generalization of the Modica--Mortola idea in vector form 
which we use in our analysis as well (see Theorem \ref{thm:gamma-convergence} and Section \ref{subsec:limsup-property}). 
The properties of the limiting multiphase perimeter functional are an important topic in the calculus of variation. The lower semicontinuity of such  functionals (including anisotropic framework) is investigated in \cite{AB90}.
A sufficient condition for the lower semicontinuity is also established in \cite{M97}, which shows that in the isotropic case (as is the case for \eqref{eq:limit-energy-perimeter}), the lower semicontinuity is equivalent to the triangle inequality \eqref{eq:triangle-ineq}.

\medskip

Using the  $\Gamma$-convergence result of \cite{B90}, \cite{LS18} shows that the vectorial Allen--Cahn equation approximates the multiphase mean curvature flow, whose convergence is proved under the same energy convergence assumption \eqref{assumption:energy-convergence}.
In the same spirit, in \cite{LO16}, the multiphase version of the 
thresholding schemes of  \cite{EO15} (see also \cite{MBO92,MBO94})
 is proved to converge to the multiphase mean curvature flow.

We note that in \cite{EO15,LO16,MBO92,MBO94} the heat content energy is used to approximate the perimeter functional. Using the same idea, \cite{JKM21} introduces a diffuse approximation of the Muskat problem with surface tension (modeling the flow of two immiscible fluids, like oil and water, through a porous medium) as well as its multiphase problem. This viewpoint of \cite{JKM21} allows to construct weak solutions and to prove the rigorous convergence to the Muskat problem under the energy assumption \eqref{assumption:energy-convergence}.

\medskip

Cell-cell adhesion is a classical topic and remains an active area of research. The Differential
Adhesion Hypothesis (DAH) was formulated by Malcolm Steinberg more than fifty years ago
\cite{S62a,S62b,S62c,S62d}, motivated by earlier adhesion-based experiments \cite{H43,TH55}. In this
series of papers, Steinberg identified four qualitative regimes of equilibrium configurations, which
we describe in Figure~\ref{fig:steinberg}. 
The analogy between immiscible fluids with distinct surface
tensions and cell populations with different adhesion strengths, together with the hypothesis that
cells reorganize to minimize a total adhesion energy in a manner similar to immiscible fluids, has
motivated extensive theoretical and experimental studies; see, for instance, \cite{FBC23,FS04} and
the references therein. 

Among the many mathematical models proposed for cell-cell adhesion (local or nonlocal, continuum or discrete), the ones most directly related to our work are the continuum formulations used in \cite{FBC23,CMMTT19}. In particular, \cite{FBC23} studies a local continuum model obtained via a Cahn--Hilliard-type approximation and provides both mathematical analysis and numerical simulations (see also \cite{FBC22,MT15} for additional simulation results). We refer to the bibliography of \cite{FBC23} for further references on adhesion-driven sorting and related modeling approaches. While our motivation is similar, we work in a different regime: we retain a genuinely nonlocal interaction structure (i.e., we do not pass through a Cahn--Hilliard approximation), and nonlinear diffusion effects with $m>2$ play a central role in our results.

\section{Preliminaries}

\subsection{Notations and definitions}\label{subsec:notations}
We collect here the main notations for the reader's convenience (although several of these notations will be recalled when they are first used): 

\begin{itemize}

\item[$\bullet$] The function $\rho_i(x)$ (or, $\rho_i(x,t)$) is a density distribution of the $i$-th cell. The state of the system is described by the vector function $\overrightarrow{\rho}(x):=(\rho_1(x),\cdots,\rho_N(x))^T$.

\item[$\bullet$] The function $f_m:\R\to\R\cup\{+\infty\}$ is defined by
\begin{equation*}
f_m(s):= 
\begin{cases} 
\displaystyle \frac{1}{m-1}s^{m} & \text{if } s\geq0, \\ 
+\infty & \text{otherwise},
\end{cases}
\qquad \mbox{ for some fixed } m>2
\end{equation*}
and we denote by $f$ the function
$$f(\overrightarrow{\rho})= 
\begin{cases}
    f_m(\rho_1+\cdots +\rho_N) & \mbox{ if } \overrightarrow{\rho}\in\R^N_{\geq0}\\
    +\infty & \mbox{  otherwise.}
\end{cases}
$$
  
\item[$\bullet$] For $1\leq i\leq N$, we fix a prescribed mass $m_i>0$ for the $i$-th cell population throughout the paper and denote $\overrightarrow{m}:=(m_1,\cdots,m_N)^T\in\R^N_{>0}$.

\item[$\bullet$] We denote by $\overrightarrow{e_i}=(0,\cdots,0,1,0,\cdots,0)^T\in\R^N$ the vector with $i$-th entry $1$ and we set $\overrightarrow{e_0}=\overrightarrow{0}$.

\item[$\bullet$] For $1\leq i,j\leq N$, the strength of the interactions between the $i$-th cell and $j$-th cell is denoted by $\alpha_{ij}\in\R$ and 
$A$ is the square matrix with coefficients $\alpha_{ij}$.

\item[$\bullet$] A matrix $Q$ is a square matrix such that $A=Q^TQ$ (which exists thanks to Assumption \ref{assumption:A1}).
Several choices are possible such as the symmetric square root $A^{1/2}$ or the upper/lower Cholesky matrices. Any of these matrices can be used.

\item[$\bullet$] 
Given density $\overrightarrow{\rho}^{\varepsilon} $, we denote by $\overrightarrow{\phi}^{\varepsilon}$ the solution of \eqref{eq:phi}. We recall that
 $\overrightarrow{\phi}^{\varepsilon}:=(\phi_1^{\varepsilon},\cdots,\phi_N^{\varepsilon})^T$ where the $\phi_i^\eps$ are the concentration of the chemoattractant associated to the cell density $\rho_i$.
We also denote 
$$\overrightarrow{\psi}^{\varepsilon}:=Q\overrightarrow{\phi}^{\varepsilon}$$
We will use the notation $\overrightarrow{\psi}:=Q\overrightarrow{\phi}$ for general $\overrightarrow{\phi}\in\R^N$ as well.

\item[$\bullet$] For $1\leq i\leq N$, we set $\theta_i:=\left(\frac{1}{2\sigma}\alpha_{ii}\right)^{\frac{1}{m-2}}$. The vectors $\theta_i \overrightarrow{e_i}$ and  $\overrightarrow{\alpha_i}:=\frac{1}{\sigma}\theta_iQ\overrightarrow{e_i}$ play a crucial role in the analysis. We will also use the notation $\overrightarrow{e_0}  = \overrightarrow{\alpha_0}=0$.

\item
$a_i:=\frac{m-2}{m-1}\theta_i^{m-1}$ and $\overrightarrow{a}:=(a_1,\cdots,a_N)^T\in\R^N$.

\item[$\bullet$] For two regular nonnegative Borel measures $\mu$ and $\nu$ on $\Omega$, the supremum $\mu \vee \nu$ is defined by
\begin{align*}
(\mu \vee \nu)(U) = \sup \{ \mu(U') + \nu(U'') : \  U' \cap U'' = \varnothing,\, U' \cup U'' \subset U,\, U' \text{ and } U'' \text{ are open sets in } \Omega \}
\end{align*}
for arbitrary open sets $U$ in $\Omega$. The supremum $\mu \vee \nu$ is the smallest regular measure among all regular measures that are greater than or equal to $\mu$ and $\nu$ on any Borel subsets of $\Omega$. The supremum of finitely many regular nonnegative Borel measures on $\Omega$ is similarly defined.

\item[$\bullet$] Unless indicated otherwise, all constants denoted by $C$  (possibly changing from line to line) depend on $A$, $m$, $N$, $\overrightarrow{m}\in\R^N_{>0}$,  the  coefficient $\sigma>0$ and the domain $\Omega$. Dependence on additional parameters/functions will be indicated.
\end{itemize}

\subsection{The multi-well structure of the energy}\label{subsec:energy}
The proof of Theorem \ref{thm:gamma-convergence} relies on 
the fact that the energy $\cE^{\varepsilon}$ can be written as the sum of a potential energy (with multiple wells) and a transition energy.
The computation below generalizes to multispecies the computation used in \cite{KMW24,M24}.

\begin{lemma}\label{lem:cE-alternative}
Define the function $W\,:\,\R^N\to\R\cup\{+\infty\}$ by
\begin{equation}\label{eq:Wdef}
W(\overrightarrow{\rho}):=  \displaystyle f(\overrightarrow{\rho})+\overrightarrow{a}\cdot\overrightarrow{\rho}-\frac{1}{2\sigma}\overrightarrow{\rho}^TA\overrightarrow{\rho}.
\end{equation}
Then, the energy \eqref{eq:definition-cE}
can be rewritten in the form
\begin{equation}\label{eq:energy-new}
\mathcal{E}^\varepsilon(\overrightarrow{\rho} ) = \frac{1}{\varepsilon} \int_{\Omega} W(\overrightarrow{\rho} )\,dx + \frac{1}{2\sigma\varepsilon}\int_{\Omega} (\overrightarrow{\rho}  - \sigma\overrightarrow{\phi}^\varepsilon)^T A (\overrightarrow{\rho}  - \sigma\overrightarrow{\phi}^\varepsilon)\,dx + \frac{\varepsilon}{2}\int_{\Omega} (\nabla \overrightarrow{\phi}^\varepsilon)^T A \nabla \overrightarrow{\phi}^\varepsilon \, dx
\end{equation}
for all $\overrightarrow{\rho}:\Omega \to \R^N_{\geq 0}$ such that $\int_{\Omega} \overrightarrow{\rho}\,dx=\overrightarrow{m}$ and with $ \overrightarrow{\phi}^\eps$ solution of \eqref{eq:phi}.
\end{lemma}
\begin{proof}

We recall that 
$$ \mathcal{E}^\varepsilon(\overrightarrow{\rho}) = \frac{1}{\varepsilon} \int_{\Omega} f(\overrightarrow{\rho}) + \overrightarrow{a} \cdot \overrightarrow{\rho} - \frac{1}{2} (\overrightarrow{\rho})^T A \overrightarrow{\phi}^\varepsilon \, dx $$
and we write
\begin{align*}
\int_\Omega - \frac{1}{2} (\overrightarrow{\rho})^T A (\sigma\overrightarrow{\phi}^\varepsilon)\, dx
& = \int_\Omega - \frac{1}{2} (\overrightarrow{\rho})^T A (\overrightarrow{\rho}) + \frac{1}{2} (\overrightarrow{\rho} -  \sigma\overrightarrow{\phi}^\varepsilon )^T A (\overrightarrow{\rho} -  \sigma\overrightarrow{\phi}^\varepsilon )\, dx \\
& \qquad + \int_\Omega \frac{1}{2} (  \sigma\overrightarrow{\phi}^\varepsilon )^T A (\overrightarrow{\rho} -  \sigma\overrightarrow{\phi}^\varepsilon )\, dx.
\end{align*}
Using Equation \eqref{eq:epsilon-problem-2}, the last term in this equality becomes
$$
\int_\Omega \frac{1}{2} ( \sigma \overrightarrow{\phi}^\varepsilon )^T A (\overrightarrow{\rho} -  \sigma\overrightarrow{\phi}^\varepsilon )\, \, dx
=- \eps^2 \int_\Omega \frac{1}{2} ( \sigma \overrightarrow{\phi}^\varepsilon )^T A (\Delta   \overrightarrow{\phi}^\varepsilon )\, \, dx
= \sigma \eps^2 \int_\Omega \frac{1}{2} ( \na  \overrightarrow{\phi}^\varepsilon )^T A (\na    \overrightarrow{\phi}^\varepsilon )\, \, dx.
$$
Putting these equalities together, we get:
\begin{align}\label{eq:cE-alternative}
\mathcal{E}^\varepsilon(\overrightarrow{\rho}) 
&= \frac{1}{\varepsilon} \int_{\Omega} f(\overrightarrow{\rho}) + \overrightarrow{a} \cdot \overrightarrow{\rho} - \frac{1}{2\sigma} (\overrightarrow{\rho})^T A \overrightarrow{\rho}\,dx \notag \\
& \qquad + \frac{1}{2\sigma\varepsilon}\int_{\Omega} (\overrightarrow{\rho} -  \sigma\overrightarrow{\phi}^\varepsilon )^T A (\overrightarrow{\rho} -  \sigma\overrightarrow{\phi}^\varepsilon )\,dx + \frac{\varepsilon}{2}\int_{\Omega} (\nabla \overrightarrow{\phi}^\varepsilon)^T A \nabla \overrightarrow{\phi}^\varepsilon \, dx
\end{align}
and the result follows.
\end{proof}

\medskip

The potential $W$ plays a key role in the asymptotic properties of $\cE^\eps$.
We recall that for the one-species model, the corresponding function, defined by 
\begin{equation}\label{eq:walpha}
w_{\alpha}(s):= 
\begin{cases} 
\displaystyle \frac{s^m}{m-1}+a_{\alpha}s-\frac12\alpha s^2 & \text{if } s\geq0, \\ 
+\infty & \text{otherwise},
\end{cases}
\qquad \mbox{ where } a_{\alpha}:=\frac{m-2}{m-1}\left(\frac12\alpha\right)^{\frac{m-1}{m-2}}.
\end{equation}
is a double-well potential: it is nonnegative on $\R$ and it vanishes at $0$ and $\theta_{\alpha}:=\left(\frac12\alpha\right)^{\frac{1}{m-2}}$.
This can be seen easily by writing 
$$w_{\alpha}(s) = s \left( \frac{s^{m-1}}{m-1} - \frac{\theta_\alpha ^{m-1}}{m-1} +   \theta_\alpha^{m-2}(\theta_\alpha - s)\right) $$
and using the convexity of $s\mapsto  \frac{s^{m-1}}{m-1} $ when $m>2$.
Similarly, using Assumptions \ref{assumption:A1} and \ref{assumption:A2} we can show that $W$ has a multi-well structure:
\begin{proposition}\label{prop:separation-potential-W}
Under Assumptions \ref{assumption:A1} and \ref{assumption:A2},
the function $\overrightarrow{\rho}\mapsto W(\overrightarrow{\rho})$ defined by \eqref{eq:Wdef} satisfies:
\begin{align}\label{eq:W-wells}
W(\overrightarrow{\rho})\geq0\quad \forall  \overrightarrow{\rho}\in \R^N\quad\text{and}\quad\{W(\overrightarrow{\rho})=0\}=\{0,\theta_1\overrightarrow{e_1},\cdots,\theta_N\overrightarrow{e_N}\}.
\end{align}
where $\theta_i=\left(\frac{1}{2\sigma}\alpha_{ii}\right)^{\frac{1}{m-2}}$.
\end{proposition}

\begin{proof}
Along the axis $\{ u \overrightarrow{e_i}\,|\, u\in\R\}$, $W$ coincides with the double-well potential \eqref{eq:walpha}, that is 
$$W(u \overrightarrow{e_i}) = w_{\alpha_{ii}} (u).$$
It follows from the comment above that $W$ is non-negative and only vanishes at the points $0,\theta_1\overrightarrow{e_1},\cdots,\theta_N\overrightarrow{e_N}$ along  the axes.

\medskip

To complete the proof, it only remains to show that $W$ 
does not have any zeroes that are  not on one of the axes. 
This will be done by proving that $W$ cannot have a local minimizer in $\R^N_{\geq0}\setminus\left(\cup_{i=1}^N\{\rho_i\overrightarrow{e_i}\,|\,\rho_i\in\R_{\geq0}\}\right)$.

Suppose, by contradiction, that $W$ has a local minimizer $\overrightarrow{\rho}^0\in\R^N_{\geq0}\setminus\left(\cup_{i=1}^N\{\rho_i\overrightarrow{e_i}\,|\,\rho_i\in\R_{\geq0}\}\right)$. Since $\overrightarrow{\rho}^0$ is not on an axis, at least two of its coordinates $\rho_i^0$ and $\rho_j^0$ must be positive.

\medskip

\textbf{Case 1.} Assume that $\alpha_{ii}\neq\alpha_{jj}$.

The function
\begin{align*}
W_0\,:\,(\rho_i, \rho_j) \longmapsto 
W \big(\rho_1^0, \ldots, \rho_{i-1}^0, \rho_i, \rho_{i+1}^0, \ldots, 
\rho_{j-1}^0, \rho_j, \rho_{j+1}^0, \ldots, \rho_N^0\big)
\end{align*}
has a minimum at $(\rho_i^0,\rho_j^0)\in\R^2_{>0}$, so its Hessian must have nonnegative determinant. This Hessian is given by:
\begin{align*}
\mathrm{Hessian}(\rho_i^0, \rho_j^0)
=\frac{1}{\sigma}
\begin{bmatrix}
S - \alpha_{ii} & S - \alpha_{ij} \\
S - \alpha_{ji} & S - \alpha_{jj}
\end{bmatrix}
\qquad\text{where }S:=\sigma f_m''\left(\sum_i\rho_i^0\right)>0.
\end{align*}
We have $\det(\mathrm{Hessian}(\rho_i^0, \rho_j^0))=\left(\frac{1}{\sigma}\right)^2\left(2\alpha_{ij}-(\alpha_{ii}+\alpha_{jj})\right)S$.
Assumption \ref{assumption:A1} implies $2\alpha_{ij}\leq 2\sqrt{\alpha_{ii}\alpha_{jj}} <\alpha_{ii}+\alpha_{jj}$
(this last inequality is only strict when $\alpha_{ii}\neq\alpha_{jj}$), which together with the fact that $S>0$ implies $\det(\mathrm{Hessian}(\rho_i^0, \rho_j^0))<0$, a contradiction.

\medskip

\textbf{Case 2.} Assume that $\alpha_{ii}=\alpha_{jj}$ (in that case, we have $\theta_i=\theta_j$ and $a_i=a_j$).

For $s\in(-\delta,\delta)$ with $\delta>0$ sufficiently small, we define
\begin{align*}
\overrightarrow{\rho}^s=\overrightarrow{\rho}^0+s(\overrightarrow{e_i}-\overrightarrow{e_j}).
\end{align*}
Then, for all $s\in(-\delta,\delta)$, we have $\rho_1^0 + \cdots + \rho_N^0 = \rho_1^s + \cdots + \rho_N^s$ and $\overrightarrow{a}\cdot\overrightarrow{\rho}^0=\overrightarrow{a}\cdot\overrightarrow{\rho}^s$ (since $a_i=a_j$). 
It follows that the first two terms of $W({\overrightarrow{\rho}}^s)= f(\rho^s)+\overrightarrow{a}\cdot{\overrightarrow{\rho}}^s-\frac{1}{2\sigma}{(\overrightarrow{\rho}^s)}^TA{\overrightarrow{\rho}}^s$ do not change with $s$. The assumption that $\overrightarrow{\rho}^0$ is a local minimizer of $W$ thus implies that the function
\begin{align*}
g(s) := \frac{1}{2}(\overrightarrow{\rho}^s)^TA\overrightarrow{\rho}^s\qquad\text{for }s\in(-\delta,\delta)
\end{align*}
attains a local maximum at $s=0$. Consequently, we must have 
\begin{align*}
g''(0) = (\overrightarrow{e_i}-\overrightarrow{e_j})^TA(\overrightarrow{e_i}-\overrightarrow{e_j})\leq0.
\end{align*}
Since $A$ is positive semidefinite, this implies that $A(\overrightarrow{e_i}-\overrightarrow{e_j}) = 0$, which contradicts the assumption \ref{assumption:A2}. 
\end{proof}

\subsection{Connection to the multiphase Allen--Cahn functional}\label{subsec:connection-Baldo}

The assumption that $A$  is symmetric and positive semidefinite (Assumption \ref{assumption:A1}) implies in particular that all the terms in \eqref{eq:energy-new} are non-negative.
This assumption also implies that $A$ can be factored as $A=Q^TQ$ for some square matrix $Q$ (such as the symmetric square root $A^{1/2}$ or the upper/lower Cholesky matrices - any such matrix $Q$ can be used and we fix one from now on). This allows us to write:
\begin{align}\label{eq:energy-psi}
\cE^{\varepsilon}(\overrightarrow{\rho})&=\frac{1}{\varepsilon}\int_{\Omega}W(\overrightarrow{\rho} )+\frac{1}{2\sigma}|Q(\sigma\overrightarrow{\phi^{\varepsilon}}-\overrightarrow{\rho} )|^2\,dx + \frac{\varepsilon}{2}\int_{\Omega}|Q\nabla\overrightarrow{\phi}^{\varepsilon}|^2\,dx \notag \\
&=\frac{1}{\varepsilon}\int_{\Omega}W(\overrightarrow{\rho})+\frac{1}{2\sigma}|\sigma\overrightarrow{\psi^{\varepsilon}}- Q\overrightarrow{\rho} |^2\,dx + \frac{\varepsilon}{2}\int_{\Omega}|\nabla\overrightarrow{\psi}^{\varepsilon}|^2\,dx, \qquad \overrightarrow{\psi}^{\varepsilon}=Q\overrightarrow{\phi}^{\varepsilon}.
\end{align}

Motivated by \cite{KMW24,M24}, we now introduce the function $h:\R^N\to\R\cup\{+\infty\}$ defined by
\begin{equation}\label{def:potential-h}
h\left(\overrightarrow{\psi}\right):=\inf_{\overrightarrow{\rho}\in\R^N}\left\{W(\overrightarrow{\rho})+\frac{1}{2\sigma}\left|\sigma\overrightarrow{\psi}- Q\overrightarrow{\rho}\right|^2\right\},
\end{equation}
which has the multi-well structure inherited from \eqref{eq:W-wells}:
It is clear from the definition that $h$ is non-negative and that $h\left(\overrightarrow{\psi}\right)=0$ if and only if  $\sigma\overrightarrow{\psi} = \theta_i Q \overrightarrow{e_i} $ (for which we can take $\overrightarrow{\rho}= \theta_i\overrightarrow{e_i} $ in the infimum).
We thus introduce the phase points
\begin{equation}\label{eq:phases}
\overrightarrow{\alpha_0}:=\overrightarrow{0},\qquad \overrightarrow{\alpha_i}:=\frac{1}{\sigma}\theta_iQ\overrightarrow{e_i}\quad \mbox{ for } i=1,\cdots,N  .
\end{equation}
Assumption \ref{assumption:A2} implies that these points are distinct:
\begin{proposition}\label{prop:mixing}
Under Assumption \ref{assumption:A1}, we have 
$$\overrightarrow{\alpha_i} =\overrightarrow{\alpha_j}\quad \Longleftrightarrow \quad A\overrightarrow{e_i} = A\overrightarrow{e_j}.$$
In particular, Assumption \ref{assumption:A2} implies that $\overrightarrow{\alpha_i} \neq \overrightarrow{\alpha_j}$ for all $i\neq j\in\{0,\cdots,N\}$.

The potential $\overrightarrow{\psi} \mapsto h(\overrightarrow{\psi})$ defined by \eqref{def:potential-h} satisfies
\begin{align}\label{eq:h-wells}
h\geq0\quad\text{on }\R^N\quad\text{and}\quad\{\overrightarrow{\psi}\in\R^N\,|\,h(\overrightarrow{\psi})=0\}=\{\overrightarrow{\alpha_0},\overrightarrow{\alpha_1},\cdots,\overrightarrow{\alpha_N}\}.
\end{align}
\end{proposition}
\begin{proof}
Since $Q\overrightarrow{x} = \overrightarrow{0} $ if and only if $A\overrightarrow{x} = \overrightarrow{0} $, we have that
$$\overrightarrow{\alpha_i} =\overrightarrow{\alpha_j} \Longleftrightarrow \theta_iA\overrightarrow{e_i} = \theta_jA\overrightarrow{e_j}.$$
Next, we note that the equality $\theta_iA\overrightarrow{e_i} = \theta_jA\overrightarrow{e_j}$ is equivalent to $\theta_i \alpha_{ki}=\theta_j\alpha_{kj}$ for all $k=1,\dots N$.
Taking $k=i$ and $k=j$, we get $\theta_i \alpha_{ii}=\theta_j\alpha_{ij}$ and $\theta_i \alpha_{ji}=\theta_j\alpha_{jj}$. Recalling that $A$ is symmetric, we deduce
$$ \theta_i^2   \alpha_{ii} = \theta_i \theta_j \alpha_{ij}= \theta_i \theta_j \alpha_{ji}=  \theta_j^2 \alpha_{jj}.$$
Since $\theta_k=\left(\frac{\chi}{2\sigma}\alpha_{kk}\right)^{\frac{1}{m-2}}$,  this implies that $\alpha_{ii}=\alpha_{jj}$ and thus that $\theta_i=\theta_j$ and $\alpha_{ki}=\alpha_{kj}$ for all $k=1,\dots N$. This proves that
$$\overrightarrow{\alpha_i} =\overrightarrow{\alpha_j}\quad \Longrightarrow \quad A\overrightarrow{e_i} = A\overrightarrow{e_j}.$$

To prove the other implication, we note similarly that the equality $A\overrightarrow{e_i} =  A\overrightarrow{e_j}$ is equivalent to $  \alpha_{ki}= \alpha_{kj}$ for all $k=1,\dots N$ and taking $k=i$ and $k=j$, we get $ \alpha_{ii}= \alpha_{ij}= \alpha_{ji}= \alpha_{jj}$ and so $\theta_i=\theta_j$.
it follows that $\theta_iA\overrightarrow{e_i} = \theta_jA\overrightarrow{e_j}$ which complete the proof.
\end{proof}

\medskip

The formula \eqref{eq:energy-psi} and the definition of $h(\overrightarrow{\psi})$ \eqref{def:potential-h}
yield:
\begin{align}\label{eq:AC-lower-bound}
\cE^{\varepsilon}(\overrightarrow{\rho}) \geq \cF^{\varepsilon}(\overrightarrow{\psi}^{\varepsilon}),  \qquad \overrightarrow{\psi}^{\varepsilon}=Q\overrightarrow{\phi}^{\varepsilon},
\end{align}
where 
\begin{align*}
\cF^{\varepsilon}(\overrightarrow{\psi}) := 
\begin{cases} 
\displaystyle \frac{1}{\varepsilon} \int_{\Omega} h(\overrightarrow{\psi}) \, dx + \frac{\varepsilon}{2} \int_{\Omega} |\nabla \overrightarrow{\psi}|^2 \, dx & \text{if } \overrightarrow{\psi} \in H^{1}(\Omega)^N, 
\\ 
+\infty & \text{otherwise.} 
\end{cases}
\end{align*}
This functional is a Allen--Cahn type functional with the multi-well structure defined by
\eqref{eq:h-wells}.
Furthermore, the mass constraint on $\overrightarrow{\rho}$ naturally implies the constraint $\int_{\Omega}\overrightarrow{\psi}\,dx=Q\overrightarrow{m}$.
The asymptotic behavior of $\cF^{\varepsilon}$ under such a mass constraint was established in \cite{B90}, thus generalizing to this multiphase framework the classical Modica--Mortola result \cite{MM77}:
\begin{theorem}[\!\!\!\protect{\cite{B90}}]\label{thm:Baldo}
The energy $\cF^{\varepsilon}$ $\Gamma$-converges as $\eps\to0$ (with respect to the $L^1(\Omega)^N$-topology) to  the functional
\begin{equation*} 
\cF^{0}(\overrightarrow{\psi}):= 
\begin{cases} 
\displaystyle \sum_{i,j=0}^N
\sigma_{ij}
\cH^{d-1}(\partial^*E_i \cap \partial^*E_j \cap \Omega) & \displaystyle\text{if } \overrightarrow{\psi}=\sum_{i=1}^N\overrightarrow{\alpha_i}\theta_i\chi_{E_i}\text{ with }\chi_{E_i}\in BV(\Omega,\{0,1\}), \\ 
+\infty & \text{otherwise}.
\end{cases}
\end{equation*}
The surface tension coefficients are given by
\begin{equation}\label{eq:metric-coefficient-before-change-of-variables}
    \sigma_{ij}:=d(\overrightarrow{\alpha_i},\overrightarrow{\alpha_j}), \qquad i,j = 0,\dots N
\end{equation}
and the function $d:\R^N\to [0,\infty]$ is the distance defined by:
\begin{align}\label{eq:distance-definition-before-change-of-variables}
d(\overrightarrow{\zeta_0},\overrightarrow{\zeta_1}):=\inf\left\{\frac{1}{\sqrt{2}}\int_0^1\sqrt{h(\gamma(t))}|\gamma'(t)|\,dt\,|\,\text{$\gamma:[0,1]\to\R^N$ is piecewise $C^1$} ,\,\gamma(0)=\overrightarrow{\zeta_0},\,\gamma(1)=\overrightarrow{\zeta_1}\right\}.
\end{align}
\end{theorem}
The limiting energy \eqref{eq:limit-energy-perimeter} of our Theorem \ref{thm:gamma-convergence} is then $\cE^0(\overrightarrow{\rho})=\cF^0(Q\overrightarrow{\rho})$, where we note that the condition on $Q\overrightarrow{\rho}$ are equivalent to 
$\overrightarrow{\rho}\in\cF(A,\overrightarrow{m})$. 
Using the alternate formula \eqref{eq:hdefl} for $h$ and a simple change of variable, one can show that 
\eqref{eq:metric-coefficient-before-change-of-variables} is equivalent to
\eqref{eq:metric-coefficient}, that is
$$
d(\overrightarrow{\alpha_i},\overrightarrow{\alpha_j})=
d_A(\theta_i\overrightarrow{e_i},\theta_j\overrightarrow{e_j})
$$ 
(with $d_A$ defined by \eqref{eq:distance-definition}).

\medskip

In view of \eqref{eq:AC-lower-bound}, the liminf property for $\cE^{\varepsilon}$ follows immediately from the corresponding property for $\cF^\eps$.
The challenge is thus to construct a recovery sequence for the limsup property. 
It was already observed in \cite{M24} that we cannot simply take the constant sequence given by the limiting characteristic function.
This simpler construction works in the limiting case $m=+\infty$ \cite{MW23,KMW24}, but fails when $m<\infty$ (see \cite[Remark 2.2]{M24}).
The recovery sequence will thus be constructed from the recovery sequence for $\cF^\eps$ whose existence is guaranteed by Theorem \ref{thm:Baldo}.
The main difficulty is to ensure that the constructed sequence satisfies the mass constraint $\int_{\Omega}\overrightarrow{\rho}^{\varepsilon}\,dx=\overrightarrow{m}$.

\section{The $\Gamma$-convergence of the energy}\label{sec:gamma-convergence}

We take $\sigma=1$ for simplicity henceforth in this paper (the general case follows similarly).

\subsection{Auxiliary results}\label{subsec:more-functionals}

Following \cite{M24}, we derive several formula that will be used in the proof of the $\Gamma$-convergence result.
First,   
we denote by  
$
f^*(\overrightarrow{v}):=\sup_{\overrightarrow{\rho}\in\R^N}\left\{\overrightarrow{\rho}\cdot\overrightarrow{v}-f(\overrightarrow{\rho})\right\}
$ the Legendre transform of $f$.
We note that $f$ can be written as a function of one variable as follows:
$f(\overrightarrow{\rho}) = g(\rho_1+\cdots \rho_N)$ with $g(s) = \frac{s^m}{m-1}$.  The definition of the Legendre transform then yields $f^*(\overrightarrow{v}) = g^*(\max \{ v_1, \dots, v_N\})$ and so (since $g^*(p) = \left( \frac{m-1}{m} \max\{p,0\}\right)^{\frac{m}{m-1}}$):
\begin{equation*}
f^*(v_1, \dots, v_N) = \left( \frac{m-1}{m} \max \{ v_1, \dots, v_N, 0 \} \right)^{\frac{m}{m-1}}.
\end{equation*}
At a point $\overrightarrow{v}$ where there is a unique index $i$ such that $v_i = \max \{ v_1, \dots, v_N, 0 \}$, we have ${ f^*}'(v_1, \dots, v_N) = \left( \frac{m-1}{m} \max \{ v_1, \dots, v_N, 0 \} \right)^{\frac{1}{m-1}} \overrightarrow{e_i}$. Otherwise, with 
$I(\overrightarrow{v}):=\{i=1,\cdots,N\,|\,v_i=\max\{v_1,\cdots,v_N,0\}\}$, we find
\begin{equation}\label{eq:subf}
\partial f^*(v_1, \dots, v_N) = \left\{\left( \frac{m-1}{m} \max \{ v_1, \dots, v_N, 0 \} \right)^{\frac{1}{m-1}} \sum_{i \in I(\overrightarrow{v})} \nu_i \overrightarrow{e_i}\,\Big|\,\nu_i\in [0,1]\text{ and } \sum_{i\in I(\overrightarrow{v})}\nu_i=1\right\}.
\end{equation}

\medskip

With these notations, we can prove the following result:

\begin{lemma}\label{lem:legendre}
The following statements hold:
\begin{enumerate}[label=(\roman*)]
\item The function $h$ defined by \eqref{def:potential-h} can also be written as:
\begin{align}\label{eq:hdefl}
h(\overrightarrow{\psi})=\frac12|\overrightarrow{\psi}|^2-f^*(Q^T\overrightarrow{\psi}-\overrightarrow{a}).
\end{align}

\item For a given $\overrightarrow{\psi}\in\R^N$, we have 
$$
h(\overrightarrow{\psi})=W(\overrightarrow{\rho})+\frac{1}{2}|\overrightarrow{\psi}-Q\overrightarrow{\rho}|^2 \iff Q^T\overrightarrow{\psi}-\overrightarrow{a}\in\partial f(\overrightarrow{\rho})  \iff 
\overrightarrow{\rho}\in \partial f^*(Q^T\overrightarrow{\psi}-\overrightarrow{a}).
$$
When $\overrightarrow{\psi}=Q\overrightarrow{\phi}\in\R^N$, these are also equivalent to
$A\overrightarrow{\phi}-\overrightarrow{a}\in\partial f(\overrightarrow{\rho})$
and 
$\overrightarrow{\rho}\in\partial f^*(A\overrightarrow{\phi}-\overrightarrow{a})$. 
Furthermore, we have
\begin{align*}
 \partial f^*(\theta_iA\overrightarrow{e_i}-\overrightarrow{a}) = \{\theta_i{\overrightarrow{e_i}}\}
\end{align*}
for all $i=0,\cdots,N$ (with $\theta_0\overrightarrow{e_0} := \overrightarrow{0}$).

\item After re-numbering the species so that $\alpha_{11}\geq\cdots\geq\alpha_{NN}$, we have: Given $i,j=1,\cdots,N$ with $j\geq i$, there holds $\alpha_{jj}\phi_j-a_j \geq \alpha_{ji}\phi_j-a_i\text{ for any }\phi_j\in[0,\theta_j].$ As a consequence, for any $j=1,\cdots,N$ and $\phi_j\in[0,\theta_j]$, there exists $\overrightarrow{\rho}\in\partial f^*(\phi_jA\overrightarrow{e_j}-\overrightarrow{a})$ such that 
$$
\overrightarrow{\rho}\cdot\overrightarrow{e_i}=0 \quad \mbox{ for all } i=1,\cdots,j-1
$$
(in other words, there is a vector in the subdifferential $\partial f^*(\phi_jA\overrightarrow{e_j}-\overrightarrow{a})$ whose
first $j-1$ components  are zero).
\end{enumerate}
\end{lemma}
\begin{proof} $\,$
\begin{enumerate}[label=(\roman*)]
\item We compute
\begin{align}
h(\overrightarrow{\psi}) & = \inf_{\rho\in\R^N}\left\{W(\overrightarrow{\rho})+\frac12|\overrightarrow{\psi}-Q\overrightarrow{\rho}|^2\right\} \notag\\
&= \inf_{\rho\in\R^N}\left\{f(\overrightarrow{\rho})+\overrightarrow{a}\cdot\overrightarrow{\rho}-\frac12|Q\overrightarrow{\rho}|^2+\frac12|\overrightarrow{\psi}|^2-\overrightarrow{\psi}^TQ\overrightarrow{\rho}+\frac12|Q\overrightarrow{\rho}|^2\right\}\notag\\
& = \inf_{\rho\in\R^N}\left\{\frac12|\overrightarrow{\psi}|^2 - \left((Q^T\overrightarrow{\psi}-\overrightarrow{a})\cdot\overrightarrow{\rho} - f(\overrightarrow{\rho}) \right)\right\} \notag\\
& = \frac12|\overrightarrow{\psi}|^2 - \sup_{\overrightarrow{\rho}\in\R^N}\left\{(Q^T\overrightarrow{\psi}-\overrightarrow{a})\cdot\overrightarrow{\rho} - f(\overrightarrow{\rho}) \right\} = \frac12|\overrightarrow{\psi}|^2 - f^*(Q^T\overrightarrow{\psi}-\overrightarrow{a}),\label{eq:h-f*}
\end{align}
which yields the first result.

\item 
Next, the computation above shows that the infimum in the definition of $h$ is attained at $\overrightarrow{\rho}\in\R^N$ if and only if the supremum in \eqref{eq:h-f*} is attained at the same $\overrightarrow{\rho}\in\R^N$.
Classically, this is equivalent to $Q^T\overrightarrow{\psi}-\overrightarrow{a}\in\partial f(\overrightarrow{\rho})$. The last equivalence is a standard property of the Legendre transform.

Finally, since we have $h(\overrightarrow{\alpha_i}) = 0 =  W(\theta_i \overrightarrow{e_i})+\frac{1}{2}|\overrightarrow{\alpha_i}-\theta_iQ\overrightarrow{e_i}|^2 $, the argument above implies that
$\theta_i{\overrightarrow{e_i}}\in\partial f^*(\theta_i A\overrightarrow{e_i}-\overrightarrow{a})$. 

Conversely, let $\overrightarrow{\rho} \in \partial f^*(\theta_i A\overrightarrow{e_i}-\overrightarrow{a})$. Then we must have $h(\overrightarrow{\alpha_i}) = 0= W(\overrightarrow{\rho})+\frac{1}{2}|\overrightarrow{\alpha_i}-Q\overrightarrow{\rho}|^2$.
But this means that $W(\overrightarrow{\rho}) = 0$ and $Q\overrightarrow{\rho} = \overrightarrow{\alpha_i}$. The first condition implies that 
$\overrightarrow{\rho}=\theta_i\overrightarrow{e_j}$ for some $j=0,\cdots,N$, and the second one then gives $ Q \theta_i\overrightarrow{e_j} = \overrightarrow{\alpha_i}$ which in view of Proposition \ref{prop:mixing} 
implies $\overrightarrow{\rho}=\theta_i\overrightarrow{e_i}$ (this is an important  consequence of our Assumption \ref{assumption:A2}).

\item Given two indices $j>i$, we first prove that
\begin{equation}\label{eq:alphaiji}
\alpha_{jj}\phi_j-a_j \geq \alpha_{ji}\phi_j-a_i\qquad\text{ for all }\phi_j\in[0,\theta_j].
\end{equation}
By interpolation, it is enough to show that \eqref{eq:alphaiji} holds when $\phi_j=0$ and $\phi_j=\theta_j$: when $\phi_j = 0$ the inequality follows from the ordering $\alpha_{jj}\leq\alpha_{ii}$ and the definition of $a_i$. When $\phi_j = \theta_j$ is it a consequence of (ii) (that is, $\theta_j{\overrightarrow{e_j}}\in\partial f^*(\theta_jA\overrightarrow{e_j}-\overrightarrow{a})$) and the form of the subgradient $\partial f^*$ \eqref{eq:subf}. 

Next, denoting $\overrightarrow{v}=\phi_jA\overrightarrow{e_j}-\overrightarrow{a}$, we have $v_i = \phi_j \alpha_{ij}-a_i$, so \eqref{eq:alphaiji} implies that 
$v_i\leq v_j$ for all $i<j$. We deduce that there exists $k\geq j$ such that $v_k=\max\{v_1,\dots, v_N,0\}$. The definition \eqref{eq:subf} implies that 
$$
\left( \frac{m-1}{m} v_k \right)^{\frac{1}{m-1}} \overrightarrow{e_k} \in \partial f^*(\phi_jA\overrightarrow{e_j}-\overrightarrow{a}).
$$
We can thus choose $\overrightarrow{\rho} =\left( \frac{m-1}{m} v_k \right)^{\frac{1}{m-1}} \overrightarrow{e_k}  $, which has the desired property.

\end{enumerate}
\end{proof}

We now introduce  the following relaxations of the energy functional:
\begin{align}\label{eq:cE-tilde}
\widetilde{\cE}^{\varepsilon}(\overrightarrow{\rho}) := 
\begin{cases} 
\displaystyle \frac{1}{\varepsilon} \int_{\Omega} W(\overrightarrow{\rho}) + \frac{1}{2}|Q\overrightarrow{\phi}^{\varepsilon} - Q\overrightarrow{\rho}|^2\,dx + \varepsilon \int_{\Omega} \frac{1}{2} |\nabla Q\overrightarrow{\phi}^{\varepsilon}|^2 \, dx & \text{if } \overrightarrow{\rho} \in L^{2}(\Omega)^N, 
\\ 
\infty & \text{otherwise,} 
\end{cases}
\end{align}
with $\overrightarrow{\phi}^\eps$ solution of \eqref{eq:phi} (
so $\widetilde{\cE}^{\varepsilon}(\overrightarrow{\rho})=\cE^{\varepsilon}(\overrightarrow{\rho})$ whenever  $\int_{\Omega}\overrightarrow{\rho}\,dx=\overrightarrow{m}$)
and
\begin{align*}
\mathcal{G}^{\varepsilon}(\overrightarrow{\rho}, \overrightarrow{\psi}) := \frac{1}{\varepsilon} \int_{\Omega} W(\overrightarrow{\rho}) + \frac{1}{2}|\overrightarrow{\psi} - Q\overrightarrow{\rho}|^2\,dx + \varepsilon \int_{\Omega} \frac{1}{2} |\nabla \overrightarrow{\psi}|^2 \, dx\quad\text{for }(\overrightarrow{\rho}, \overrightarrow{\psi})\in L^2(\Omega)^N\times H^1(\Omega)^N
\end{align*}
(so $\mathcal{G}^{\varepsilon}(\overrightarrow{\rho},Q \overrightarrow{\phi}^\eps ) =  \widetilde{\cE}^{\varepsilon}(\overrightarrow{\rho}) $ if $\overrightarrow{\phi}^\eps$ solution of \eqref{eq:phi}).

We also recall the definition of the multiphase Allen--Cahn functional introduced earlier:
\begin{align*}
\cF^{\varepsilon}(\overrightarrow{\psi}) := 
\begin{cases} 
\displaystyle \frac{1}{\varepsilon} \int_{\Omega} h(\overrightarrow{\psi}) \, dx + \varepsilon \int_{\Omega} \frac{1}{2} |\nabla \overrightarrow{\psi}|^2 \, dx & \text{if } \overrightarrow{\psi} \in H^{1}(\Omega)^N, 
\\ 
+ \infty & \text{otherwise.} 
\end{cases}
\end{align*}

We then have the following result:
\begin{proposition}\label{prop:minimization}
The following statements hold:
\begin{enumerate}[label=(\roman*)]
\item For all $\overrightarrow{\rho}\in L^{2}(\Omega)^N$, it holds
\begin{align*}
\min_{\overrightarrow{\psi} \in H^{1}(\Omega)^N} \mathcal{G}^{\varepsilon}(\overrightarrow{\rho}, \overrightarrow{\psi}) = \widetilde{\cE}^{\varepsilon}(\overrightarrow{\rho}).
\end{align*}
The minimum is attained when $\overrightarrow{\psi}^{\varepsilon}=Q\phi^\eps$ with $\phi^\eps$ solution of \eqref{eq:phi}.

\item For all $\overrightarrow{\psi} \in H^{1}(\Omega)^N$, it holds 
\begin{align*}
\min_{\overrightarrow{\rho} \in L^{2}(\Omega)^N} \mathcal{G}^{\varepsilon}(\overrightarrow{\rho}, \overrightarrow{\psi}) = \cF^{\varepsilon}(\overrightarrow{\psi}).
\end{align*}
The minimum is attained at $\overrightarrow{\rho}\in\partial f^*(Q^T\overrightarrow{\psi}-\overrightarrow{a})$.

\item There exists a constant $C$ such that: If $\overrightarrow{\rho}^1,\overrightarrow{\rho}^2\in L^{\infty}(\Omega)^N$ have disjoint supports with $$\mathrm{dist}(\mathrm{supp}(\overrightarrow{\rho}^1),\mathrm{supp}(\overrightarrow{\rho}^2))>\delta$$ for some $\delta>0$, then
\begin{align*}
\widetilde{\cE}^{\varepsilon}(\overrightarrow{\rho}^1+\overrightarrow{\rho}^2) \leq \widetilde{\cE}^{\varepsilon}(\overrightarrow{\rho}^1) + \widetilde{\cE}^{\varepsilon}(\overrightarrow{\rho}^2) + C\|\overrightarrow{\rho}^1\|_{L^{\infty}(\Omega)^N}\|\overrightarrow{\rho}^2\|_{L^{\infty}(\Omega)^N}e^{-\frac{\delta}{C\varepsilon}}
\end{align*}
for $\varepsilon\in(0,1)$ small enough.
\end{enumerate}
\end{proposition}
\begin{proof} $\,$
\begin{enumerate}[label=(\roman*)]
\item Write
\begin{align*}
\min_{\overrightarrow{\psi} \in H^{1}(\Omega)^N} \mathcal{G}^{\varepsilon}(\overrightarrow{\rho}, \overrightarrow{\psi}) = \frac{1}{\varepsilon}\int_{\Omega}W(\overrightarrow{\rho})\,dx + \min_{\overrightarrow{\psi} \in H^{1}(\Omega)^N}\left\{\frac{1}{\varepsilon}\int_{\Omega}\frac{1}{2}|\overrightarrow{\psi} - Q\overrightarrow{\rho}|^2\,dx + \varepsilon\int_{\Omega}\frac12|\nabla\overrightarrow{\psi}|^2\,dx\right\}.
\end{align*}
It is now classical that the minimization problem from the second term is attained for $\overrightarrow{\psi}$ (variational) solution of
$$\begin{cases}
    \overrightarrow{\psi} - \eps^2 \Delta \overrightarrow{\psi} = Q\overrightarrow{\rho} & \mbox{ in } \Omega\\
    \na \overrightarrow{\psi} \cdot n  = 0 & \mbox{ on } \pa\Omega.
\end{cases}$$
The result follows.

\item The conclusion follows from Lemma \ref{lem:legendre} and from the fact that the form $\overrightarrow{\rho}\in\partial f^*(Q^T\overrightarrow{\psi}-\overrightarrow{a})$ guarantees $\overrightarrow{\rho}\in L^{2}(\Omega;\R^N_{\geq0})$.

\item We note that \eqref{eq:cE-tilde} can be rewritten as
\begin{align*}
\widetilde{\cE}^{\varepsilon}(\overrightarrow{\rho})=\frac{1}{\varepsilon}\int_{\Omega}W(\overrightarrow{\rho}) + \overrightarrow{a}\cdot\overrightarrow{\rho}-\frac12\overrightarrow{\rho}^TA\overrightarrow{\phi}^{\varepsilon}\,dx
\end{align*}
(via the same computation as in the derivation of \eqref{eq:cE-alternative}). Since $\overrightarrow{\rho}^1$ and $\overrightarrow{\rho}^2$ have disjoint supports, we have
\begin{align*}
\widetilde{\cE}^{\varepsilon}(\overrightarrow{\rho}^1+\overrightarrow{\rho}^2) - \widetilde{\cE}^{\varepsilon}(\overrightarrow{\rho}^1) - \widetilde{\cE}^{\varepsilon}(\overrightarrow{\rho}^2) 
& = -\frac{1}{2\eps} \int_\Omega 
{\overrightarrow{\rho}^1}^TA\overrightarrow{\phi^2}^{\varepsilon}+{\overrightarrow{\rho}^2}^TA\overrightarrow{\phi^1}^{\varepsilon}\,dx\\
&\leq \frac{M}{\varepsilon}\max_{i,j=1,\cdots,N}\int_{\Omega}\int_{\Omega}G(x,y)\rho^1_j(x)\rho^2_j(y)\,dxdy
\end{align*}
where $M:=\max\{|\alpha_{ij}|\,;\,i,j=1,\cdots,N\}$ and $G$ is the green kernel of the equation \eqref{eq:epsilon-problem-2} (so that $ \phi_i^\eps = \int_\Omega G(x,y)\rho_i(y)\, dy$). The result then follows from the assumption $\mathrm{dist}(\mathrm{supp}(\overrightarrow{\rho}^1),\mathrm{supp}(\overrightarrow{\rho}^2))>\delta$ and the classical fact that
\begin{align*}
G(x,y)\leq Ce^{-\frac{|x-y|}{C\varepsilon}}
\end{align*}
for some $C>0$ large enough (depending only on $\Omega$).
\end{enumerate}
\end{proof}

\subsection{The liminf property}\label{subsec:liminf-property}
In this subsection, we prove Theorem \ref{thm:gamma-convergence}-(i), which follows from the inequality \eqref{eq:AC-lower-bound} and the $\Gamma$-convergence of the functional $\cF^\eps$ given by Theorem \ref{thm:Baldo}.
We recall the main steps below, as they will be useful later on, and refer to \cite{B90} for details about the proofs.

First, an important role is played by the functions $\varphi_i\,:\,\R^N\to\R$ defined for $i=0,\cdots,N$ by
\begin{align*}
\varphi_i:=\max\left\{2d(\cdot,\overrightarrow{\alpha_i}),M\right\}
\end{align*}
where $d$ is the distance function defined by \eqref{eq:distance-definition-before-change-of-variables}
and $M>0$ is a fixed constant larger than $2d(\overrightarrow{\alpha_i},\overrightarrow{\alpha_j})+1$ for any $i,j=0,\cdots,N$. 
The following result is the key to adapting the well-known Modica--Mortola trick to this multiphase setting:
\begin{proposition}[\!\!\!\protect{\cite[Proposition 2.1, (2.2)]{B90}}]\label{prop:MM-functional-Baldo}
For any $i=0,\cdots,N$, the function $\varphi_i$ is bounded and Lipschitz continuous on $\R^N$. For any $\overrightarrow{\psi}\in H^1(\Omega;\R^N)$, it holds that $\varphi_i\circ\overrightarrow{\psi}\in W^{1,1}(\Omega)$ and that
\begin{align*}
\int_{\Omega'} \left|\nabla(\varphi_i \circ \overrightarrow{\psi})\right| \, dx \leq \int_{\Omega'} \sqrt{2h(\overrightarrow{\psi}(x))} \left|\nabla\overrightarrow{\psi}(x)\right| \, dx
\end{align*}
for any open subsets $\Omega'$ of $\Omega$.
\end{proposition}

With this result in hand, we can prove the liminf property: 
Consider a sequence $\overrightarrow{\rho}^{\varepsilon}$ which converges to $\overrightarrow{\rho}^0$ strongly in $L^1(\Omega)^N$ as $\varepsilon\to0$. If $\liminf_{\varepsilon\to0}\mathcal{E}^{\varepsilon}(\overrightarrow{\rho}^{\varepsilon})=+\infty$, the statement trivially holds, so we  assume  that $\liminf_{\varepsilon\to0}\cE^{\varepsilon}(\overrightarrow{\rho}^{\varepsilon})<+\infty$ and fix $M>0$ and a subsequence (still denoted by $\overrightarrow{\rho}^{\varepsilon}$) along which
\begin{equation}\label{boundontheenergy}
\cE^{\varepsilon}(\overrightarrow{\rho}^{\varepsilon})\leq M.
\end{equation}
In particular, we have $\overrightarrow{\rho}^{\varepsilon}\in\R^N_{\geq0}$ a.e. on $\Omega$ and $\int_{\Omega}\overrightarrow{\rho}\,dx=\overrightarrow{m}$. Up to another sequence, we can further assume that $\overrightarrow{\rho}^{\varepsilon}$ converges to $\overrightarrow{\rho}^0$ a.e. on $\Omega$, and therefore that $W(\overrightarrow{\rho}^{\varepsilon}) \to W(\overrightarrow{\rho}^0)$ a.e. on $\Omega$. Since $\cE^{\varepsilon}(\overrightarrow{\rho}^{\varepsilon})\leq M$, we have
$$
\int_{\Omega}W(\overrightarrow{\rho}^{\varepsilon})\,dx\leq \varepsilon\cE^{\varepsilon}(\overrightarrow{\rho}^{\varepsilon})\leq M\varepsilon,
$$ and Fatou lemma implies that $W(\overrightarrow{\rho}^0)=0$ a.e. on $\Omega$. Consequently, we have $\overrightarrow{\rho}^0=\sum_{i=1}^N\theta_i\chi_{E_i}\overrightarrow{e_i}$ for some family of subsets $\{E_i\}_{i=1}^N$ of $\Omega$ and so $\overrightarrow{\rho}^0\in \cF(A,\overrightarrow{m}$).
 
\medskip

Denote $\overrightarrow{\phi}^{\varepsilon}$ solution to \eqref{eq:phi} and $\overrightarrow{\psi}^{\varepsilon}:=Q\overrightarrow{\phi^{\varepsilon}}$. 
For any collection  of  open sets $\{U_i\}_{i=0}^N$ such that $U_i\cap U_j=\emptyset$   for $i\neq j$,  Proposition \ref{prop:MM-functional-Baldo} and \eqref{eq:AC-lower-bound}  implies
\begin{align}\label{eq:MM-trick}
\sum_{i=0}^N\left|\nabla(\varphi_i\circ \overrightarrow{\psi}^{\varepsilon})\right|(U_i) 
& \leq  \sum_{i=0}^N\int_{U_i} \sqrt{2h(\overrightarrow{\psi}^{\varepsilon}(x))} \left|\nabla\overrightarrow{\psi}^{\varepsilon}(x)\right| \, dx \notag \\
& \leq \int_{\Omega} \sqrt{2h(\overrightarrow{\psi}^{\varepsilon}(x))} \left|\nabla\overrightarrow{\psi}^{\varepsilon}(x)\right| \, dx\notag\\
&\leq \frac{1}{\varepsilon}\int_{\Omega}h(\overrightarrow{\psi^{\varepsilon}})\,dx + \varepsilon\int_{\Omega}\frac12|\nabla\overrightarrow{\psi}^{\varepsilon}|^2\,dx \leq \cE^{\varepsilon}(\overrightarrow{\rho}^{\varepsilon}).
\end{align}
By taking the supremum over all such collections, 
we obtain\footnote{For two regular nonnegative Borel measures $\mu$ and $\nu$ on $\Omega$, the supremum $\mu \vee \nu$ is defined by
\begin{align*}
(\mu \vee \nu)(U) = \sup \{ \mu(U') + \nu(U'') : \  U' \cap U'' = \varnothing,\, U' \cup U'' \subset U,\, U' \text{ and } U'' \text{ are open sets in } \Omega \}.
\end{align*} 
It is the smallest regular measure which is greater than or equal to $\mu$ and $\nu$ on any Borel subsets of $\Omega$. 
}
\begin{align}\label{eq:MM-applied}
\bigvee_{i=0}^N\left|\nabla(\varphi_i\circ \overrightarrow{\psi}^{\varepsilon})\right|(\Omega) &:= \sup_{\{U_i\}}\left\{\sum_{i=0}^N\left|\nabla(\varphi_i\circ \overrightarrow{\psi}^{\varepsilon})\right|(U_i)\right\} 
\leq \cE^{\varepsilon}(\overrightarrow{\rho}^{\varepsilon}).
\end{align}

\medskip

We can then check that $\varphi_i\circ\overrightarrow{\psi}^{\varepsilon}$ converges to $\varphi_i\circ Q\overrightarrow{\rho}^0$ strongly in $L^1(\Omega)$. Indeed, the bound on the energy  \eqref{boundontheenergy} implies
\begin{align*}
\|\varphi_i\circ\overrightarrow{\psi}^{\varepsilon}-\varphi_i\circ Q\overrightarrow{\rho}^0\|_{L^1(\Omega)} & \leq \|\varphi_i\circ Q\overrightarrow{\phi}^{\varepsilon}-\varphi_i\circ Q\overrightarrow{\rho}^{\varepsilon}\|_{L^1(\Omega)} + \|\varphi_i\circ Q\overrightarrow{\rho}^{\varepsilon}-\varphi_i\circ Q\overrightarrow{\rho}^0\|_{L^1(\Omega)} \\
& \leq C\left(\left\|Q\left(\overrightarrow{\phi}^{\varepsilon}-\overrightarrow{\rho}^{\varepsilon}\right)\right\|_{L^2(\Omega)}+\|\overrightarrow{\rho}^{\varepsilon}-\overrightarrow{\rho}^{0}\|_{L^1(\Omega)}\right) \\
& \leq C\left(\sqrt{M\varepsilon}+\|\overrightarrow{\rho}^{\varepsilon}-\overrightarrow{\rho}^{0}\|_{L^1(\Omega)}\right)\to0\qquad\text{as }\varepsilon\to0.
\end{align*}
The lower semicontinuity of the BV norm then yields
\begin{align}\label{eq:lsc-BV}
\left|\nabla(\varphi_i\circ Q\overrightarrow{\rho}^{0})\right|(U)\leq\liminf_{\varepsilon\to0}|\nabla(\varphi_i\circ\overrightarrow{\psi}^{\varepsilon})|(U)
\end{align}
for all open subsets $U\subset \Omega$, and therefore
\begin{align*}
\bigvee_{i=0}^N\left|\nabla(\varphi_i\circ Q\overrightarrow{\rho}^{0})\right|(\Omega) & = \sup\left\{\sum_{i=0}^N\left|\nabla(\varphi_i\circ Q\overrightarrow{\rho}^{0})\right|(U_i)\,|\,U_i\subset\Omega\text{ is open and }U_i\cap U_j=\emptyset\text{ for $i\neq j$}\right\} \\
& \leq \liminf_{\varepsilon\to0} \bigvee_{i=0}^N\left|\nabla(\varphi_i\circ \overrightarrow{\psi}^{\varepsilon})\right|(\Omega).
\end{align*}

We deduce (using \eqref{eq:MM-applied}):
\begin{align*}
\bigvee_{i=0}^N\left|\nabla(\varphi_i\circ Q\overrightarrow{\rho}^{0})\right|(\Omega)\leq \liminf_{\varepsilon\to0} \bigvee_{i=0}^N\left|\nabla(\varphi_i\circ \overrightarrow{\psi}^{\varepsilon})\right|(\Omega) \leq \liminf_{\varepsilon\to0} \cE^{\varepsilon}(\overrightarrow{\rho}^{\varepsilon}).
\end{align*}
Finally, we recall that $\overrightarrow{\rho}^0=\sum_{i=1}^N\theta_i\chi_{E_i}\overrightarrow{e_i}$ and  \eqref{eq:MM-trick} and \eqref{eq:lsc-BV} imply that $\varphi_i\circ Q\overrightarrow{\rho}^0\in BV(\Omega)$ for all $i=0,\cdots,N$.  We thus have   $\overrightarrow{\rho}^0\in\cF(\Omega,\overrightarrow{m})$. Finally,  \cite[Proposition 2.2]{B90} gives
\begin{equation}\label{eq:bigVd}
\bigvee_{i=0}^N\left|\nabla(\varphi_i\circ Q\overrightarrow{\rho}^{0})\right|(\Omega)
=\sum_{i,j=0}^Nd(\overrightarrow{\alpha_i},\overrightarrow{\alpha_j})\cH^{d-1}(\partial^*E_i \cap \partial^*E_j \cap \Omega)
\end{equation}
which completes the proof.

\subsection{The limsup property}\label{subsec:limsup-property}
We now prove Theorem \ref{thm:gamma-convergence}-(ii), namely the existence of a recovery sequence.

\begin{proof}[Proof of Theorem \ref{thm:gamma-convergence}-(ii)]
Fix $\overrightarrow{\rho}^0\in L^1(\Omega)^N$. If $\cE^0(\overrightarrow{\rho}^0)=+\infty$, the statement holds trivially, so we assume that $\cE^0(\overrightarrow{\rho}^0)<+\infty$. We can then write $\overrightarrow{\rho}^0=\sum_{i=1}^N\theta_i\chi_{E_i}\overrightarrow{e_i}$
for a partition $\{E_i\}_{i=0}^N$ of $\Omega$ into $BV(\Omega)$-sets with volumes $|E_i|=\frac{m_i}{\theta_i}$ for $i=1,\cdots,N$ and $|E_0|=\left|\Omega\setminus\left(\cup_{i=1}^NE_i\right)\right|=|\Omega|-\sum_{i=1}^N\frac{m_i}{\theta_i}>0$ by \eqref{assumption:vacuum}. 
By \cite[Lemma 3.1]{B90} and a diagonal argument, we can assume without loss of generality that each $E_i$ is a polygonal domain in $\Omega$ and satisfies $\cH^{d-1}(\partial E_i\cap\partial\Omega)=0$ and that the vacuum set $E_0$ is a nonempty open set.

\medskip

As in \cite{M24}, the construction of the recovery sequence is done by using a recovery sequence $\{\overrightarrow{\psi}^\eps\}$ for the functional $\cF^\eps$ and using the relation (see Lemma \ref{lem:legendre})
\begin{align}\label{eq:equivalence-relation}
\overrightarrow{\rho}^\eps\in
\partial f^*(Q^T\overrightarrow{\psi^\eps}-\overrightarrow{a})
\end{align}
to define the corresponding recovery sequence $\{\overrightarrow{\rho}^{\varepsilon}\}$.

An important issue in the construction of the recovery sequence (already present for the one-species model $N=1$ in \cite{M24}) is that the function $\overrightarrow{\rho}^{\varepsilon}$  defined by \eqref{eq:equivalence-relation}  may not satisfy the mass constraint $\int_\Omega \overrightarrow{\rho}^{\varepsilon}(x) \, dx = \overrightarrow{m}$. 
This issue is significantly more difficult to address in the present case than in the case $N=1$ because modifying one component of $\overrightarrow{\psi^\eps}$ affects all the components of $\overrightarrow{\rho}^{\varepsilon}$.

\medskip

\noindent{\bf Step 1:} Construct a sequence $\{\overrightarrow{\rho^{\eps,0}}\}$ which does not satisfy the mass constraint $m_j$.

\medskip

As mentioned above, we first consider a sequence of functions $\{\overrightarrow{\psi}^{\varepsilon,0}\}_{\varepsilon\in(0,1)}$ such that $\overrightarrow{\psi}^{\varepsilon,0} \to \overrightarrow{\psi}^{0} :=Q\overrightarrow{\rho}^0$ as $\eps\to0$ and 
\begin{align}\label{eq:recovery=property-1}
\limsup_{\varepsilon\to0}\cF^{\varepsilon}\left(\overrightarrow{\psi}^{\varepsilon,0}\right) \leq \cE^0(\overrightarrow{\rho}^0).
\end{align}
The existence of such a sequence follows from the $\Gamma$-convergence of $\cF^\eps$ and its construction can be found in \cite{B90}. We record here some facts that will be used in the proof:
\begin{enumerate}
    \item The sequence $\{\overrightarrow{\psi}^{\varepsilon,0}\}$ is bounded in $L^\infty(\Omega)^N$ by a constant depending only $\overrightarrow{\rho}^0$.
    \item For all $i=0,\dots N$, there exist  sets $E_i^\eps \subset E_i$ such that $\overrightarrow{\psi}^{\varepsilon,0}(x)=\overrightarrow{\alpha}_i$ for all $x\in E_i^\eps$ and 
    \begin{equation}\label{eq:massEi}
     |E_i\setminus E_i^\eps|\leq C\eps.
    \end{equation} 
\item In particular 
\begin{equation}\label{eq:E_0}
\mathrm{supp}(\overrightarrow{\psi}^{\varepsilon,0})\subset \{x\in\Omega\,|\,\mathrm{dist}(x,\partial E_0) < C\varepsilon\text{ or }x\notin E_0\}
\end{equation}
for some constant $C=C(\overrightarrow{\rho}^0)>0$.  
    \item The transition region
$
\Omega_{tr} := \Omega \setminus  \cup_{i=0}^N E_i^\eps$ satisfies:
\begin{equation}\label{eq:transition-region-small}
|\Omega_{tr}|\leq C_1\varepsilon.
\end{equation}
for some  constant $C_1=C_1(\overrightarrow{\rho}^0)>0$. 
\end{enumerate}

With this recovery sequence at hand, we now define the corresponding density distribution using \eqref{eq:equivalence-relation}:
$$\widetilde{\overrightarrow{\rho}}^{\varepsilon,0} \in \partial f^*(Q^T\overrightarrow{\psi^{\varepsilon,0}}-\overrightarrow{a})$$
for which we have (see Proposition \ref{prop:minimization}):
\begin{align}\label{eq:energy-bound-rho-tilde}
\widetilde{\cE}^{\varepsilon} \left(\widetilde{\overrightarrow{\rho}}^{\varepsilon,0}\right) \leq \mathcal{G}^{\varepsilon}\left(\widetilde{\overrightarrow{\rho}}^{\varepsilon,0},\overrightarrow{\psi}^{\varepsilon,0}\right)=\cF^{\varepsilon}\left(\overrightarrow{\psi}^{\varepsilon,0}\right).
\end{align}

We immediately note that Lemma \ref{lem:legendre} (ii) guarantees that 
$\widetilde{\overrightarrow{\rho}}^{\varepsilon,0}(x)=\theta_j\overrightarrow{e_j}$ 
on the sets $E_i^\eps$, and so 
$$\widetilde{\rho}_i^{\varepsilon,0} (x)= \begin{cases} \theta_i & \mbox{ when } x\in E_i^\eps \\
0 & \mbox{ when } x\in E_j^\eps, \quad j\neq i.
\end{cases}
$$
Recalling that $|E_i| = \frac{m_i}{\theta_i}$, we then deduce (using \eqref{eq:massEi}, \eqref{eq:transition-region-small} and the fact that $\overrightarrow{\psi}^{\varepsilon,0}$ is bounded)  that
\begin{equation}\label{eq:masseps}
m_i-C_2\varepsilon \leq \int_{\Omega}\widetilde{\rho}_i^{\varepsilon,0}\,dx \leq m_i+C_2\varepsilon \qquad\text{for all }i=1,\cdots,N.
\end{equation}

\medskip

We now need to show that we can modify this sequence $\overrightarrow{\rho}^{\varepsilon,0} $ to get the correct mass constraint. This will be done in two steps: First, we show that we can replace $\overrightarrow{\rho}^{\varepsilon,0} $ such that each component $\widetilde{\rho}_i^{\varepsilon,0} $ has a smaller mass than desired: $\int_{\Omega}\widetilde{\rho}_i^{\varepsilon,0}\,dx \leq m_i$.
Then we will add a bit of mass to $\widetilde{\rho}_i^{\varepsilon,0} $ in the vacuum region $E_0^\eps$ to recover the mass constraint.

\medskip

\noindent{\bf Step 2:} Construct a sequence  $\{\overrightarrow{\rho}^{\eps,0}\}$ with masses smaller than $m_j$.

\medskip

We define
\begin{align*}
\overrightarrow{\rho}^{\varepsilon,0} := \zeta\,\widetilde{\overrightarrow{\rho}}^{\varepsilon,0}\qquad\text{with }\zeta := \min_{j=1,\cdots,N}\left\{\frac{m_j-\varepsilon }{m_j+C_2\varepsilon }\right\}.
\end{align*}
In view of \eqref{eq:masseps}, the components of  $\overrightarrow{\rho}^{\varepsilon,0}  $ satisfy $\int_{\Omega}\rho_i^{\varepsilon,0}\,dx\leq m_i$ and the mass defects satisfy
\begin{align}\label{eq:step1-conclusion}
\overline{m}_i^{\varepsilon}:=m_i-\int_{\Omega}\rho_i^{\varepsilon,0}\,dx\in\left[C_0^{-1}\varepsilon ,C_0\varepsilon\right]\qquad\text{for all $i=1,\cdots,N$}
\end{align}
for some constant $C_0=C_0(\overrightarrow{\rho}^0)>1$. 

Moreover, we can show that
\begin{align}\label{eq:rho-0-approximate}
\widetilde{\cE}^{\varepsilon} \left(\overrightarrow{\rho}^{\varepsilon,0}\right)\leq \widetilde{\cE}^{\varepsilon} \left(\widetilde{\overrightarrow{\rho}}^{\varepsilon,0}\right) + C\varepsilon .
\end{align}
Indeed, 
denoting by $\widetilde{\overrightarrow{\phi}}^{\varepsilon,0}$ (respectively $\overrightarrow{\phi}^{\varepsilon,0}=\zeta\,\widetilde{\overrightarrow{\phi}}^{\varepsilon,0}$) the solution of \eqref{eq:phi}  with $\widetilde{\overrightarrow{\rho}}^{\varepsilon,0}$ (respectively $\overrightarrow{\rho}^{\varepsilon,0}$) and 
using the fact that $\zeta\leq1$, we get:
\begin{align*}
& \frac{1}{\varepsilon}\int_{\Omega}\frac12|Q\overrightarrow{\phi}^{\varepsilon,0}-Q\overrightarrow{\rho}^{\varepsilon,0}|^2\,dx + \varepsilon\int_{\Omega}\frac12|\nabla Q\overrightarrow{\phi}^{\varepsilon,0}|^2\,dx \\
&\qquad\qquad = \zeta^2\left(\frac{1}{\varepsilon}\int_{\Omega}\frac12|Q\widetilde{\overrightarrow{\phi}}^{\varepsilon,0}-Q\widetilde{\overrightarrow{\rho}}^{\varepsilon,0}|^2\,dx + \varepsilon\int_{\Omega}\frac12|\nabla Q\widetilde{\overrightarrow{\phi}}^{\varepsilon,0}|^2\,dx\right) \\
&\qquad\qquad  \leq \frac{1}{\varepsilon}\int_{\Omega}\frac12|Q\widetilde{\overrightarrow{\phi}}^{\varepsilon,0}-Q\widetilde{\overrightarrow{\rho}}^{\varepsilon,0}|^2\,dx + \varepsilon\int_{\Omega}\frac12|\nabla Q\widetilde{\overrightarrow{\phi}}^{\varepsilon,0}|^2\,dx.
\end{align*}

To control the contribution of the double-well potential, we first note that 
$0\leq 1-\zeta\leq \frac{1+C_2}{\inf m_j} \eps$ and so
$$
\left\|\widetilde{\overrightarrow{\rho}}^{\varepsilon,0} - \overrightarrow{\rho}^{\varepsilon,0}\right\|_{L^\infty(\Omega)}\leq C\varepsilon .$$
Using the fact that $W$ is locally Lipschitz continuous, we can then write
\begin{align*}
\frac{1}{\varepsilon}\int_{\Omega_{tr}}W\left(\overrightarrow{\rho}^{\varepsilon,0}\right)\,dx &\leq \frac{1}{\varepsilon}\int_{\Omega_{tr}}W\left(\widetilde{\overrightarrow{\rho}}^{\varepsilon,0}\right)+C\varepsilon \,dx \\
&\leq \frac{1}{\varepsilon}\int_{\Omega_{tr}}W\left(\widetilde{\overrightarrow{\rho}}^{\varepsilon,0}\right) \,dx + C\varepsilon^{-1}|\Omega_{tr}|\varepsilon  \leq \frac{1}{\varepsilon}\int_{\Omega_{tr}}W\left(\widetilde{\overrightarrow{\rho}}^{\varepsilon,0}\right) \,dx + C\varepsilon 
\end{align*}
where we used the fact that the transition region is small (see \eqref{eq:transition-region-small}).

Next, the fact that $W$ has a minimum at $\theta_j\overrightarrow{e_j}$ (for all $j=1,\cdots,N$) implies that 
$W(\theta_j\overrightarrow{e_j}+\overrightarrow{p})\leq C|\overrightarrow{p}|^2$    for  $|\overrightarrow{p}| $  small enough satisfying $\theta_j\overrightarrow{e_j}+\overrightarrow{p}\in \R^N_{\geq0}$.
This allows us to control the contribution of the double-well potential in the bulk region  $  \Omega \setminus \Omega_{tr}$, 
\begin{align*}
\frac{1}{\varepsilon}\int_{ \Omega \setminus \Omega_{tr}}W\left(\overrightarrow{\rho}^{\varepsilon,0}\right)\,dx &\leq \frac{1}{\varepsilon}\int_{ \Omega \setminus \Omega_{tr}}W\left(\widetilde{\overrightarrow{\rho}}^{\varepsilon,0}\right) + C\varepsilon ^2\,dx \leq \int_{ \Omega \setminus \Omega_{tr}}W\left(\widetilde{\overrightarrow{\rho}}^{\varepsilon,0}\right)\,dx + C\varepsilon .
\end{align*}
This completes the proof of  \eqref{eq:rho-0-approximate}, which, together with 
 \eqref{eq:energy-bound-rho-tilde}, imply
\begin{align}\label{eq:no-cost-0}
\widetilde{\cE}^{\varepsilon} \left(\overrightarrow{\rho}^{\varepsilon,0}\right) \leq \cF^{\varepsilon}\left(\overrightarrow{\psi}^{\varepsilon,0}\right) + C\varepsilon. 
\end{align}

\medskip

\noindent{\bf Step 3:} Use the vacuum region to increase the masses as needed.

\medskip

For this construction, we assume (without loss of generality) that 
\begin{equation}\label{eq:order}
\alpha_{11} \geq \alpha_{22} \geq \dots \geq \alpha_{NN}.
\end{equation}

\medskip

We recall that $\overrightarrow{\psi}^{\varepsilon,0}$, and therefore $\overrightarrow{\rho}^{\varepsilon,0}$, vanishes on the vacuum set $E_0^\eps$.
We will thus construct a corrector $\overrightarrow{\eta}^{\varepsilon}$
supported in $E^\eps_0$ such that $\overrightarrow{\rho}^{\varepsilon}=\overrightarrow{\rho}^{\varepsilon,0}+\overrightarrow{\eta}^{\varepsilon}$ is a recovery sequence satisfying 
$\int_{\Omega}\overrightarrow{\rho}^{\varepsilon}\,dx=\overrightarrow{m}$, that is
\begin{equation}\label{eq:corrector}
\mathrm{Supp}\overrightarrow{\eta}^{\varepsilon} \subset E_0^\eps , \quad \eta_i^\eps \geq 0 \mbox{ in } \Omega ,\quad
\int_\Omega \eta_i^\eps(x)\, dx = \overline m_i^\eps \qquad i=1,\dots ,N
\end{equation}
with the mass defects $\overline m_i^\eps$ defined by \eqref{eq:step1-conclusion}.

\medskip

\medskip
We recall that $E_0^\eps$ approaches the vacuum set $E_0$ - see \eqref{eq:E_0}. 
By assumption \ref{assumption:vacuum}, this set satisfies $|E_0|=\left|\Omega\setminus\left(\cup_{i=1}^NE_i\right)\right|=|\Omega|-\sum_{i=1}^N\frac{m_i}{\theta_i}>0$.
We can thus find $N$ points $x_1,\dots x_N$ in the interior of $E_0^\eps$ such that
\begin{align*}
\mathrm{dist}(x_i,x_j)>\delta_0\quad\text{for all }i\neq j \quad \text{ and }\quad \mathrm{dist}(x_i,\mathrm{supp}(\overrightarrow{\rho}^{\varepsilon,0}))>\delta_0\text{ for any }i=1,\cdots,N
\end{align*}
for some $\delta_0=\delta_0(\overrightarrow{\rho}^0)>0$. 
We further define the radii $r_j$ such that $\theta_j|B(x_j,r_j)| =  \overline{m}_j^\varepsilon$ and note that \eqref{eq:step1-conclusion} then implies that 
\begin{equation}\label{eq:rjeps}
   \bar C_0^{-1} \eps^{\frac 1 d} \leq r_j\leq \bar C_0\eps^{\frac 1 d}. 
\end{equation}
In particular, for $\eps$ small enough, the balls $B_j:=B\left(x_j,r_j+\varepsilon^{\frac{3}{2d}}\right)$ are disjoint subsets of $E_0^\eps$ 
such that
\begin{align}\label{eq:disjoint}
\mathrm{dist}(B_i,B_j)>\delta_0/2\quad\text{for any }i\neq j \text{ and }\mathrm{dist}(B_i,\mathrm{supp}(\overrightarrow{\rho}^{\varepsilon,0}))>\delta_0/2\text{ for any }i=1,\cdots,N,
\end{align}

The construction of the corrector is based on the following functions supported in $\overline{B_j}$ (see Figure \ref{fig:discrete-motion-law}):
\begin{align*}
\chi_j^{\varepsilon, s_j}(x) :=
\begin{cases}
    1 & \mbox{ when } |x-x_j|\leq s_j \\
    1-\varepsilon^{-\frac{3}{2d}}(|x-x_j|-s_j) &  \mbox{ when } s_j \leq |x-x_j| \leq s_j+\eps^{\frac{3}{2d}} \\
    0 &\mbox{ when }  |x-x_j| \geq s_j+\eps^{\frac{3}{2d}}
\end{cases}
\end{align*}
where $s_j\in[0,r_j]$ will be chosen later.
\medskip

\begin{figure}[thp]
	\begin{center}
            \includegraphics[height=5cm]{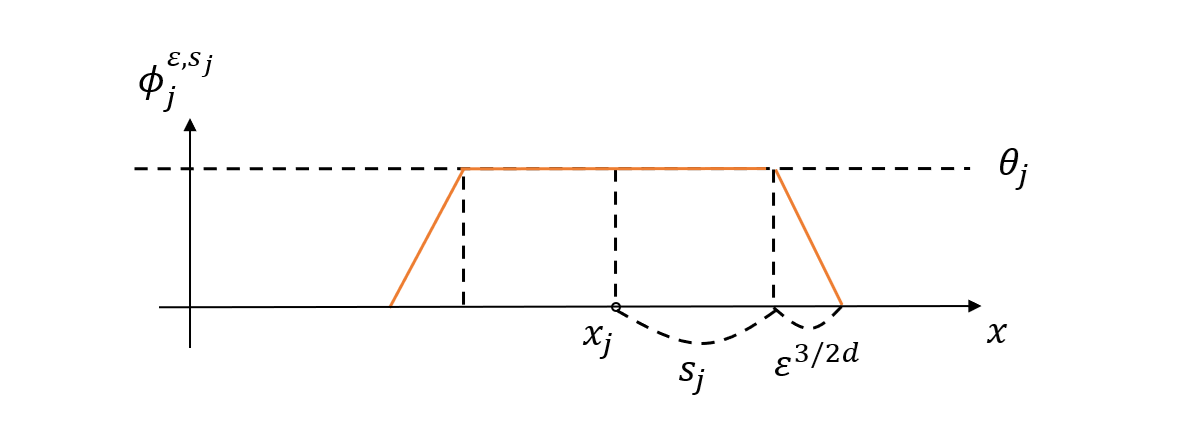}
		\vskip 0pt
		\caption{The graph of $\phi_j^{\varepsilon, s_j}$.}
        \label{fig:discrete-motion-law}
	\end{center}
\end{figure}

We now proceed to define the corrector in each of the ball $B_i$:

In $B_1$, we define the corrector $\overrightarrow{\eta}^\eps$ so that 
\begin{equation}\label{eq:eta1}
\overrightarrow{\eta}^\eps(x)\in \pa f^*(\chi_1^{\eps,s_1}(x) Q^T  \overrightarrow{\alpha_1}-\overrightarrow{a}) \qquad \forall x\in B_1
\end{equation}
(this amounts to adding the corrector  $\chi^{\eps,s_1} (x)\overrightarrow{\alpha_1}$ to the recovery sequence $\overrightarrow{\psi}^{\eps,0}(x)$).
We make the following observations:
\begin{itemize}
\item  $| \overrightarrow{\eta}^\eps(x)|$ is bounded (uniformly in $x$) by some constant $C_3$ depending only on $m$ and $A$.
\item In view of Lemma \ref{lem:legendre}-(ii), \eqref{eq:eta1} implies that  
$\overrightarrow{\eta}^\eps =\theta_1\overrightarrow{e_1}$ in $ B_{s_1}(x_1)$ and
$\overrightarrow{\eta}^\eps =  0 $ outside the balls $ \overline {B_{s_j+ \eps^{\frac{3}{2d}}}}(x_j)$. 
\item  In particular, when we take $s_1=r_1$, we have
\begin{align}\label{eq:lower-bound-continuity} \int_{B_1}\eta_1^\eps(x) \,dx \geq \int_{B_{r_1}(x_1)}\theta_1\,dx = \overline{m}_1^{\varepsilon} 
\end{align}
(by definition of $r_j$) and when  we take $s_j=0$ we find
\begin{align}\label{eq:upper-bound-continuity}
\int_{B_1}\eta_1^\eps (x)\,dx \leq C_3\omega_d\varepsilon^{\frac32},
\end{align}
The right hand side in \eqref{eq:upper-bound-continuity} is smaller than $\overline{m}_1$ as long as $\varepsilon\in(0,1)$ is small enough since $C_3\omega_d\varepsilon^{\frac32} \leq C_0^{-1}\varepsilon \leq \overline{m}_1$. 
\end{itemize}

By a continuity argument, it is thus possible to find $s_1\in [0,r_1]$ and $\overrightarrow{\eta}^\eps$ satisfying \eqref{eq:eta1} such that
$$\int_{B_1}\eta_1^\eps(x) \,dx  = \overline{m}_1^{\varepsilon}  . $$

We need to check what happens with the other components of $\overrightarrow{\eta}^\eps$: For $j=2,\dots, N$, we have $\eta_j^\eps=0$ in  $B_{s_1}(x_1) $ (since $\overrightarrow{\eta}(x) =\theta_1\overrightarrow{e_1}$ in that set),
so these other components have support in the ring $\{ s_1\leq  |x-x_1| \leq  s_1+\varepsilon^{\frac{3}{2d}}\}$. Since this will add some mass to the corresponding densities, we need to modify the mass defects $\overline{m}_i^{\varepsilon}  $ for $i=2,\dots N$ before moving forward with the construction of $\eta(x)$ in $B_2$.
Fortunately, this ring has a small volume. In fact, its volume is the largest when $s_1=r_1$ so that (recalling that $r_j \sim \eps^{\frac 1 d}$ - see \eqref{eq:rjeps}):
\begin{equation}\label{eq:ring}
|\{ s_1\leq  |x-x_1| \leq  s_1+\varepsilon^{\frac{3}{2d}}\}|\leq \left(\left(r_j + \varepsilon^{\frac{3}{2d}}\right)^d-r_j^d\right) \leq C \eps^{1+\frac{1}{2d}} .
\end{equation}
We thus have 
\begin{align}\label{eq:shell}
\int_{B_1}\eta_j^\eps(x)\,dx \leq   C_4 \varepsilon^{1+\frac{1}{2d} }
\qquad  \mbox{ for all } j=2,\dots , N.
\end{align}
and we can change the mass defect $\overline{m}_j^{\varepsilon}$, $j=2,\dots,N$ to
\begin{align*}
\overline{m}_j'^{\varepsilon} := \overline{m}_j^{\varepsilon} - \int_{B_1}\eta_j^{\varepsilon}(x)\,dx.
\end{align*}
In view of \eqref{eq:step1-conclusion} and \eqref{eq:shell}, we still have
\begin{align}\label{eq:second-mass-nonnegative}
 C_0^{-1} \eps/2 \leq \overline{m}_j'^{\varepsilon}\leq C_0\eps
\end{align}
as long as  $\varepsilon\in(0,1)$ is small enough. 
\medskip

We can now proceed in the same way to construct the corrector $\overrightarrow{\eta}^\eps$ in $B_2$. It should satisfy
\begin{equation}\label{eq:eta2}
\overrightarrow{\eta}^\eps(x)\in \pa f^*(\chi_2^{\eps,s_2}(x) Q^T \overrightarrow{\alpha_2}-\overrightarrow{a}) \qquad \forall x\in B_2.
\end{equation}
and for the same reasons as above, it is possible to choose $s_2\in[0,r_2]$ so that
\begin{align*}
\int_{B_2} \eta_2^{\varepsilon}(x)\,dx = {\overline{m}'}_2^{\varepsilon}.
\end{align*}
This implies that 
$$\int_{B_1\cup B_2} \eta_2^{\varepsilon}(x)\,dx = \overline{m}_2^{\varepsilon}.$$

Crucially, this corrector might have non trivial components $\eta_j(x)$ for $j=3,\dots ,N$  in the ring
$\{ s_2\leq  |x-x_2| \leq  s_2+\varepsilon^{\frac{3}{2d}}\}$,
but thanks to \eqref{eq:order} and Lemma \ref{lem:legendre}-(iii), we can choose $\overrightarrow{\eta}^\eps(x)$ satisfying \eqref{eq:eta2} such that $\eta_1^\eps(x) = 0$ in $B_2$. We thus still have
$$\int_{B_1\cup B_2} \eta_1^{\varepsilon}(x)\,dx = \overline{m}_1^{\varepsilon}.$$
This is the key remark that allows us to adjust the mass of the second component  while keeping the mass of the first component constant.

We can now iterate this process, taking $\eps$ a bit smaller at each step and constructing a corrector $\overrightarrow{\eta}^\eps(x)$ satisfying in particular
$$
\mathrm{Supp}\; \eta_i^\eps \subset \cup_{j\leq i}B_j, \quad \int_\Omega \eta_i^\eps(x) \, dx= \overline m_i^\eps.
$$

The function  
$$\overrightarrow{\rho}^{\varepsilon} = \overrightarrow{\rho}^{\varepsilon,0} + \overrightarrow{\eta}^\eps$$
 satisfies
\begin{align}\label{eq:rho-conclusion}
\overrightarrow{\rho}^\varepsilon\in L^{2}(\Omega;\R^N_{\geq0}),\,\int_{\Omega}\overrightarrow{\rho}^\varepsilon\,dx=\overrightarrow{m},\text{ and }\overrightarrow{\rho}^\varepsilon\text{ converges to $\overrightarrow{\rho}^0$ strongly in }L^{1}(\Omega;\R^N).
\end{align}
Furthermore, since the $\chi_i^{\eps,s_i}(x)$ have disjoint support, we can write:
$$ \overrightarrow{\eta}^\eps(x) \in \pa f^*(Q^T\overrightarrow{\chi}^\eps(x) -\overrightarrow{a}) \qquad \forall x\in B_1
$$
where
$
\overrightarrow{\chi}^\eps = \sum_{i=1}^N\chi_i^{\eps,s_i}(x) \overrightarrow{\alpha_i}.$

\medskip
\medskip

\noindent{\bf Step 3:} Estimate the contribution of the corrector to the energy.

\medskip

Define 
$$\overrightarrow{\psi}^{\varepsilon} =  \overrightarrow{\psi}^{\varepsilon, 0} + \overrightarrow{\chi}^{\varepsilon}.$$
We note that the two term have disjoint supports so we can write:
\begin{align}\label{eq:psi-disjoint-sum}
\cF^{\varepsilon}(\overrightarrow{\psi}^{\varepsilon}) = \cF^{\varepsilon}\left(\overrightarrow{\psi}^{\varepsilon, 0}\right)+\sum_{j=1}^N\left(\frac{1}{\varepsilon}\int_{B_j}h(\overrightarrow{\chi}^{\varepsilon})\,dx + \varepsilon\int_{B_j}\frac12|\nabla\overrightarrow{\chi}^{\varepsilon}|^2\,dx\right).
\end{align}
Since $\overrightarrow{\chi}^{\varepsilon}= \overrightarrow{\alpha_j}$ in $B_{s_j}(x_j)$, the function $h(\overrightarrow{\chi}^{\varepsilon})$ is supported in the ring $\{ s_j\leq  |x-x_1| \leq  s_j+\varepsilon^{\frac{3}{2d}}\}$ and so
(using \eqref{eq:ring}):
\begin{align}\label{eq:psi-term-1}
\frac{1}{\varepsilon}\int_{B_j}h(\overrightarrow{\chi}^{\varepsilon})\,dx 
& \leq 
\frac{C}{\varepsilon}|\{ s_j\leq  |x-x_1| \leq  s_j+\varepsilon^{\frac{3}{2d}}\}|   \nonumber \\
& \leq 
C\varepsilon^{-1}\omega_d\left(\left(r_j + \varepsilon^{\frac{3}{2d}}\right)^d-r_j^d\right) \leq C\varepsilon^{\frac{1}{2d}} \to0\qquad\text{as }\varepsilon\to0.
\end{align}
Similarly, $\nabla\overrightarrow{\chi}^{\varepsilon}$ is supported in that same ring and satisfies $|\na\overrightarrow{\chi}^{\varepsilon} (x)|\leq\varepsilon^{-\frac3{2d}}$ so that 
\begin{align}\label{eq:psi-term-2}
\varepsilon\int_{B_j}\frac12|\nabla\overrightarrow{\chi}^{\varepsilon}|^2\,dx \leq C\varepsilon\,\omega_d\left(\left(r_j + \varepsilon^{\frac{3}{2d}}\right)^d-r_j^d\right)\,\varepsilon^{-\frac3{2d}} \leq C\varepsilon^{1-\frac{1}{d}} \to0\qquad\text{as }\varepsilon\to0
\end{align}
for $d\geq2$. Combining \eqref{eq:recovery=property-1}, \eqref{eq:psi-disjoint-sum}, \eqref{eq:psi-term-1}, and \eqref{eq:psi-term-2}, we obtain
\begin{align}\label{eq:recovery}
\limsup_{\varepsilon\to0}\cF^{\varepsilon}(\overrightarrow{\psi}^{\varepsilon}) \leq \cE^0(\overrightarrow{\rho}^0).
\end{align}

\medskip

For $j=1,\cdots,N$, denote by $\overrightarrow{\rho}^{\varepsilon,j}$ defined by $\overrightarrow{\rho}^{\varepsilon,j}=\overrightarrow{\rho}^{\varepsilon}$ on the $j$-th ball $B_j$ and by zero outside $B_j$. Similarly, for $j=1,\cdots,N$, denote by $\overrightarrow{\psi}^{\varepsilon,j}$ defined by $\overrightarrow{\psi}^{\varepsilon,j}=\overrightarrow{\psi}^{\varepsilon}$ on the $j$-th ball $B_j$ and by zero outside $B_j$. As \eqref{eq:equivalence-relation} holds with $\overrightarrow{\rho}^{\varepsilon,j}$ and $\overrightarrow{\psi}^{\varepsilon,j}$ for each $j=1,\cdots,N$, we have, by Proposition \ref{prop:minimization},
\begin{align}\label{eq:no-cost-j}
\widetilde{\cE}^{\varepsilon}\left(\overrightarrow{\rho}^{\varepsilon,j}\right)\leq \cF^{\varepsilon}\left(\overrightarrow{\psi}^{\varepsilon,j}\right)\qquad\text{for each }j=1,\cdots,N.
\end{align}
Therefore, combining \eqref{eq:no-cost-0}, \eqref{eq:disjoint}, \eqref{eq:rho-conclusion}, \eqref{eq:no-cost-j}, and successive applications of Proposition \ref{prop:minimization}, we obtain
\begin{align*}
\cE^{\varepsilon}(\overrightarrow{\rho}^{\varepsilon}) = \widetilde{\cE}^{\varepsilon}\left(\sum_{j=0}^N\overrightarrow{\rho}^{\varepsilon,j}\right) & \leq \sum_{j=0}^N\widetilde{\cE}^{\varepsilon}\left(\overrightarrow{\rho}^{\varepsilon,j}\right) + NCe^{-\frac{\delta_0}{C\varepsilon}} \\
& \leq \sum_{j=0}^N\cF^{\varepsilon}\left(\overrightarrow{\psi}^{\varepsilon,j}\right) + C\varepsilon\delta^{-1} + NCe^{-\frac{\delta_0}{C\varepsilon}} \\
& \leq \cF^{\varepsilon}\left(\overrightarrow{\psi}^{\varepsilon}\right) + C\varepsilon\delta^{-1} + NCe^{-\frac{\delta_0}{C\varepsilon}}.
\end{align*}
Taking the limit supremum as $\varepsilon\to0$ to the above and using \eqref{eq:recovery} finish the proof.
\end{proof}

\section{Applications to cell-cell adhesion and the DAH: Proof of Theorem \ref{thm:DAH}}\label{sec:DAH}
In this section, we establish some key properties of the limiting interfacial energy $\cE^0$ depending on the interaction coefficients $\alpha_{ij}$.
We recall that the  adhesion coefficients appearing in $\cE^0$ depend on the $\alpha_{ij}$ via the formula\eqref{eq:metric-coefficient-before-change-of-variables} which we recall here for convenience:
$$    \sigma_{ij}:=d(\overrightarrow{\alpha_i},\overrightarrow{\alpha_j}), \qquad i,j = 0,\dots N
$$
where the function $d:\R^N\to [0,\infty]$ is the distance defined by \eqref{eq:distance-definition-before-change-of-variables} and also recalled here:
\begin{align*}
d(\overrightarrow{\zeta_0},\overrightarrow{\zeta_1}):=\inf\left\{\frac{1}{\sqrt{2}}\int_0^1\sqrt{h(\gamma(t))}|\gamma'(t)|\,dt\,|\,\text{$\gamma:[0,1]\to\R^N$ is piecewise $C^1$} ,\,\gamma(0)=\overrightarrow{\zeta_0},\,\gamma(1)=\overrightarrow{\zeta_1}\right\}
\end{align*}
for any $\overrightarrow{\zeta_0},\overrightarrow{\zeta_1}\in\R^N$.
In what follows, we will refer as geodesic the curve that realizes the infimum in this definition (assuming it exists. the rigorous arguments will rely on almost geodesics).

As detailed in Section \ref{subsec:steinberg}, the properties of $\cE^0$ are related to cases of equality in the triangle inequality
$$
d(\overrightarrow{\alpha_i},\overrightarrow{\alpha_j})
\leq d(\overrightarrow{\alpha_i},\overrightarrow{\alpha_k})
+d(\overrightarrow{\alpha_k},\overrightarrow{\alpha_j})
$$
which we explore in this section.

\subsection{Non-interacting species}\label{subsec:potentials}

We start with a simple particular case in which cells of different species do not interact ($\alpha_{ij}=0$ if $i\neq j$). In that case we can prove the following proposition which implies Theorem \ref{thm:DAH}-(i) when $N=2$:
\begin{proposition}\label{prop:sorting}
If $A$ is diagonal, then $d(\overrightarrow{\alpha_i},\overrightarrow{\alpha_j})
= d(\overrightarrow{\alpha_i},0) +d(0,\overrightarrow{\alpha_j})$ that is:
\begin{align*}
\sigma_{ij} = \sigma_{i0} + \sigma_{0j}\qquad\text{for any }i,j=1,\cdots,N.
\end{align*}
Therefore, for $\overrightarrow{\rho}=\sum_{i=1}^N\theta_i\chi_{E_i}\overrightarrow{e_i}\in\cF(A,\overrightarrow{m})$,
we have
\begin{align*}
\cE^{0}(\overrightarrow{\rho}) = \sum_{i,j=0}^N\sigma_{ij}\cH^{d-1}(\partial^*E_i \cap \partial^*E_j \cap \Omega) = 2 \sum_{i=1}^N \sigma_{0i}\cH^{d-1}(\partial^*E_i \cap \Omega).
\end{align*}
\end{proposition}
In particular, if  $ \overrightarrow{\rho} $ minimizes $\cE^0$, then  each $E_i$ minimizes the energy $\cH^{d-1}(\partial^*E_i \cap \Omega)$, which leads to surfaces with constant mean-curvature that intersects the fixed boundary $\Omega$ with normal contact angle.

\begin{proof}[Proof of Proposition \ref{prop:sorting}]
Since $A$ is diagonal, the matrix $Q$ is the diagonal matrix with entries $\alpha_{jj}^{1/2}$ and we have 
$$ \overrightarrow{\alpha_i} = \theta_i \alpha_{ii}^{1/2}  \overrightarrow{e_i} .$$
In particular, we can write:
\begin{align}
h(\overrightarrow{\psi}) & = \inf_{\overrightarrow{\rho}\in\R^N}\left\{W(\overrightarrow{\rho})+\frac12|\overrightarrow{\psi}-Q\overrightarrow{\rho}|^2\right\} \nonumber\\
& = \inf_{\overrightarrow{\rho}\in\R^N}\left\{f\left(\sum_j\rho_j\right) + \sum_ja_j\rho_j-\frac12\sum_j\alpha_{jj}\rho_j^2+\frac12\sum_j(\psi_j-\alpha_{jj}^{1/2}\rho_j)^2\right\} \nonumber \\
& \geq \inf_{\overrightarrow{\rho}\in\R^N}\left\{\sum_jf\left(\rho_j\right) + \sum_ja_j\rho_j-\frac12\sum_j\alpha_{jj}\rho_j^2+\frac12\sum_j(\psi_j-\alpha_{jj}^{1/2}\rho_j)^2\right\} \nonumber \\
& \geq \sum_{j} h(\psi_j\overrightarrow{e_j})\label{eq:hine}
\end{align}
where we used the fact that $\sum_jf\left(\rho_j\right)\leq f\left(\sum_j\rho_j\right)$ when $f$ is given by \eqref{eq:f}.

We can now use this inequality to show that the geodesic (that is the curve that realizes the minimum in the definition of $d(\overrightarrow{\alpha_i},\overrightarrow{\alpha_j})$ \eqref{eq:distance-definition-before-change-of-variables}) 
from $\overrightarrow{\alpha_i}$ to $\overrightarrow{\alpha_j}$  is the curve $\gamma_0$ which is the  piecewise straight line along the axes from $\overrightarrow{\alpha_i}$ to $0$ and then from $0$ to  $\overrightarrow{\alpha_j}$:

Fix  $i=1$ and $j=2$ (without loss of generality) and let 
$\overrightarrow{\psi}=(\psi_1,\cdots,\psi_N):[0,1]\to\R^N$ be an arbitrary piecewise $C^1$ curve joining $\overrightarrow{\alpha_1}$ to $\overrightarrow{\alpha_2}$, that is
\begin{equation}\label{eq:endpoints}
\overrightarrow{\psi}(0) = (\theta_1 \alpha_{11}^{1/2},0,\dots 0), \qquad 
\overrightarrow{\psi}(1) = (0,\theta_2 \alpha_{22}^{1/2},0,\dots 0).
\end{equation}

By Cauchy-Schwarz inequality and \eqref{eq:hine}, we have
\begin{align*}
\int_0^1\sqrt{h(\overrightarrow{\psi})}|\overrightarrow{\psi}'|\,dt &\geq \int_0^1\sqrt{\sum_{k=1}^N h(\psi_k\overrightarrow{e_k})}\sqrt{\sum_{k=1,2}\psi_k'^2}\,dt \\
& \geq \sum_{k=1}^N\int_0^1\sqrt{h(\psi_k\overrightarrow{e_k})}|\psi_k'|\,dt.
\end{align*}
In view of \eqref{eq:endpoints}, this right-hand side is minimum when $\psi_k(s)=0$ for all $s\in [0,1]$ for $k=3, \dots, N$.
Furthermore, the right-hand side does not depend on the choice of the curve $\overrightarrow{\psi}$ and agrees with the left-hand side when $\overrightarrow{\psi}=\gamma_0$, which finishes the proof.
\end{proof}

\subsection{The phase space and engulfment}\label{subsec:engulfment}
We now take $N=2$ (only two types of cell are interacting) and write the interaction coefficients matrix as $A=\begin{pmatrix}
\alpha & \gamma \\
\gamma & \beta
\end{pmatrix}$ with $\alpha\geq\beta>0$. We choose the matrix  $Q$ (which satisfies $A=Q^TQ$) to be $Q=\begin{pmatrix}
\sqrt{\alpha} & \frac{\gamma}{\sqrt{\alpha}} \\
0 & \sqrt{\beta - \frac{\gamma^2}{\alpha}}
\end{pmatrix}$.

\medskip

The phase vectors $\overrightarrow{\alpha_i}$ are then given by
\begin{align*}
\overrightarrow{\alpha_0} = \overrightarrow{0},\quad\overrightarrow{\alpha_1} &= \theta_1 Q \overrightarrow{e_1} = \left(\frac12 \alpha\right)^{\frac{1}{m-2}} \begin{pmatrix} \sqrt{\alpha} \\ 0 \end{pmatrix},\quad\text{and }\,\,\overrightarrow{\alpha_2} = \theta_2 Q \overrightarrow{e_2} = \left(\frac12 \beta\right)^{\frac{1}{m-2}} \begin{pmatrix} \frac{\gamma}{\sqrt{\alpha}} \\ \sqrt{\beta - \frac{\gamma^2}{\alpha}} \end{pmatrix},
\end{align*}
and are shown in Figure \ref{fig:phase} in the phase plane. 
As $\gamma$ varies, the vector $\overrightarrow{\alpha_2}$ moves along a circle of a fixed radius in the clockwise direction.

\begin{figure}[htbp]
	\begin{center}
            \includegraphics[height=5cm]{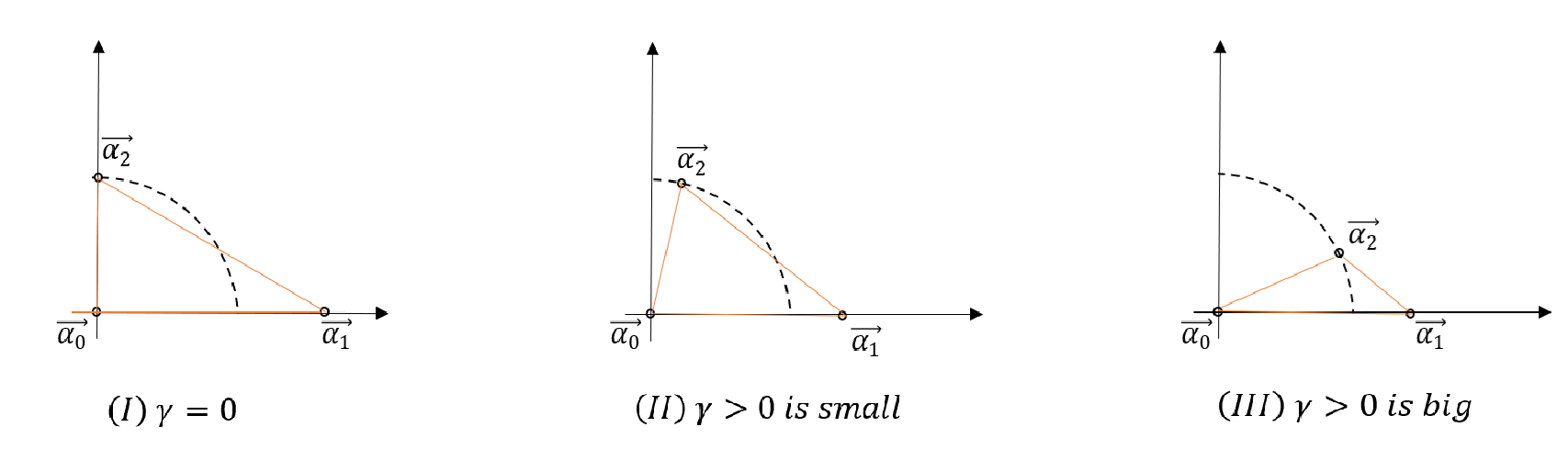}
            \captionsetup{width=0.8\textwidth} 
		\vskip 0pt
		\caption{The phase vectors $\overrightarrow{\alpha_0}$ (vacuum), $\overrightarrow{\alpha_1}$ (say, cell A), and $\overrightarrow{\alpha_2}$ (cell B).}
        \label{fig:phase}
	\end{center}
\end{figure}

In the weak coupling regime, when $\gamma>0$ is small (Figure \ref{fig:phase}(II)), the vector $\overrightarrow{\alpha_2}$ enters the first quadrant and gets closers to the vector $\overrightarrow{\alpha_1}$ (in the Euclidean metric), indicating that the cross-adhesion between cells A and B becomes stronger. We will prove that the geodesic from  $\overrightarrow{\alpha_1}$ to $\overrightarrow{\alpha_2}$ does not pass through $\overrightarrow{\alpha_0}$ (as was the case when $\gamma=0$) leading to strict triangle inequalities and proving Theorem \ref{thm:DAH}-(ii).
  
In the strong coupling regime  (larger $\gamma$, see Figure \ref{fig:phase}(III)), the vector $\overrightarrow{\alpha_2}$ gets very close to the $x$-axis. 
Because the potential function $h$ vanishes at $\overrightarrow{\alpha_2}$, we will show that the geodesic from $\overrightarrow{\alpha_0}$ to $\overrightarrow{\alpha_1}$ passes through $\overrightarrow{\alpha_2}$, leading to an equality in the triangle inequality and proving Theorem \ref{thm:DAH}-(iii).
 
\medskip

These geodesics are shown (numerically) in Appendix \ref{subsec:numerics} (see Figure \ref{fig:geodesics}). These computations suggest that there is a critical value of $\gamma$ at which this transition takes place.

\medskip

We now introduce some notations that will be used in the proofs of Theorem \ref{thm:DAH} below.
We will use $x$ and $y$ as the coordinates in the phase plane (so $\overrightarrow{\psi}=(x,y)$). Then, the potential $h(\overrightarrow{\psi})$ is given by (using Lemma \ref{lem:legendre}(i)),
\begin{align}\label{eq:h-in-xy}
h(x,y) = \frac12(x^2+y^2) - \left(\frac{m-1}{m}\max\left\{v_0,v_1(x,y),v_2(x,y)\right\}\right)^{\frac{m}{m-1}}
\end{align}
where
\begin{align*}
v_0 (x,y): = 0,\qquad v_1(x,y) := \sqrt{\alpha}x - a_1,\quad\text{and}\quad v_2(x,y) := \frac{\gamma}{\sqrt{\alpha}}x + \sqrt{\beta-\frac{\gamma^2}{\alpha}}y - a_2.
\end{align*}
We divide the phase plane into the three regions 
\begin{align*}
V_i:=\{(x,y)\,|\,v_i(x,y) =\max\{v_0(x,y),v_1(x,y),v_2(x,y)\} \}. 
\end{align*}
These regions can be computed explicitly and are plotted in  Figure \ref{fig:coordinate-settings}.
Note that the separating  lines
\begin{align*}
L_{ij} := \{(x,y)\,|\,v_i (x,y)= v_j(x,y) \}
\end{align*}
meet at one point $M$.

\begin{figure}[htbp]
	\begin{center}
            \includegraphics[height=7cm]{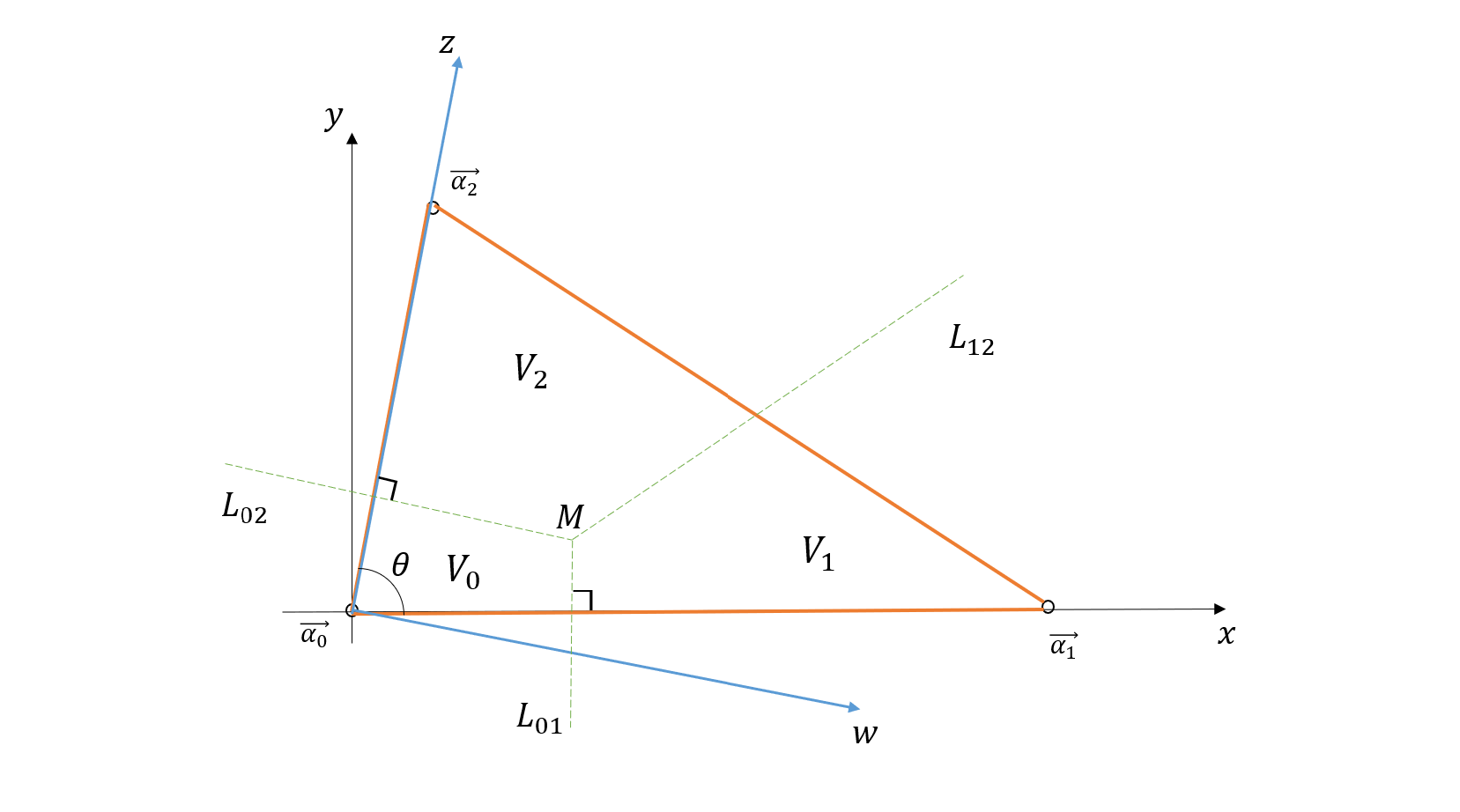}
            \captionsetup{width=0.8\textwidth} 
		\vskip 0pt
		\caption{The cells $V_i$, the bisection lines $L_{ij}$, and the $zw$-coordinate system.}
        \label{fig:coordinate-settings}
	\end{center}
\end{figure}

In the region $V_0 \cup V_1$, we can write
\begin{align}\label{eq:decomposition-xy}
h(x,y) = h_1(x) + \frac12y^2\text{ in }V_0 \cup V_1,\quad \text{ where }h_1(x) := \frac12x^2 - \left(\frac{m-1}{m}\max\{\sqrt{\alpha}x-a_1,0\}\right)^{\frac{m}{m-1}}.
\end{align}
This decomposition will be useful later on and 
in order to write something similar in the region $V_0 \cup V_2$, we need to use a different set of coordinates:
The $(z,w)$-coordinate system (the blue axis on Figure \ref{fig:coordinate-settings}) is defined from $(x,y)$ by rotation: $z = x \cos \theta + y \sin \theta$ and $w = x \sin \theta - y \cos \theta$ where 
\begin{equation}\label{eq:theta}
\cos\theta = \frac{\gamma}{\sqrt{\alpha\beta}}, \quad   \sin\theta = \sqrt{1-\frac{\gamma^2}{\alpha\beta}}.
\end{equation}
The arguments $v_0, v_1,v_2$ in \eqref{eq:h-in-xy} can then be written as functions of $(z,w)$ as follows:
\begin{align*}
v_0(z,w)=0,\quad v_1(z,w) = \frac{\gamma}{\sqrt{\beta}}z + \sqrt{\alpha - \frac{\gamma^2}{\beta}}w - a_1,\quad\text{and}\quad v_2(z,w) = \sqrt{\beta}z - a_2.
\end{align*}
In the region $V_0 \cup V_2$, we can then write (to be compared with \eqref{eq:decomposition-xy}):
\begin{align}\label{eq:decomposition-zw}
h(z,w) = h_2(z) + \frac12w^2\text{ in }V_0 \cup V_2,\quad  \text{ where }h_2(z) := \frac12z^2 - \left(\frac{m-1}{m}\max\{\sqrt{\beta}z-a_2,0\}\right)^{\frac{m}{m-1}}.
\end{align}

Finally, since our goal is to identify the geodesics, we denote
\begin{align*}
I(\gamma) : = \int_0^{\text{Length}(\gamma)}\sqrt{h(\gamma(s))}\,ds\quad\text{for any piecewise $C^1$ curve $\gamma$},
\end{align*}
where $s$ is the arclength parameter of $\gamma$ (with respect to the Euclidean metric).

\begin{proof}[Proof of Theorem \ref{thm:DAH}-(ii)]
When $\gamma =0$, we have $\sigma_{12} = \sigma_{10} + \sigma_{02}$ (see Proposition \ref{prop:sorting}).
In particular we have $\sigma_{10} = \sigma_{12} - \sigma_{02} <\sigma_{12} + \sigma_{02} $. The symmetry $\sigma_{ij}=\sigma_{ji}$ and the continuity of the distance $d$ with respect to $\gamma$ (which follows from the growth of $h$ at infinity) imply
\begin{align*}
\sigma_{01} < \sigma_{02} + \sigma_{21} \quad\text{for small }\gamma>0.
\end{align*}
We can show similarly that 
$$  \sigma_{02} < \sigma_{01} + \sigma_{12}  \quad\text{for small }\gamma>0.
$$

We will now show the following inequality which implies the remaining strict triangle  inequality: for small $\gamma>0$ we have
\begin{align}\label{eq:conclusion-strict-ineq1}
\sigma_{12}  \leq \sigma_{10} + \sigma_{02}- C_{\gamma}
\end{align}
with $C_{\gamma} := \frac{\gamma^2}{2\sqrt2\alpha\beta}\min\left\{\frac{a_1}{2\sqrt{\alpha}},\frac{a_2}{2\sqrt{\beta}}\right\}^2>0$.

\medskip

For any $\delta>0$, there exists a curve $\gamma_{ij}$ from $\overrightarrow{\alpha_i}$ to $\overrightarrow{\alpha_j}$ for $(i,j) = (2,0)$ and $(0,1)$, respectively, satisfying
\begin{align}\label{eq:almost-geodesic-11}
I(\gamma_{20}) \leq \sigma_{20} + \delta \quad\text{and}\quad I(\gamma_{01}) \leq \sigma_{01} + \delta.
\end{align}
By the decompositions \eqref{eq:h-in-xy} and \eqref{eq:decomposition-zw}, we can assume without loss of generality that both the $y$-component and the $w$-component of $\gamma_{20}$ and $\gamma_{01}$ are nonnegative. Otherwise, we can replace the curve $\gamma_{20}(t) = (x(t),y(t))$ by $(x(t),\max\{y(t),0\})$ (or, $(z(t),w(t))$ by $(z(t),\max\{w(t),0\})$) with cheaper cost. We let $\gamma$ be the concatenation of the curves $\gamma_{20}$ and $\gamma_{01}$ so that $\gamma$ travels from $\overrightarrow{\alpha_2}$ through $\overrightarrow{\alpha_0}$ to $\overrightarrow{\alpha_1}$.

Next, we note that the region $V_0$ is included in the domain where $x<\frac{a_1}{\sqrt \alpha}$ (since $v_1(x,y)<0$ there) and $z<\frac{a_2}{\sqrt\beta}$ (which gives $v_2(z,w)<0$). So the disk $D := \{x^2+y^2\leq\eta^2 \}$ with $\eta := \min\left\{\frac{a_1}{2\sqrt{\alpha}},\frac{a_2}{2\sqrt{\beta}}\right\}$ is contained in $V_0$.
We now consider the last hitting point, $P_0$, of $\gamma_{20}$ on $\partial D$. Similarly, consider the first hitting point, $P_1$, of $\gamma_{01}$ on $\partial D$. With these choices, the subarcs $\gamma_{P_0O}$ (of $\gamma_{20}$) and $\gamma_{OP_1}$ (of $\gamma_{01}$) stay within the region $V_0$ where \eqref{eq:h-in-xy} applies. 
\medskip

\begin{figure}[htbp]
	\begin{center}
            \includegraphics[height=7cm]{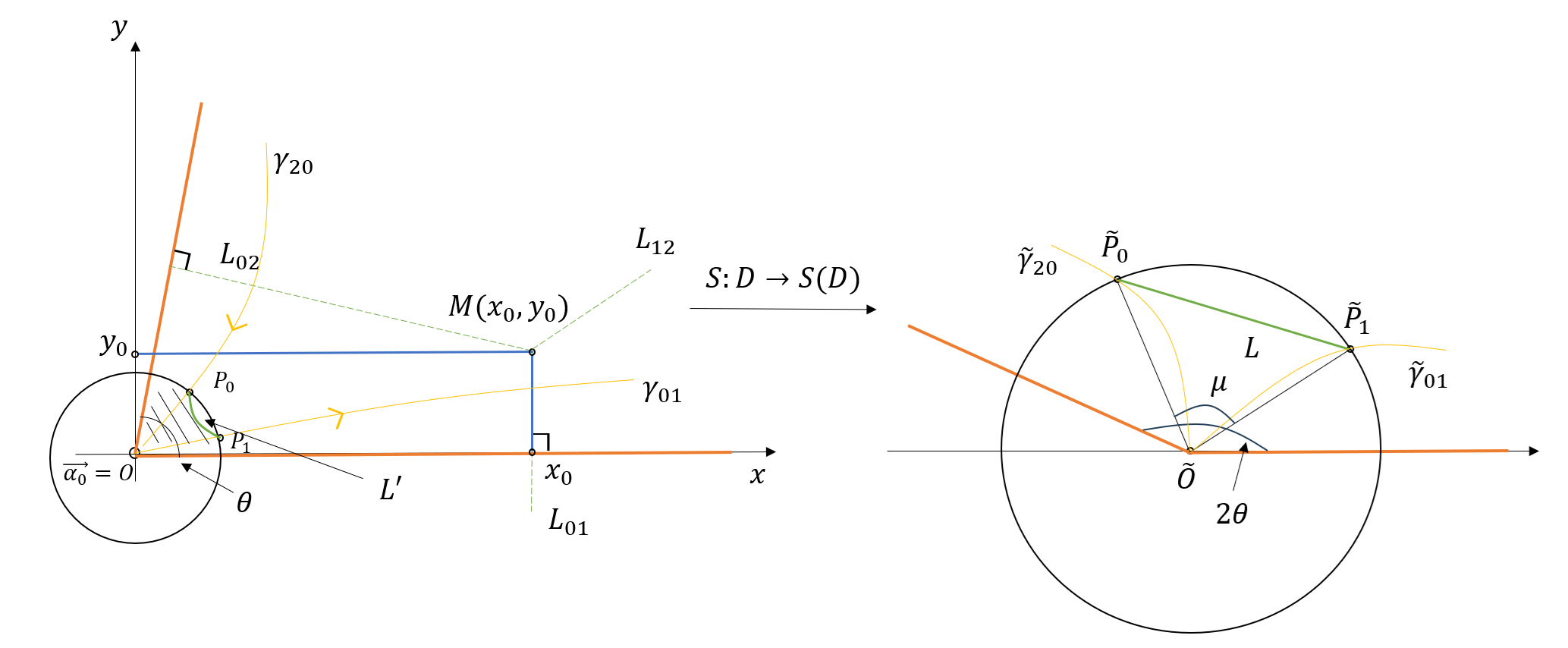}
            \captionsetup{width=0.9\textwidth} 
		\vskip 0pt
		\caption{The domain $\{x^2+y^2\leq\eta^2\text{ and }y,w\geq0\}$ (the shaded region) and the square map $S$.}
        \label{fig:square-transformation}
	\end{center}
\end{figure}
We now identify the $(x,y)$-plane with the complex plane $\mathbb C$, denoting $\xi = x+iy$.
Since $h(x,y) = \frac12(x^2+y^2)$ in $D$, we can write $h(\xi) = \frac 1 2 |\xi|^2$ and so  the length of a curve $\xi(t)$ joining $\xi(0)=\xi_1$ to $\xi(1)=\xi_2$ is 
$$
I(\xi) = \int_0^1 \sqrt{\frac 1 2 |\xi(t)|^2} |\xi'(t)|\, dt =\frac{1}{\sqrt 2} \int_0^1 |\xi(t)| |\xi'(t)|\, dt = \frac 1 {2\sqrt2} \int_0^1 |(\xi^2)'(t)| \, dt 
$$
The complex transformation $\tilde\xi=S(\xi):=\xi^2$ is thus an isometry from $(D,\sqrt{h(\xi)}|d\xi|)$ to $(S(D),\frac{1}{2\sqrt2}|d\tilde{\xi}|)$.

The points $\tilde{P_0}=S(P_0)$ and $\tilde {P_1}=S(P_1)$ lie on the circle of radius $\eta^2$, so their $\frac{1}{2\sqrt2}|d\tilde \xi|$-distance to $O$ is $r = \frac{\eta^2}{2\sqrt2}$.
Let $L$ denote the straight line joining  $\tilde{P_0}$ and $\tilde{P_1}$ in $S(D)$. This is the geodesic since the set $S(D)$ is convex ($\theta \leq \pi/2$) and it has $\frac{1}{2\sqrt2}|d\tilde \xi|$ length $2r \sin(\mu/2)$.
Let ${L}'$ denote the pullback of $L$ under the transformation $S$. 
Since $I(\gamma_{P_0O}) \geq r$ and $I(\gamma_{OP_1}) \geq r$, we deduce
\begin{align}\label{eq:save-cost1}
I(\gamma_{P_0O}) + I(\gamma_{OP_1}) - I(L') \geq 
 2r(1-\sin (\mu/2) )\geq 2r(1-\sin(\theta)) \geq r \frac{\gamma^2}{\alpha\beta}=
 \frac{\eta^2\gamma^2}{2\sqrt2\alpha\beta}=C_{\gamma}
\end{align}
where we recall that $\theta$ is defined by \eqref{eq:theta}. 

\medskip

Consider the alternative curve $\gamma_1$ defined by the concatenation of the three curves $\gamma_{\overrightarrow{\alpha_2}P_0}$, $L'$, and $\gamma_{P_1\overrightarrow{\alpha_1}}$, where the subarcs $\gamma_{\overrightarrow{\alpha_2}P_0}$ and $\gamma_{P_1\overrightarrow{\alpha_1}}$ are from $\gamma_{20}$ and $\gamma_{01}$, respectively. Then, \eqref{eq:almost-geodesic-11} and \eqref{eq:save-cost1} imply
\begin{align*}
\sigma_{12} \leq I(\gamma_1) \leq I(\gamma) - C_{\gamma} &\leq I(\gamma_{20}) + I(\gamma_{01}) - C_{\gamma} \leq \sigma_{20} + \sigma_{01} + 2\delta - C_{\gamma} 
\end{align*}
and \eqref{eq:conclusion-strict-ineq1} follows by taking $\delta\to0$.
\end{proof}

\begin{proof}[Proof of Theorem \ref{thm:DAH}(iii)]
In order to show that $\sigma_{01} = \sigma_{02} + \sigma_{21}$, it suffices to show that for any curve $\gamma$ joining $\overrightarrow{\alpha_0}$ to $\overrightarrow{\alpha_1}$, there exists a curve $\widetilde{\gamma}$ joining $\overrightarrow{\alpha_0}$ through $\overrightarrow{\alpha_2}$ to $\overrightarrow{\alpha_1}$ satisfying $ I(\widetilde{\gamma})\leq I(\gamma)$. As before, we can assume that the $w$-component of $\gamma$ is nonnegative. Indeed, the curve $(z(t),\max\{w(t),0\})$ has a shorter length than $(z(t),w(t))$ since
\begin{align*}
h(z,w) \geq h(z,\max\{w,0\}) \quad \text{for }(z,w)\in V_i \quad (\text{for any }i=0,1,2).
\end{align*}

\medskip

The proof Theorem \ref{thm:DAH}(iii) now relies on the following claim:
When  $\gamma\geq\beta$, we have:
\begin{align}\label{eq:decomposition-gamma>beta}
h(z,w) \geq h(z,0) + h(\theta_2,w) \qquad\text{whenever}\quad z\leq\theta_2,\,w\geq0.
\end{align}
Postponing the proof of \eqref{eq:decomposition-gamma>beta} until the end of this section, we conclude the proof of Theorem \ref{thm:DAH}(iii) as follows:
\begin{figure}[htbp]
	\begin{center}
            \includegraphics[height=7cm]{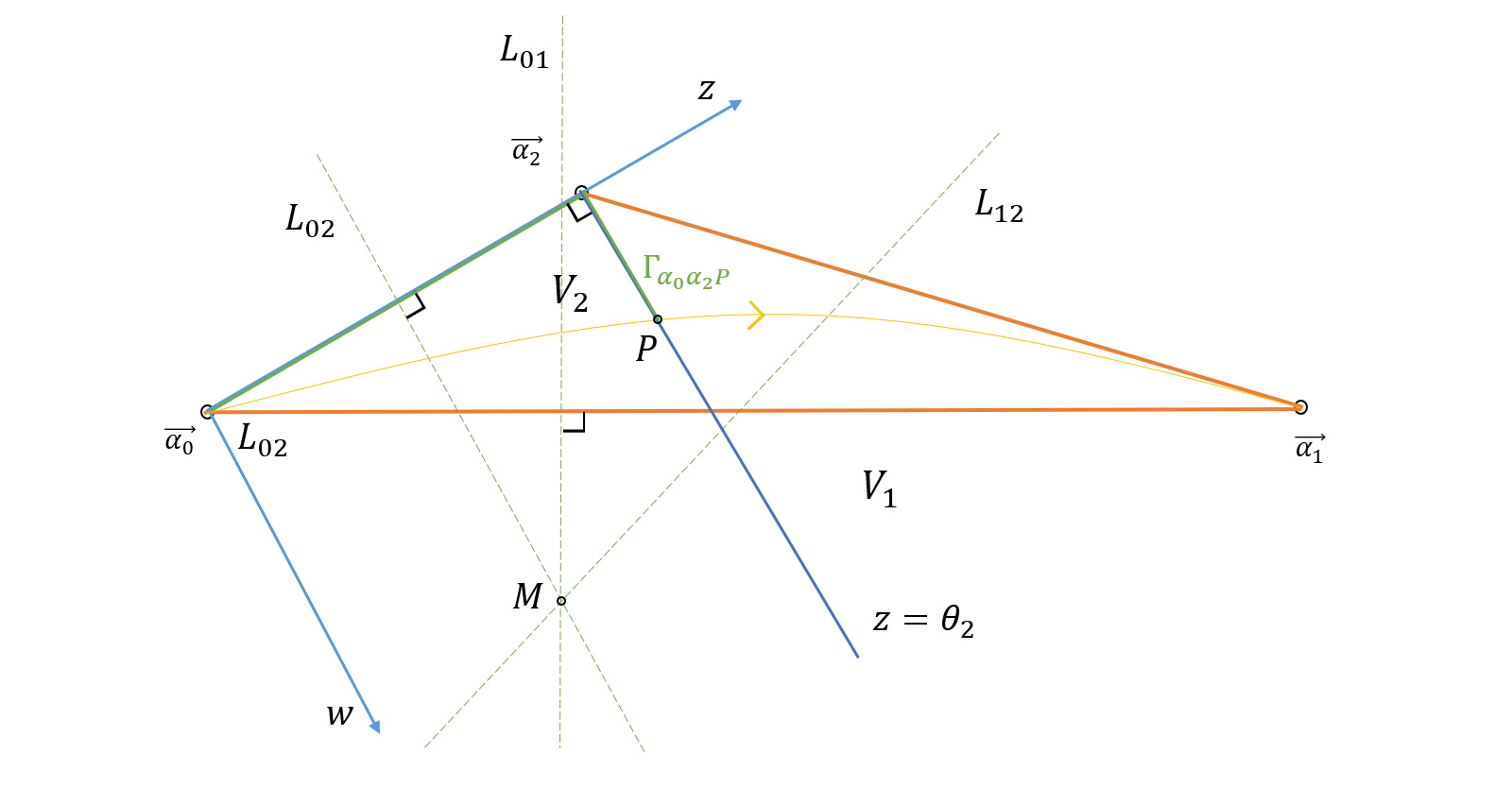}
            \captionsetup{width=0.8\textwidth} 
		\vskip 0pt
		\caption{The curve $\Gamma_{\alpha_0\alpha_2P}$ is cheaper than $\gamma_{\alpha_0P}$.} 
        \label{fig:full-engulfment}
	\end{center}
\end{figure}
Given a curve $\gamma$ from $\overrightarrow{\alpha_0}$ to $\overrightarrow{\alpha_1}$ that does not pass through $\overrightarrow{\alpha_2}$ we denote by $P$ the first point of $\gamma$ on the line $\{z=\theta_2\}$, and let $\gamma_{\alpha_0P}$ (resp. $\gamma_{P\alpha_1}$) be the subarc of $\gamma$ from $\overrightarrow{\alpha_0}$ to $P$ (resp. from $P$ to $\overrightarrow{\alpha_1}$) - see Figure \ref{fig:full-engulfment}.  Then, $\gamma_{\alpha_0P}$ is contained in the region $\{z\leq\theta_2,\,w\geq0\}$ so that \eqref{eq:decomposition-gamma>beta} holds and implies:
\begin{align*}
I(\gamma_{\alpha_0P}) &= \int_0^1\sqrt{h(z(t),w(t))}\sqrt{\dot{z}(t)^2 + \dot{w}(t)^2}\,dt\notag \\
&\geq \int_0^1 \sqrt{h(z(t),0)}|\dot{z}(t)|\,dt + \sqrt{h(\theta_2,w(t))}|\dot{w}(t)|\,dt\notag \\
&\geq I(\Gamma_{\alpha_0\alpha_2}) + I(\Gamma_{\alpha_2P}) \notag
\end{align*}
Here, $\Gamma_{\alpha_0\alpha_2}$ is the straight line from $\overrightarrow{\alpha_0}$ to $\overrightarrow{\alpha_2}$, $\Gamma_{\alpha_2P}$ is the straight line from $\overrightarrow{\alpha_2}$ to $P$.
This inequality shows that the curve $\widetilde{\gamma}$ given by the concatenation of $\Gamma_{\alpha_0\alpha_2}$ and $\Gamma_{\alpha_2P}$ and $\gamma_{P\alpha_1}$ 
has a smaller distance than $\gamma$ and passes through $\overrightarrow{\alpha_2}$.

\medskip

It thus only remains to show \eqref{eq:decomposition-gamma>beta} in order to complete the proof. In order to do this, we define
$$
H(z, w) := h(z, w) - h(z, 0) - h(\theta_2, w).
$$
Since 
\begin{align*}
h(z,w) 
& = \frac12(z^2+w^2) - \left(\frac{m-1}{m}\max\left\{v_0,v_1(z,w),v_2(z,w)\right\}\right)^{\frac{m}{m-1}}\\
& = \frac12(z^2+w^2) - G\left(\max\left\{v_0,v_1(z,w),v_2(z,w)\right\}\right) ,
\end{align*}
with $G(u) := \left(\frac{m-1}{m}u\right)^{\frac{m}{m-1}}$, 
we can write:
\begin{align*}
H(z, w) = G(\max \{v_0, v_1(\theta_2, w), v_2(\theta_2) \})\, +\, & G(\max \{v_0,v_1(z,0), v_2(z) \})\\
& - G(\max\{v_0,v_1(z, w),v_2(z)\}) - \frac{1}{2} \theta_2^2.
\end{align*}

Since $v_1(z,w) \geq v_1(z,0)$ when $w>0$, we have $H(z,w) =   G(\max \{v_0, v_1(\theta_2, w), v_2(\theta_2) \}) - \frac{1}{2} \theta_2^2$ whenever $v_2(z) \geq v_1(z,w)$ and so
\begin{align*}
\frac{\partial H}{\partial z}(z,w) =0 \quad\text{when }v_2(z) > v_1(z,w).
\end{align*}
In the set where $v_1(z,0)<v_2(z)<v_1(z,w)$, we have 
$$
\frac{\partial H}{\partial z}(z,w) = \pa_z  G( \max\{0,  v_2(z) \})- \pa_z G(\max\{0,v_1(z, w) \})
$$
so  using the definition of $v_i$ and the assumption $\gamma\geq\beta$, we see that
\begin{align*}
\frac{\partial H}{\partial z}(z,w) &=  \sqrt{\beta} \, G'(\max\{v_2(z), 0\}) - \frac{\gamma}{\sqrt{\beta}} \, G'(\max \{v_1(z, w),0\})
\\
&\leqslant \left( \sqrt{\beta} - \frac{\gamma}{\sqrt{\beta}} \right) G'(\max \{v_1(z, w),0\}) \leqslant 0 \quad\text{when }v_1(z,0)<v_2(z)<v_1(z,w).
\end{align*}
Finally, we have similarly
\begin{align*}
\frac{\partial H}{\partial z}(z,w) &=  \frac{\gamma}{\sqrt{\beta}} \, G'(\max\{v_1(z,0), 0\}) - \frac{\gamma}{\sqrt{\beta}} \, G'(\max \{v_1(z, w),0\})\leq0\quad\text{when }v_2(z)<v_1(z,0).
\end{align*}
We conclude that $\frac{\partial H}{\partial z} (x,w)\leq 0$ for all $w\geq0$. Since $H(\theta_2,w)=0$, it follows that $H(z,w)\geq0$ in the region $\{(z,\theta)\,|\, z\leq \theta_2,\,w\geq0\}$. This completes the proof of \eqref{eq:decomposition-gamma>beta} and of Theorem \ref{thm:DAH}-(iii).
\end{proof}

\section{Convergence of the gradient flow}

We keep $\sigma=1$ in this section as well for simplicity and we denote by $\overrightarrow{\rho}^{\varepsilon}(t,x)=(\rho_1^{\varepsilon},\cdots,\rho_N^{\varepsilon})^T$ a solution to \eqref{eq:epsilon-problem-1}-\eqref{eq:epsilon-problem-2} with   initial data $\{\overrightarrow{\rho}^{\varepsilon}_{in}\}_{\varepsilon\in(0,1)}$ satisfying
\begin{align*}
\sup_{\varepsilon\in(0,1)}\cE^{\varepsilon}(\overrightarrow{\rho}^{\varepsilon}_{in})\leq M\qquad\text{for some constant $M>0$.}
\end{align*}
We  recall that $j^\eps_i$
denotes the flux associated to the $i$-th population  and that
the energy inequality (Theorem \ref{thm:well-posedness}-(v)) yields:
\begin{align}\label{eq:dissipation}
\cE^{\varepsilon}(\overrightarrow{\rho}^{\varepsilon}(t))+\sum_{i=1}^N\frac12\int_0^t\int_{\Omega}\frac{|j_i^{\varepsilon}|^2}{\rho^{\varepsilon}_i}\,dxdt\leq \cE^{\varepsilon}(\overrightarrow{\rho}^{\varepsilon}_{in}) \leq M.
\end{align}

\subsection{Phase separation: Theorem \ref{thm:phase-separation-convergence}(i)-(ii)}\label{subsec:phase-separation}

We begin by recording uniform estimates that are needed in the proof:

\begin{lemma}\label{lem:estimates}
The following hold, where $C_M>0$ is a suitable constant that is large depending only on $M$.
\begin{enumerate}[label=(\roman*)]
\item $\mathcal{E}^\varepsilon(\overrightarrow{\rho}^{\varepsilon}(t)) \le M$ for all $t\geq0$.

\item $\|\rho_i^\varepsilon\|_{L^\infty(0, \infty; L^m(\Omega))} \le C_M$ for $i=1,\cdots,N$.

\item $\|j_i^\varepsilon\|_{L^2\left(0, \infty; L^{\frac{2m}{m+1}}(\Omega)\right)} \le C_M$ for $i=1,\cdots,N$.

\item $\|\rho_i^\varepsilon(t) - \rho_i^\varepsilon(s)\|_{W^{-1, \frac{2m}{m+1}}(\Omega)} \le C_M \sqrt{t-s},$ for all $0 \le s \le t.$
\end{enumerate}
\end{lemma}
\begin{proof}
\begin{enumerate}[label=(\roman*)]
\item This is a direct consequence of \eqref{eq:dissipation}.

\item Since $f=f_m$ with $m>2$, there are constants $c\in(0,1),\,C>0$ such that
\begin{align*}
W(\overrightarrow{\rho})\geq c|\overrightarrow{\rho}|^m\qquad\text{whenever }|\overrightarrow{\rho}|\geq C\text{ for }\overrightarrow{\rho}\in\R^N_{\geq0}.
\end{align*}
The bound (ii) then follows from \eqref{eq:dissipation}.

\item By Cauchy-Schwarz's inequality, we have
\begin{align*}
\| j_i^\varepsilon \|_{L^{\frac{2m}{m+1}}(\Omega)} \le \| (\rho_i^\varepsilon)^{1/2} \|_{L^{2m}(\Omega)} \left\| \frac{j_i^\varepsilon}{(\rho_i^\varepsilon)^{1/2}} \right\|_{L^2(\Omega)},
\end{align*}
which implies
\begin{align*}
\int_0^T \| j_i^\varepsilon \|_{L^{\frac{2m}{m+1}}(\Omega)}^2 \, dt \le C_M\qquad\text{for each }i=1,\cdots,N.
\end{align*}

\item We first note that for $\psi\in W^{1,\frac{2m}{m-1}}(\Omega)$, we have, by (iii),
\begin{align*}
\left\| \rho_i^\varepsilon |\nabla \psi|^2 \right\|_{L^1(\Omega)}^{1/2} &\le \left\| \rho_i^\varepsilon \right\|_{L^m(\Omega)}^{1/2} \left\| |\nabla \psi|^2 \right\|_{L^{\frac{m}{m-1}}(\Omega)}^{1/2} \le C  \left\| \nabla \psi \right\|_{L^{\frac{2m}{m-1}}(\Omega)}.
\end{align*}
From the continuity equation
\begin{align*}
\partial_t\rho^{\varepsilon}_i+\div(j^{\varepsilon}_i)=0,
\end{align*}
we see that for any $\psi\in W^{1,\frac{2m}{m-1}}(\Omega)$ and $0\leq s\leq t$, it holds that
\begin{align*}
\left| \int_{\Omega} \left( \rho_i^{\varepsilon}(x, t) - \rho_i^{\varepsilon}(x, s) \right) \psi(x) \, dx \right| 
&= \left| \int_{s}^{t} \int_{\Omega} j_i^{\varepsilon} \cdot \nabla \psi \, dx \, dt \right| \\
&\le \left( \int_{s}^{t} \int_{\Omega} \frac{|j_i^{\varepsilon}|^2}{\rho_i^{\varepsilon}} \, dx \, dt \right)^{1/2} \left( \int_{s}^{t} \int_{\Omega} \rho_i^{\varepsilon} |\nabla \psi|^2 \, dx \, dt \right)^{1/2} \\
&\le C \|\nabla \psi\|_{L^{\frac{2m}{m-1}}(\Omega)} \sqrt{t - s},
\end{align*}
which implies the statement (v). Here, \eqref{eq:dissipation} is used in the last inequality.
\end{enumerate}
\end{proof}

Our next proposition, which follows from Lemma \ref{lem:estimates}, implies Theorem \ref{thm:phase-separation-convergence}(i)-(ii):
\begin{proposition}\label{prop:convergence}
Let $\varepsilon\to0$ be a sequence such that $\overrightarrow{\rho}^\eps_{in}$ converges to $\sum_{i=1}^N\rho_{i,in}\overrightarrow{e_i}$ as $\varepsilon\to0$ strongly in $L^1(\Omega)^N$. The following hold.
\begin{enumerate}[label=(\roman*)]
\item There exists a subsequence (still denoted by $\varepsilon$) along which $\rho_i^{\varepsilon}(t)$ converges to $\rho_i(t)$ in $W^{-1,\frac{2m}{m+1}}(\Omega)$ locally uniformly in $t\geq0$, and  $j_i^{\varepsilon}$ converges to $j_i$ weakly in $L^2(0,\infty;L^{\frac{2m}{m+1}}(\Omega))$ for each $i=1,\cdots,N$.

\item For each $i=1,\cdots,N$, there exists a velocity field $v_i\in L^2(\Omega\times(0,\infty),d\rho_i)^d$ such that the continuity equation with the flux $j_i:=\rho_i v_i$ holds
\begin{align*}
\begin{cases}
\partial_t \rho_i + \operatorname{div}(j_i) = 0 & \textit{in } \Omega\times(0,\infty),\\
j_i \cdot n = 0 & \textit{on } \partial \Omega\times(0,\infty), \\
\rho_i(\cdot, 0) = \rho_{i, in} & \textit{in } \Omega.
\end{cases}
\end{align*}

\item Along a further subsequence, $\rho^{\varepsilon}_i(t)$ converges strongly in $L^1(\Omega)$ locally uniformly in $t\geq0$. Moreover, we have
\begin{align*}
\rho_i(t)\in BV(\Omega,\{0,\theta_i\})\qquad\text{for all $t\geq0$ and }i=1,\cdots,N.
\end{align*}
\end{enumerate}
\end{proposition}

\begin{proof}
The statement (i) is a consequence of Lemma \ref{lem:estimates}. We pass to the limit $\varepsilon\to0$ in the continuity equation  Theorem \ref{thm:well-posedness}(ii) to obtain
\begin{align*}
\int_{\Omega} \rho_{i,in}(x)\,\zeta(x,0)\,dx
\;+\; \int_0^{\infty}\int_{\Omega} \rho_i \,\partial_t \zeta
+ j_i \cdot \nabla \zeta \,dxdt
= 0\qquad\text{for any }\zeta\in C_c^{\infty}(\overline{\Omega}\times[0,\infty)).
\end{align*}
The fact that $j$ is of the form $\rho v$ follows from the lower semicontinuity of the following functional (as in \cite{MR-CS10}) with respect to the weak converge of measure of $\mu$ a scalar measure and $F$ a vector measure:
\begin{align*}
\Theta : (\mu, F) \mapsto 
\begin{cases} 
\displaystyle \int_0^T \int_{\Omega} \frac{|F|^2}{\mu} & \text{if } F \ll \mu \text{ a.e. } t \in [0, T]; \\ 
+\infty & \text{otherwise.} 
\end{cases}
\end{align*}
From the energy dissipation \eqref{eq:dissipation}, we have $\Theta(\rho_i^{\varepsilon},j_i^{\varepsilon})\leq 2M$, and so the lower semicontinuity of $\Theta$ implies $\Theta(\rho_i ,j_i )\leq 2M$. 
It follows that $j_i$ is absolutely continuous with respect to $\rho_i$ and that there exists a vector field $v_i\in L^2(\Omega\times(0,\infty),d\rho_i)^d$ satisfying $j_i=\rho_i v_i$.

\medskip

It remains to prove (iii). We recall here the functions
$\varphi_i:=\max\left\{2d(\cdot,\overrightarrow{\alpha_i}),M\right\}$
used in the proof of the liminf in the $\Gamma$-convergence of the energy.
As in Section \ref{subsec:liminf-property}, we write, using  Proposition \ref{prop:MM-functional-Baldo},
\begin{align*}
\cE^{\varepsilon}(\overrightarrow{\rho}^{\varepsilon}) &\geq \frac{1}{\varepsilon}\int_{\Omega}h(\overrightarrow{\psi^{\varepsilon}})\,dx + \varepsilon\int_{\Omega}\frac12|\nabla\overrightarrow{\psi}^{\varepsilon}|^2\,dx\geq\int_{\Omega}\sqrt{2h(\overrightarrow{\psi}^{\varepsilon})}|\nabla\overrightarrow{\psi}^{\varepsilon}|\,dx \geq \int_{\Omega}|\nabla(\varphi_i\circ\overrightarrow{\psi}^{\varepsilon})|\,dx
\end{align*} 
which implies
\begin{align*}
\int_{\Omega}|\nabla(\varphi_i\circ\overrightarrow{\psi}^{\varepsilon})|\,dx\leq\cE^{\varepsilon}(\overrightarrow{\rho}^{\varepsilon}) \leq M.
\end{align*}
Since the function $\varphi_i$ is bounded, we deduce
\begin{align}\label{eq:varphi-BV-bounded}
\|\varphi_i(\overrightarrow{\psi}^\varepsilon)\|_{L^{\infty}(0,\infty;BV(\Omega))}\leq M\qquad\text{for each }i=0,\cdots,N.
\end{align}

\medskip

Given $i=0,\dots,N$, we have $\varphi_i(Q\theta_j \overrightarrow{e_j}) = 2d(\overrightarrow{\alpha_j},\overrightarrow{\alpha_i})$ for $j=0,\dots,N$ so we introduce an affine mapping $L_i:\R^N\to\R$ such that $L_i(\theta_j \overrightarrow{e_j}) = 2d(\overrightarrow{\alpha_j},\overrightarrow{\alpha_i})$.
This mapping is given by 
\begin{align*}
L_i(\overrightarrow{\rho}) := \overrightarrow{R}_i \cdot \overrightarrow{\rho}+ S_i, \qquad \mbox{ where }
\begin{cases}
\overrightarrow{R}_i := 2 
\begin{pmatrix}
\theta_1^{-1} (d(\overrightarrow{\alpha_1}, \overrightarrow{\alpha_i}) - d(\overrightarrow{\alpha_0}, \overrightarrow{\alpha_i})) \\
\vdots \\
\theta_N^{-1} (d(\overrightarrow{\alpha_N}, \overrightarrow{\alpha_i}) - d(\overrightarrow{\alpha_0}, \overrightarrow{\alpha_i}))
\end{pmatrix}, \\[20pt]
S_i := 2 d(\overrightarrow{\alpha_0}, \overrightarrow{\alpha_i}).
\end{cases}
\end{align*}
We then claim that (for all $i=0,\cdots,N$) we have
\begin{align}\label{eq:auxiliary-ineq}
\left\| \varphi_i(\overrightarrow{\psi}^\varepsilon) - L_i( \overrightarrow{\rho}^\varepsilon ) \right\|_{L^\infty(0, \infty; L^2(\Omega))} \longrightarrow 0 \quad \text{as } \varepsilon \to 0 .
\end{align}

To prove \eqref{eq:auxiliary-ineq}, we first note that 
since the function $\varphi_i$ is bounded and Lipschitz, there exists a constant $C>0$ such that
\begin{align}\label{eq:auxiliary-ineq-1}
&\| \varphi_i(\overrightarrow{\psi}^\varepsilon) - \varphi_i(Q\overrightarrow{\rho}^\varepsilon) \|_{L^2(\Omega)} \le C \|  \overrightarrow{\psi}^\varepsilon - Q \overrightarrow{\rho}^\varepsilon \|_{L^2(\Omega)} \le C \varepsilon^{1/2}\to0\quad\text{as }\varepsilon\to0.
\end{align}
(using the energy inequality \eqref{eq:dissipation} and the definition \eqref{eq:energy-psi} of $\cE^\eps$).

Next, we note that the definitions of $L_i$ imply that $\varphi_i ( Q \overrightarrow{\rho} ) - L_i( \overrightarrow{\rho})
$ vanishes when $\overrightarrow{\rho}=\overrightarrow{0},\theta_1\overrightarrow{e_1},\cdots,\theta_N\overrightarrow{e_N}$ (which are the zeroes of $W$). Therefore, for any $\delta>0$, denoting by $V_\delta$ the $\delta$-neighborhood of $\{\overrightarrow{\alpha_0},\cdots,\overrightarrow{\alpha_N}\}$, we have
\begin{align*}
\left| \varphi_i ( Q \overrightarrow{\rho}^\varepsilon ) - L_i(  \overrightarrow{\rho}^\varepsilon  ) \right| \leq C \delta^2\qquad\text{if }\overrightarrow{\rho}^\varepsilon\in Q^{-1}V_{\delta}.
\end{align*}
Furthermore, since $W(\overrightarrow{\rho})>0$ for $\overrightarrow{\rho} \notin  Q^{-1}V_{\delta}$, there exists a constant $C_{\delta}>0$ such that
\begin{align*}
\left| \varphi_i ( Q \overrightarrow{\rho}^\varepsilon ) - L_i(\overrightarrow{\rho}^\varepsilon  ) \right| \leq C_{\delta} W( \overrightarrow{\rho}^\varepsilon )\qquad\text{if }\overrightarrow{\rho}^\varepsilon\notin Q^{-1}V_{\delta}.
\end{align*}
Consequently, we can write (using the energy inequality \eqref{eq:dissipation}):
\begin{align*}
\int \left| \varphi_i ( Q \overrightarrow{\rho}^{\varepsilon} ) - L_i (   \overrightarrow{\rho}^{\varepsilon}   ) \right|^2 dx \quad &\leq \int_{\{ \overrightarrow{\rho}^{\varepsilon} \in Q^{-1} V_{\delta} \}} C  \delta^2 dx + C_{\delta} \int_{\{ \overrightarrow{\rho}^{\varepsilon} \notin Q^{-1} V_{\delta} \}} W(\overrightarrow{\rho}^{\varepsilon}) dx \\
& \leq C |\Omega| \delta^2 + C_{\delta} \varepsilon.
\end{align*}
We conclude
\begin{align}\label{eq:auxiliary-ineq-2}
\limsup_{\varepsilon \to 0} \left\| \varphi_i ( Q \overrightarrow{\rho}^{\varepsilon} ) - L_i(   \overrightarrow{\rho}^{\varepsilon}  ) \right\|_{L^{\infty}(0, \infty; L^2(\Omega))} = 0\qquad\text{for each }i=0,\cdots,N.
\end{align}
The claim \eqref{eq:auxiliary-ineq} then follows from \eqref{eq:auxiliary-ineq-1} and \eqref{eq:auxiliary-ineq-2}.

\medskip

Before we proceed further, we recall the following Lions-Aubin compactness type lemma (see \cite[Lemma B.1]{KMW24}): 
\begin{lemma}\label{lem:LA}
Let $u_n$ be a sequence bounded in $L^{\infty}(0,T;BV(\Omega))$ and assume that $u_n$ converges to $u$ strongly in $L^{\infty}(0,T;W^{-1,\frac{2m}{m+1}}(\Omega))$. Then,
\begin{align*}
\|u_n-u\|_{L^{\infty}(0,T;L^1(\Omega))}\to 0.
\end{align*}
\end{lemma}
We note that Proposition \ref{prop:convergence}(i) and \eqref{eq:auxiliary-ineq} imply
\begin{align*}
\varphi_i ( \overrightarrow{\psi}^\varepsilon ) \longrightarrow  L_i(  \overrightarrow{\rho} )\quad \text{strongly in} \quad L^\infty \left( 0, T; W^{-1, \frac{2m}{m+1}}(\Omega) \right),
\end{align*}
which together with \eqref{eq:varphi-BV-bounded} and Lemma \ref{lem:LA} yields:
\begin{align*}
\varphi_i ( \overrightarrow{\psi}^\varepsilon ) \longrightarrow L_i(  \overrightarrow{\rho} )\quad \text{strongly in} \quad L^\infty \left( 0, T;L^1(\Omega) \right).
\end{align*}
Using \eqref{eq:auxiliary-ineq-1}, we also get:
\begin{align}\label{eq:convergence-linear-form}
\varphi_i(Q\overrightarrow{\rho}^{\varepsilon}) \longrightarrow L_i(  \overrightarrow{\rho} ) \quad \text{strongly in} \quad L^\infty \left( 0, T;L^1(\Omega) \right).
\end{align}

\medskip

Next, we claim that  $L_i(\overrightarrow{\rho})$ takes values in $\{\varphi_i(\overrightarrow{\alpha}_j)\}_{j=0..N}$. Indeed, for an a.e. $t>0$ that is fixed, and for $\delta\in(0,1)$ small, if $\cC_{\delta}$ denotes the set
$$
\cC_{\delta} := \{x\in\Omega\,|\,\mathrm{dist}(Q\overrightarrow{\rho}^{\varepsilon}(x),\cA) \geq \delta\}\quad\text{where}\quad\cA:=\{\overrightarrow{\alpha_0},\cdots,\overrightarrow{\alpha_N}\},
$$
then, as $W(\overrightarrow{\rho}^{\varepsilon}(x)) \geq c_{\delta}$ on $\cC_{\delta}$ for some constant $c_{\delta}>0$, it holds that
\begin{align*}
|\cC_{\delta}| \leq c_{\delta}^{-1}\int_{\Omega} W(\overrightarrow{\rho}^{\varepsilon}(x))\,dx \leq C_{\delta}\varepsilon
\end{align*}
for another constant $C_{\delta}>0$. This implies that the distance function $\mathrm{dist}(Q\overrightarrow{\rho}^{\varepsilon},\cA)$ converges to zero in measure, and consequently, up to a further subsequence of $\varepsilon\to0$, we get
\begin{align*}
\mathrm{dist}(Q\overrightarrow{\rho}^{\varepsilon},\cA)\to0\quad\text{a.e. on }\Omega.
\end{align*}
We deduce that $L_i(\overrightarrow{\rho}(t,x))\subset \{\varphi_i(\overrightarrow{\alpha}_j)\}_{j=0}^N$ a.e. $t>0$ and $x\in \Omega$.
We now set
$$
E_i(t) := \{x\in\Omega\,|\,L_i(\overrightarrow{\rho}(t))=0\}\quad\text{for all}\quad i=0,\cdots,N.
$$
We see that for a.e. $t>0$,  $\{E_i(t)\}_{i=0}^N$ is a partition of $\Omega$ (up to a set of measure zero).

We now claim that
\begin{align}\label{eq:L1-convergence}
\overrightarrow{\rho}^{\varepsilon} \longrightarrow  \sum_{i=0}^N\theta_i\chi_{E_i}\overrightarrow{e_i} \quad \text{strongly in} \quad L^\infty \left( 0, T;L^1(\Omega)^N \right).
\end{align}
For this, it suffices to show
\begin{align}\label{eq:convergence-linear-form-2}
\overrightarrow{\rho}^{\varepsilon} \longrightarrow  \theta_i\overrightarrow{e_i} \quad \text{strongly in} \quad L^\infty \left( (0, T),dt;L^1(E_i(t)) \right)
\end{align}
for each $i=0,\cdots,N$.

Let $\mathcal{N}(Q):=\{\overrightarrow{\mathbf{x}}\in\R^N\,|\,Q\overrightarrow{\mathbf{x}}=\overrightarrow{0}\}$ denote the null space of $Q$. For any $\delta\in(0,1)$ small enough, there exists a constant $c_{\delta}\in(0,\delta)$ such that
\begin{align*}
|\varphi_i(Q\overrightarrow{\rho}^{\varepsilon}(x))|\geq c_{\delta}\qquad\text{on }A_1:=\{x\in E_i(t)\,|\,\mathrm{dist}(\overrightarrow{\rho}^{\varepsilon}(x),\mathcal{N}(Q)+\theta_i\overrightarrow{e_i})\geq\delta\}
\end{align*}
and that
\begin{align*}
&|\overrightarrow{\rho}^{\varepsilon}(x)-\theta_j\overrightarrow{e_j}|\geq c_{\delta}\quad\text{for any $j=0,\cdots,N$} \\
&\qquad\qquad\qquad\text{on }A_2:=\{x\in E_i(t)\,|\,\mathrm{dist}(\overrightarrow{\rho}^{\varepsilon}(x),\mathcal{N}(Q)+\theta_i\overrightarrow{e_i})\leq\delta\text{ and }|\overrightarrow{\rho}^{\varepsilon}(x)-\theta_i\overrightarrow{e_i}|\geq \delta\}.
\end{align*}
Here, we used Assumptions \ref{assumption:A1}, \ref{assumption:A2}, and Proposition \ref{prop:mixing} in particular for the latter. With $A_3:=\{x\in E_i(t)\,|\,|\overrightarrow{\rho}^{\varepsilon}(x)-\theta_i\overrightarrow{e_i}|\leq \delta\}$, the sets $A_1,A_2,A_3$ cover $E_t(t)$, and therefore,
\begin{align}\label{eq:A-123}
\int_{E_i(t)}|\overrightarrow{\rho}^{\varepsilon}(x)-\theta_i\overrightarrow{e_i}|\,dx &\leq \int_{A_1}|\overrightarrow{\rho}^{\varepsilon}(x)-\theta_i\overrightarrow{e_i}|\,dx + \int_{A_2}|\overrightarrow{\rho}^{\varepsilon}(x)-\theta_i\overrightarrow{e_i}|\,dx + \int_{A_3}|\overrightarrow{\rho}^{\varepsilon}(x)-\theta_i\overrightarrow{e_i}|\,dx.
\end{align}
By definition of the set $A_1$, we have
\begin{align*}
\int_{A_1}|\overrightarrow{\rho}^{\varepsilon}(x)-\theta_i\overrightarrow{e_i}|\,dx & \leq \int_{\{x\in E_i(t)\,|\,|\varphi_i(Q\overrightarrow{\rho}^{\varepsilon}(x))|\geq c_{\delta}\}}|\overrightarrow{\rho}^{\varepsilon}(x)-\theta_i\overrightarrow{e_i}|\,dx \\
& \leq C|\{x\in E_i(t)\,|\,|\varphi_i(Q\overrightarrow{\rho}^{\varepsilon}(x))|\geq c_{\delta}\}|^{\frac{m-1}{m}}\|\overrightarrow{\rho}^{\varepsilon}\|_{L^{\infty}(0,\infty;L^m(\Omega))} \\
& \leq C|\{x\in E_i(t)\,|\,|\varphi_i(Q\overrightarrow{\rho}^{\varepsilon}(x))|\geq c_{\delta}\}|^{\frac{m-1}{m}}.
\end{align*}
Here, H\"older's inequality and Lemma \ref{lem:estimates}(ii) are used in the second-last and last inequality, respectively. By Chebyshev's inequality, we further have
\begin{align*}
|\{x\in E_i(t)\,|\,|\varphi_i(Q\overrightarrow{\rho}^{\varepsilon}(x))|\geq c_{\delta}|\}| \leq c^{-1}_{\delta}\|\varphi_i(Q\overrightarrow{\rho}^{\varepsilon})\|_{L^\infty \left( (0, T),dt;L^1(E_i(t)) \right)}.
\end{align*}
Since \eqref{eq:convergence-linear-form} implies that $\varphi_i(Q\overrightarrow{\rho}^{\varepsilon}) \longrightarrow L_i(  \overrightarrow{\rho} ) = 0 $ in $E_i$,  (by definition of the set $E_i$), we deduce
\begin{align}\label{eq:A-1}
\limsup_{\varepsilon\to0}\|\overrightarrow{\rho}^{\varepsilon}-\theta_i\overrightarrow{e_i}\|_{L^\infty \left( (0, T),dt;L^1(A_1) \right)}=0.
\end{align}
Next, the definition of the set $A_2$ implies
\begin{align*}
\int_{A_2}|\overrightarrow{\rho}^{\varepsilon}(x)-\theta_i\overrightarrow{e_i}|\,dx \leq C_{\delta}\int_{\Omega}W(\overrightarrow{\rho}^{\varepsilon}(x))\,dx
\end{align*}
for some constant $C_{\delta}>0$ since the potential $W$ only vanishes on $\overrightarrow{\alpha_0},\cdots,\overrightarrow{\alpha_N}$. By using \eqref{eq:dissipation}, we can conclude
\begin{align}\label{eq:A-2}
\limsup_{\varepsilon\to0}\|\overrightarrow{\rho}^{\varepsilon}-\theta_i\overrightarrow{e_i}\|_{L^\infty \left( (0, T),dt;L^1(A_2) \right)}=0.
\end{align}
Therefore, combining \eqref{eq:A-123}, \eqref{eq:A-1}, \eqref{eq:A-2}, and the definition of $A_3=\{x\in E_i(t)\,|\,|\overrightarrow{\rho}^{\varepsilon}(x)-\theta_i\overrightarrow{e_i}|\leq \delta\}$ yields
\begin{align*}
\limsup_{\varepsilon\to0}\|\overrightarrow{\rho}^{\varepsilon}-\theta_i\overrightarrow{e_i}\|_{L^\infty \left( (0, T),dt;L^1(E_i(t)) \right)} \leq \limsup_{\varepsilon\to0}\|\overrightarrow{\rho}^{\varepsilon}-\theta_i\overrightarrow{e_i}\|_{L^\infty \left( (0, T),dt;L^1(A_3) \right)} \leq |\Omega|\delta.
\end{align*}
Taking $\delta\to0$ proves \eqref{eq:convergence-linear-form-2} and thus proves \eqref{eq:L1-convergence}.

\medskip

It remains to show that $\rho_i(t):=\overrightarrow{\rho}(t)\cdot\overrightarrow{e_i}=\theta_i\chi_{E_i(t)}\in BV(\Omega,\{0,\theta_i\})$, or equivalently, $E_i(t)$ is a $BV(\Omega)$-set for $i=1,\cdots,N$, and $t>0$. We note, from \eqref{eq:varphi-BV-bounded}, \eqref{eq:auxiliary-ineq-1}, \eqref{eq:L1-convergence}, and the lower semicontinuity of the $BV(\Omega)$-norm with respect to the $L^1(\Omega)$-norm, we have
\begin{align*}
\|\varphi_i(Q\overrightarrow{\rho})\|_{L^{\infty}(0,\infty;BV(\Omega))}\leq M.
\end{align*}
Since the function $\varphi_i(Q\overrightarrow{\rho}(t))$ takes only finitely many values and $E_i(t)=\{x\in\Omega\,|\,\varphi_i(Q\overrightarrow{\rho}(x,t))=0\}$, which can be seen similarly as above from the strong $L^1(\Omega)$-convergence, we conclude that the set $E_i(t)$ is indeed a $BV(\Omega)$-set. 
\end{proof}

\subsection{Convergence of pressure and the evolution law: Theorem \ref{thm:phase-separation-convergence}(iii)-(iv)}\label{subsec:evolution-law}
We define the pressure $p^{\varepsilon}$ by
\begin{align}\label{eq:pressure-def}
p^{\varepsilon}:=\frac{1}{\varepsilon}\left(\rho^{\varepsilon}f'(\rho^{\varepsilon})+\overrightarrow{a}\cdot\overrightarrow{\rho}^{\varepsilon}-(\overrightarrow{\rho}^{\varepsilon})^TA\overrightarrow{\phi}^{\varepsilon}\right)+m^{\varepsilon}(t),
\end{align}
where  $m^{\varepsilon}(t)\in\R$ is chosen so that $\int_{\Omega}p^{\varepsilon}(x,t)\,dx=0$.

With these notations, we have the following proposition:
\begin{proposition}\label{prop:ep-evolution}
It holds that
\begin{align}\label{eq:motion-law-epsilon}
\int_0^T\int_{\Omega} j^\varepsilon \cdot \xi - p^\varepsilon \operatorname{div}(\xi)\, dxdt &= - \frac{1}{\varepsilon} \int_0^T\int_{\Omega} \left( W(\overrightarrow{\rho}^\varepsilon) + \frac{1}{2} |\overrightarrow{\psi}^\varepsilon- Q\overrightarrow{\rho}^\varepsilon |^2 + \frac{1}{2} \varepsilon^2 | \nabla \overrightarrow{\psi}^\varepsilon |^2 \right) \operatorname{div}(\xi) \, dxdt \notag \\
&\qquad\qquad\qquad\qquad\qquad\qquad\qquad + \varepsilon \int_0^T\int_{\Omega} (\nabla\overrightarrow{\psi}^{\varepsilon})^T\nabla\overrightarrow{\psi}^{\varepsilon}:D\xi \, dxdt.
\end{align}
\end{proposition}
\begin{proof}
Theorem \ref{thm:well-posedness}-(iii) gives:
\begin{align}\label{eq:continuity-total}
\int_0^{T}\int_{\Omega} j^{\varepsilon} \cdot \xi \, dxdt=\frac{1}{\varepsilon }\int_0^{T}\int_{\Omega}
\left(\left(\rho^{\varepsilon}f'(\rho^{\varepsilon})-f(\rho^{\varepsilon})\right) \div(\xi)+(\overrightarrow{\rho}^{\varepsilon})^T A(\nabla
\overrightarrow{\phi}^{\varepsilon}\cdot\xi)\right)\,dxdt.
\end{align}
Therefore, using \eqref{eq:continuity-total} and the definition of the pressure \eqref{eq:pressure-def}, we obtain
\begin{align}\label{eq:j-epsilon}
\int_0^T\int_{\Omega} j^\varepsilon \cdot \xi - p^\varepsilon \operatorname{div}(\xi) \, dxdt &= -\frac{1}{\varepsilon} \int_0^T\int_{\Omega} \left( f(\rho^\varepsilon) + \overrightarrow{a} \cdot \overrightarrow{\rho}^\varepsilon - (\overrightarrow{\rho}^{\varepsilon})^TA\overrightarrow{\phi}^{\varepsilon} \right) \operatorname{div}(\xi) \, dxdt \notag \\
&\qquad\qquad\qquad\qquad + \frac{1}{\varepsilon} \int_0^T\int_{\Omega} (\overrightarrow{\rho}^{\varepsilon})^T A(\nabla
\overrightarrow{\phi}^{\varepsilon}\cdot\xi) \, dxdt.
\end{align}

\medskip

In order to complete the proof, we need to show the following identity:
(for any $\xi\in C^{\infty}(\overline{\Omega};\R^d)$ with $\xi\cdot n=0$ on $\partial\Omega$):
\begin{align}\label{eq:computation-result}
&\frac{1}{\varepsilon} \int_{\Omega} \left( f(\rho^\varepsilon) + \overrightarrow{a} \cdot \overrightarrow{\rho}^\varepsilon - (\overrightarrow{\rho}^\varepsilon)^T A \overrightarrow{\phi}^\varepsilon \right) \operatorname{div}(\xi) \, dx - \frac{1}{\varepsilon} \int_{\Omega} (\overrightarrow{\rho}^{\varepsilon})^T A(\nabla
\overrightarrow{\phi}^{\varepsilon}\cdot\xi) \, dx \\
&\quad = \frac{1}{\varepsilon} \int_{\Omega} \left( W(\overrightarrow{\rho}^\varepsilon) + \frac{1}{2} | \overrightarrow{\psi}^\varepsilon-Q\overrightarrow{\rho}^\varepsilon |^2 + \frac{1}{2} \varepsilon^2 |  \nabla \overrightarrow{\psi}^\varepsilon |^2 \right) \operatorname{div}(\xi) \, dx  - \varepsilon \int_{\Omega} (\nabla\overrightarrow{\psi}^{\varepsilon})^T\nabla\overrightarrow{\psi}^{\varepsilon}:\nabla\xi \, dx.\notag
\end{align}
Here, $\nabla\overrightarrow{\psi}^{\varepsilon}=\left((\psi^{\varepsilon}_i)_{x_k}\right)_{1\leq i\leq N,\,1\leq k\leq d}$ is the Jacobian matrix of size $N$ by $d$.

\medskip

The remainder of the proof is devoted to \eqref{eq:computation-result}. We first write:
\begin{align}\label{eq:expansion}
&\frac{1}{\varepsilon} \int_{\Omega} \left( f(\rho^\varepsilon) + \overrightarrow{a} \cdot \overrightarrow{\rho}^\varepsilon - (\overrightarrow{\rho}^\varepsilon)^T A \overrightarrow{\phi}^\varepsilon \right) \operatorname{div}(\xi) \, dx - \frac{1}{\varepsilon} \int_{\Omega} (\overrightarrow{\rho}^{\varepsilon})^T A(\nabla
\overrightarrow{\phi}^{\varepsilon}\cdot\xi) \, dx \notag \\
&\qquad= \frac{1}{\varepsilon} \int_{\Omega} \left( f(\rho^\varepsilon) + \overrightarrow{a} \cdot \overrightarrow{\rho}^\varepsilon - \frac{1}{2} (\overrightarrow{\rho}^\varepsilon)^T A \overrightarrow{\rho}^\varepsilon \right) \operatorname{div}(\xi) \, dx \notag \\
&\qquad\qquad + \frac{1}{\varepsilon} \int_{\Omega} \left( \frac{1}{2} (\overrightarrow{\rho}^\varepsilon)^T A \overrightarrow{\rho}^\varepsilon - (\overrightarrow{\rho}^\varepsilon)^T A \overrightarrow{\phi}^\varepsilon + \frac{1}{2} (\overrightarrow{\phi}^\varepsilon)^T A \overrightarrow{\phi}^\varepsilon \right) \operatorname{div}(\xi) \, dx \notag\\
&\qquad\qquad\qquad - \frac{1}{\varepsilon} \int_{\Omega} \frac{1}{2} (\overrightarrow{\phi}^\varepsilon)^T A \overrightarrow{\phi}^\varepsilon \operatorname{div}(\xi) \, dx - \frac{1}{\varepsilon} \int_{\Omega} (\overrightarrow{\rho}^{\varepsilon})^T A(\nabla
\overrightarrow{\phi}^{\varepsilon}\cdot\xi) \, dx.
\end{align}

We then compute  each term in the right-hand side: First of all,  the definition of  $W$ implies
\begin{align}\label{eq:RHS-term-1}
\frac{1}{\varepsilon} \int_{\Omega} \left( f_m(\rho^\varepsilon) + \overrightarrow{a} \cdot \overrightarrow{\rho}^\varepsilon - \frac{1}{2} (\overrightarrow{\rho}^\varepsilon)^T A \overrightarrow{\rho}^\varepsilon \right) \operatorname{div}(\xi) \, dx=\frac{1}{\varepsilon} \int_{\Omega} W(\overrightarrow{\rho}^{\varepsilon}) \operatorname{div}(\xi) \, dx.
\end{align}
Next, by the decomposition $A=Q^TQ$ and $\overrightarrow{\psi}^{\varepsilon}=Q\overrightarrow{\phi}^{\varepsilon}$, we have
\begin{align}\label{eq:RHS-term-2}
\frac{1}{\varepsilon} \int_{\Omega} \left( \frac{1}{2} (\overrightarrow{\rho}^\varepsilon)^T A \overrightarrow{\rho}^\varepsilon - (\overrightarrow{\rho}^\varepsilon)^T A \overrightarrow{\phi}^\varepsilon + \frac{1}{2} (\overrightarrow{\phi}^\varepsilon)^T A \overrightarrow{\phi}^\varepsilon \right) \operatorname{div}(\xi) \, dx = \frac{1}{\varepsilon}\int_{\Omega}\frac{1}{2} | \overrightarrow{\psi}^\varepsilon-Q\overrightarrow{\rho}^\varepsilon |^2\operatorname{div}(\xi)\,dx.
\end{align}
Finally, the remaining two terms in \eqref{eq:expansion} becomes (after some integration by parts and using the equation \eqref{eq:epsilon-problem-2}):
\begin{align*}
&\int_{\Omega} \frac{1}{2} (\overrightarrow{\phi}^\varepsilon)^T A \overrightarrow{\phi}^\varepsilon \operatorname{div}(\xi) \, dx + \int_{\Omega} (\overrightarrow{\rho}^{\varepsilon})^T A(\nabla
\overrightarrow{\phi}^{\varepsilon}\cdot\xi) \, dx \\
&\qquad = - \int_{\Omega} \frac12\nabla\left( (\overrightarrow{\phi}^\varepsilon)^T A \overrightarrow{\phi}^\varepsilon \right) \cdot\xi \, dx + \int_{\Omega} (\overrightarrow{\rho}^{\varepsilon})^T A(\nabla
\overrightarrow{\phi}^{\varepsilon}\cdot\xi) \, dx \\
&\qquad = - \int_{\Omega} (\overrightarrow{\phi}^\varepsilon)^T A (\nabla\overrightarrow{\phi}^\varepsilon\cdot\xi) \, dx + \int_{\Omega} (\overrightarrow{\rho}^{\varepsilon})^T A (\nabla
\overrightarrow{\phi}^{\varepsilon}\cdot\xi) \, dx \\
&\qquad = \int_{\Omega} (\overrightarrow{\rho}^{\varepsilon} - \overrightarrow{\phi}^{\varepsilon})^T A (\nabla
\overrightarrow{\phi}^{\varepsilon}\cdot\xi) \, dx \\
&\qquad = - \varepsilon^2\int_{\Omega} (\Delta\overrightarrow{\phi}^{\varepsilon})^T A (\nabla
\overrightarrow{\phi}^{\varepsilon}\cdot\xi) \, dx
\end{align*}
which then gives (after further integration by parts):
\begin{align*}
&\int_{\Omega} \frac{1}{2} (\overrightarrow{\phi}^\varepsilon)^T A \overrightarrow{\phi}^\varepsilon \operatorname{div}(\xi) \, dx + \int_{\Omega} (\overrightarrow{\rho}^{\varepsilon})^T A(\nabla
\overrightarrow{\phi}^{\varepsilon}\cdot\xi) \, dx \\
&\qquad = - \varepsilon^2\int_{\Omega} (\Delta\overrightarrow{\phi}^{\varepsilon})^T A (\nabla
\overrightarrow{\phi}^{\varepsilon}\cdot\xi) \, dx\\
&\qquad =  \varepsilon^2\int_{\Omega} (\nabla\overrightarrow{\phi}^{\varepsilon})^T A \nabla(\nabla\overrightarrow{\phi}^{\varepsilon}\cdot\xi) \, dx \\
&\qquad = \varepsilon^2\int_{\Omega} (\nabla\overrightarrow{\phi}^{\varepsilon})^T A (\nabla^2\overrightarrow{\phi}^{\varepsilon}\cdot\xi + \nabla\overrightarrow{\phi}^{\varepsilon}\cdot\nabla\xi) \, dx \\
&\qquad = \varepsilon^2\int_{\Omega} \frac{1}{2} \nabla\left((\overrightarrow{\phi}^\varepsilon)^T A \overrightarrow{\phi}^\varepsilon\right)\cdot\xi \, dx + \varepsilon^2\int_{\Omega} (\nabla\overrightarrow{\phi}^\varepsilon)^T A (\nabla\overrightarrow{\phi}^{\varepsilon}\cdot\nabla\xi) \, dx \\
&\qquad = - \varepsilon^2\int_{\Omega} \frac{1}{2} |\nabla\overrightarrow{\psi}^{\varepsilon}|^2\operatorname{div}(\xi) \, dx + \varepsilon^2 \int_{\Omega} (\nabla\overrightarrow{\psi}^{\varepsilon})^T\nabla\overrightarrow{\psi}^{\varepsilon}:\nabla\xi \, dx.
\end{align*}
In the last line, we used the decomposition $A=Q^TQ$ and the notation $\overrightarrow{\psi}^{\varepsilon}=Q\overrightarrow{\phi}^{\varepsilon}$. Therefore, we obtain
\begin{small}
\begin{align}\label{eq:RHS-term-3}
- \frac{1}{\varepsilon} \int_{\Omega} \frac{1}{2} (\overrightarrow{\phi}^\varepsilon)^T A \overrightarrow{\phi}^\varepsilon \operatorname{div}(\xi) \, dx - \frac{1}{\varepsilon} \int_{\Omega} (\overrightarrow{\rho}^{\varepsilon})^T A(\nabla
\overrightarrow{\phi}^{\varepsilon}\cdot\xi) \, dx =  \varepsilon\int_{\Omega} \frac{1}{2}|\nabla\overrightarrow{\psi}^{\varepsilon}|^2 \div{\xi} \, dx - \varepsilon\int_{\Omega}  (\nabla\overrightarrow{\psi}^{\varepsilon})^T\nabla\overrightarrow{\psi}^{\varepsilon}:\nabla\xi \, dx
\end{align}
\end{small}
and we complete the proof of \eqref{eq:computation-result} by combining \eqref{eq:expansion}-\eqref{eq:RHS-term-3}. We finish the proof of the proposition.
\end{proof}

We now prove Theorem \ref{thm:phase-separation-convergence}(iii) and (iv).

\begin{proof}[Proof of Theorem \ref{thm:phase-separation-convergence}(iii)]
By rearranging the terms of \eqref{eq:computation-result} of Proposition \ref{prop:ep-evolution}, we can write
\begin{align*}
\int_0^T\int_{\Omega} p^\varepsilon \operatorname{div}(\xi) \, dxdt &= \frac{1}{\varepsilon} \int_0^T\int_{\Omega} \left( W(\overrightarrow{\rho}^\varepsilon) + \frac{1}{2} |\overrightarrow{\psi}^\varepsilon- Q\overrightarrow{\rho}^\varepsilon |^2 + \frac{1}{2} \varepsilon^2 | \nabla \overrightarrow{\psi}^\varepsilon |^2 \right) \operatorname{div}(\xi) \, dxdt \notag \\
&\qquad - \varepsilon \int_0^T\int_{\Omega} (\nabla\overrightarrow{\psi}^{\varepsilon})^T\nabla\overrightarrow{\psi}^{\varepsilon}:D\xi \, dxdt + \int_0^T\int_{\Omega} j^\varepsilon \cdot \xi \, dxdt.
\end{align*}
Lemma \ref{lem:estimates}(i) and (iii) imply that there exists a constant  $C>0$ such that
\begin{align}\label{eq:pressure-C^1-star}
\left| \int_0^T\int_{\Omega} p^\varepsilon \operatorname{div}(\xi) \, dxdt \right| &\leq C\left( \| D\xi \|_{L^1(0, T; L^\infty(\Omega))} + \| j^\varepsilon \|_{L^2(0, T; L^1(\Omega))} \| \xi \|_{L^2(0, T; L^\infty(\Omega))}\right) \notag \\
&\leq C\| \xi \|_{L^2(0, T; C^1(\Omega))}.
\end{align}
Therefore, $\nabla p^{\varepsilon}$ is bounded in $L^2(0,T;C^1(\Omega)^*)$ uniformly in $\varepsilon\in(0,1)$.

We now proceed as in \cite{JKM21}: For $\varphi\in L^2(0,T;C^s(\Omega))$, consider the solution $u\in L^2(0,T;C^{2,s}(\Omega))$ to
\begin{align*}
\begin{cases}
\displaystyle \Delta u = \varphi - \frac{1}{|\Omega|}\int_{\Omega} \varphi \, dx & \text{in } \Omega, \\
\nabla u \cdot n = 0 & \text{on } \partial\Omega.
\end{cases}
\end{align*}
Then, by the fact that $\int_{\Omega}p^{\varepsilon}(x,t)\,dx=0$ and \eqref{eq:pressure-C^1-star}, we obtain
\begin{align*}
\left|  \int_{\Omega} p^{\varepsilon} \varphi \, dx  \right| = \left|  \int_{\Omega} p^{\varepsilon} \div (\nabla u) \, dx  \right|  &\leq C_M \| \nabla u \|_{L^{2}((0,T); C^1(\Omega))} \\
&\leq C_M \| u \|_{L^{2}((0,T); C^{2,s}(\Omega))} \leq C_M \| \varphi \|_{L^{2}((0,T); C^{s}(\Omega))}.
\end{align*}
This proves that $p^{\varepsilon}$ is bounded in $L^2(0,T;(C^s(\Omega))^*)$ uniformly in $\varepsilon\in(0,1)$ for any $s\in(0,1)$ is thus weak$^*$ convergent.
\end{proof}

\medskip

\begin{proof}[Proof of Theorem \ref{thm:phase-separation-convergence}(iv)]

By Proposition \ref{prop:convergence}(i) and Theorem \ref{thm:phase-separation-convergence}(iv), the left-hand side of \eqref{eq:motion-law-epsilon} satisfies
\begin{align}\label{eq:motion-law-LHS}
\lim_{\varepsilon\to0}\int_0^T\int_{\Omega} j^\varepsilon \cdot \xi - p^\varepsilon \operatorname{div}(\xi)\, dxdt=\int_0^T\int_{\Omega} j \cdot \xi - p \operatorname{div}(\xi)\, dxdt.
\end{align}

\medskip

Introduce the functions $u_{\varepsilon},v_{\varepsilon}\geq0$ defined by
\begin{align*}
u_{\varepsilon}^{2} = \frac{1}{\varepsilon} \left( W(\overrightarrow{\rho}^{\varepsilon}) + \frac{1}{2} |\overrightarrow{\psi}^{\varepsilon} - Q \overrightarrow{\rho}^{\varepsilon}|^{2} \right) \qquad\text{and}\qquad v_{\varepsilon}^{2} = \frac{1}{2} \varepsilon |\nabla \overrightarrow{\psi}^{\varepsilon}|^{2}. 
\end{align*}
Then, by \cite[Proposition 2.1]{B90} applied to all $i=0,\cdots,N$, we have
\begin{align*}
\bigvee_{i=0}^{N} \int_{\Omega} |\nabla \varphi_i(\overrightarrow{\psi}^{\varepsilon})| \, dx \leq \int_{\Omega} \sqrt{2h(\overrightarrow{\psi}^{\varepsilon})}  |\nabla \overrightarrow{\psi}^{\varepsilon}| \, dx &\leq \int_{\Omega} 2 u_{\varepsilon} v_{\varepsilon} \, dx \\
&= \int_{\Omega} u_{\varepsilon}^{2} + v_{\varepsilon}^{2} - (u_{\varepsilon} - v_{\varepsilon})^{2} \, dx \leq \int_{\Omega} u_{\varepsilon}^{2} + v_{\varepsilon}^{2} \, dx = \mathcal{E}^{\varepsilon}(\overrightarrow{\rho}^{\varepsilon}).
\end{align*}
Integrating in $t\in[0,T]$ and applying the assumption \eqref{assumption:energy-convergence}, and by the lower semicontinuity of the left-hand side and Fatou's Lemma, we derive
\begin{align*}
\int_0^T\cE^0(\overrightarrow{\rho}^0(t))\,dt &\leq \int_0^T\liminf_{\varepsilon\to0}\bigvee_{i=0}^{N} \int_{\Omega} |\nabla \varphi_i(\overrightarrow{\psi}^{\varepsilon})| \, dxdt \\
& \leq \liminf_{\varepsilon\to0}\int_0^T\bigvee_{i=0}^{N} \int_{\Omega} |\nabla \varphi_i(\overrightarrow{\psi}^{\varepsilon})| \, dxdt \leq \liminf_{\varepsilon\to0} \int_{0}^T\cE^{\varepsilon}(\overrightarrow{\rho}^\varepsilon(t))\,dt = \int_0^T\cE^0(\overrightarrow{\rho}^0(t))\,dt.
\end{align*}
The equality between the left and right hand side means that all these inequalities are in fact equalities in the limit, resulting in the following convergences: 
\begin{align}\label{eq:convergence-squeezing}
\left\{
\begin{aligned}
u_{\varepsilon}^2 + v_{\varepsilon}^2 - \bigvee_{i=0}^{N} |\nabla \varphi_i(\overrightarrow{\psi}^{\varepsilon})| &\longrightarrow 0  && \text{ in } L^1(\Omega\times(0,T)), \\
2v_{\varepsilon}^2 - \bigvee_{i=0}^{N} |\nabla \varphi_i(\overrightarrow{\psi}^{\varepsilon})| &\longrightarrow 0  && \text{ in } L^1(\Omega\times(0,T)), \\
\sqrt{2h(\overrightarrow{\psi}^{\varepsilon})}|\nabla\overrightarrow{\psi}^{\varepsilon}| - \bigvee_{i=0}^{N} |\nabla \varphi_i(\overrightarrow{\psi}^{\varepsilon})| &\longrightarrow 0  && \text{ in } L^1(\Omega\times(0,T)), \\
\bigvee_{i=0}^{N} |\nabla \varphi_i(\overrightarrow{\psi}^{\varepsilon})|(\Omega) &\longrightarrow \bigvee_{i=0}^{N} |\nabla (\varphi_i\circ Q\overrightarrow{\rho}^0)|(\Omega)  && \text{ in } L^1(0,T). \\
\end{aligned}
\right.
\end{align}
Therefore, by \eqref{eq:convergence-squeezing}, Proposition \ref{prop:BV-fact} and \eqref{eq:bigVd}, we obtain that
\begin{align}\label{eq:motion-law-RHS-1}
\lim_{\varepsilon\to0}\frac{1}{\varepsilon} \int_0^T\int_{\Omega} \left( W(\overrightarrow{\rho}^{\varepsilon}) + \frac{1}{2} |\overrightarrow{\psi}^{\varepsilon}-Q\overrightarrow{\rho}^{\varepsilon} |^{2} + \frac{1}{2} \varepsilon^{2} |\nabla \overrightarrow{\psi}^{\varepsilon}|^{2} \right) \operatorname{div}\xi \, dx  dt
&=\lim_{\varepsilon\to0}  \int_0^T\int_{\Omega}  \left( u_\eps^2+v_\eps^2\right) \operatorname{div}\xi \, dx  dt \nonumber \\
&=\lim_{\varepsilon\to0}  \int_0^T\int_{\Omega}  \left(\bigvee_{i=0}^{N} |\nabla \varphi_i(\overrightarrow{\psi}^{\varepsilon})|\right) \operatorname{div}\xi \, dx  dt \nonumber \\
&=  \int_0^T\int_{\Omega}  \left(\bigvee_{i=0}^{N} |\nabla (\varphi_i\circ Q\overrightarrow{\rho}^0)|\right) \operatorname{div}\xi \, dx  dt \nonumber \\
& =\sum_{i, j=0}^{N} \sigma_{ij} \int_0^T\int_{\Omega} \operatorname{div}\xi \, d\mu_{ij}dt.
\end{align}

\medskip

Next, we write
\begin{align*}
\varepsilon\int_0^T\int_{\Omega}  (\nabla\overrightarrow{\psi}^{\varepsilon})^T\nabla\overrightarrow{\psi}^{\varepsilon}:\nabla\xi \, dxdt &= \varepsilon\int_0^T\int_{\Omega}  \frac{(\nabla\overrightarrow{\psi}^{\varepsilon})^T}{|\nabla\overrightarrow{\psi}^{\varepsilon}|}\frac{\nabla\overrightarrow{\psi}^{\varepsilon}}{|\nabla\overrightarrow{\psi}^{\varepsilon}|}:\nabla\xi |\nabla\overrightarrow{\psi}^{\varepsilon}|^2\, dxdt \\
&= \int_0^T\int_{\Omega}  \frac{(\nabla\overrightarrow{\psi}^{\varepsilon})^T}{|\nabla\overrightarrow{\psi}^{\varepsilon}|}\frac{\nabla\overrightarrow{\psi}^{\varepsilon}}{|\nabla\overrightarrow{\psi}^{\varepsilon}|}:\nabla\xi \cdot 2v_{\varepsilon}^2\, dxdt
\end{align*}
and use \eqref{eq:convergence-squeezing} to conclude:
\begin{align*}
\lim_{\varepsilon\to0}\varepsilon\int_0^T\int_{\Omega}  (\nabla\overrightarrow{\psi}^{\varepsilon})^T\nabla\overrightarrow{\psi}^{\varepsilon}:\nabla\xi \, dxdt = \lim_{\varepsilon\to0} \int_0^T\int_{\Omega}  \frac{(\nabla\overrightarrow{\psi}^{\varepsilon})^T}{|\nabla\overrightarrow{\psi}^{\varepsilon}|}\frac{\nabla\overrightarrow{\psi}^{\varepsilon}}{|\nabla\overrightarrow{\psi}^{\varepsilon}|}:\nabla\xi \cdot \sqrt{2h(\overrightarrow{\psi}^{\varepsilon})}|\nabla \overrightarrow{\psi}^{\varepsilon}|\, dxdt
\end{align*}
where the latter limit exists by \cite[Proposition 3.1, (3.12)]{LS18} which gives
\begin{align}\label{eq:motion-law-RHS-2}
\lim_{\varepsilon\to0} \int_0^T\int_{\Omega}  \frac{(\nabla\overrightarrow{\psi}^{\varepsilon})^T}{|\nabla\overrightarrow{\psi}^{\varepsilon}|}\frac{\nabla\overrightarrow{\psi}^{\varepsilon}}{|\nabla\overrightarrow{\psi}^{\varepsilon}|}:\nabla\xi \cdot \sqrt{2h(\overrightarrow{\psi}^{\varepsilon})}|\nabla \overrightarrow{\psi}^{\varepsilon}|\, dxdt = \sum_{i,j=0}^N\int_0^T\int_{\Omega} \nu_{ij} \otimes \nu_{ij} : \nabla\xi\, d\mu_{ij}dt.
\end{align}
Combining \eqref{eq:motion-law-LHS}, \eqref{eq:motion-law-RHS-1}, and \eqref{eq:motion-law-RHS-2} finishes the proof.
\end{proof}

\appendix

\section*{Appendix}
\section{Well-posedness of \eqref{eq:epsilon-problem-1}-\eqref{eq:epsilon-problem-2}}\label{sec:appendix-A}
We state the existence of solutions to \eqref{eq:epsilon-problem-1}-\eqref{eq:epsilon-problem-2}.

\begin{theorem}\label{thm:well-posedness}
Let $A=(\alpha_{ij})_{1\leq i,j\leq N}$ be given (not necessarily  assuming \ref{assumption:A1} and \ref{assumption:A2}). Fix $T>0$, $\varepsilon\in(0,1)$ and let $\{\rho_{i,in}^{\varepsilon}\}_{i=1}^N$ be a set of initial data such that $\rho_{i,in}^{\varepsilon}\in L^m(\Omega,\R_{\geq0})$ with $\int_{\Omega}\rho_{i,in}^{\varepsilon}\,dx=m_i$ for $i=1,\cdots,N$. Then, there exists a solution $\overrightarrow{\rho}^{\varepsilon}=(\rho^{\varepsilon}_1,\cdots,\rho^{\varepsilon}_N)^T$ to the system \eqref{eq:epsilon-problem-1}-\eqref{eq:epsilon-problem-2} with finite time horizon $T>0$. That is, for each $i=1,\cdots,N$,
\begin{enumerate}[label=(\roman*)]
\item The density function $\rho_i^{\varepsilon}$ satisfies $\rho_i^{\varepsilon}(t)\in m_i\cP_{ac}(\Omega)$ a.e. $t>0$ and 
\begin{align*}
\rho_i^{\varepsilon}\in L^{\infty}\left(0,T;L^m(\Omega)\right)\cap C^{1/2}\left((0,T);W_0^{-1,\frac{2m}{m-1}}(\Omega)\right)
\end{align*}

\item The continuity equation
\begin{align*}
\partial_t\rho_i^{\varepsilon}+\div(j_i^{\varepsilon})=0
\end{align*}
holds with the Neumann boundary condition and initial condition $\rho_{i,in}^{\varepsilon}$ in the sense that
\begin{align*}
\int_{\Omega} \rho_{i,in}^{\varepsilon}(x)\,\zeta(x,0)\,dx
\;+\; \int_0^{\infty}\int_{\Omega} \rho_i^{\varepsilon} \,\partial_t \zeta
+ \rho_i^{\varepsilon} v_i^{\varepsilon} \cdot \nabla \zeta \,dxdt
= 0\qquad\text{for any }\zeta\in C_c^{\infty}(\overline{\Omega}\times[0,T)).
\end{align*}
Furthermore, we have $j_i^{\varepsilon}=\rho_i^{\varepsilon}v_i^{\varepsilon}$ for some $v^{\varepsilon}_i\in L^2(\Omega\times(0,T),d\rho_i^{\varepsilon};\R^d)$.

\item The total flux $j^\eps = j_1^\eps + \dots + j_N^\eps$ satisfies
\begin{align*}
\int_0^{\infty}\int_{\Omega} j^{\varepsilon} \cdot \xi \, dxdt=\frac{1}{\varepsilon }\int_0^{\infty}\int_{\Omega}
\left(-\left(f_m(\rho^{\varepsilon}) - \rho^{\varepsilon} f_m'(\rho^{\varepsilon})\right) \div \xi+\sum_{j=1}^{N} \alpha_{ij} \, \rho_i^{\varepsilon} \nabla \phi_j^{\varepsilon} \cdot \xi\right)\,dxdt
\end{align*}
 for any $\xi\in C_c^{\infty}(\overline{\Omega}\times[0,T);\R^d)$ with $\xi\cdot n=0$ on the boundary $\partial\Omega$ where  $\rho^{\varepsilon}:=\rho_1^{\varepsilon}+\cdots+\rho^{\varepsilon}_N$.

\item The function $\phi_i^{\varepsilon}$ solves \eqref{eq:epsilon-problem-2}, in the sense that $\phi_i^{\varepsilon}\in L^{\infty}(0,T;H^1(\Omega))$ and
\begin{align*}
\int_{\Omega} (\rho(\cdot, t) - \phi(\cdot, t)) \psi + \varepsilon^2 \nabla \phi(\cdot, t) \cdot \nabla \psi \, dx = 0
\end{align*}
for any $t\in[0,T]$ and $\psi\in H^1(\Omega)$.

\item The energy dissipation holds
\begin{align*}
\cE^{\varepsilon}(\overrightarrow{\rho}^{\varepsilon}(t))+\sum_{i=1}^N\frac12\int_0^t\int_{\Omega}\frac{|j_i^{\varepsilon}|^2}{\rho_i^\eps} \,dxdt\leq \cE^{\varepsilon}(\overrightarrow{\rho}^{\varepsilon}_{in})
\end{align*}
for any $t\in[0,T]$, where $\overrightarrow{\rho}^{\varepsilon}_{in}:=(\rho^{\varepsilon}_{1,in},\cdots,\rho^{\varepsilon}_{N,in})^T$.
\end{enumerate}
\end{theorem}

We can also show that $j_i^{\varepsilon}:=\rho_i^{\varepsilon}v_i^{\varepsilon}$ is given by
\begin{align*}
j_i^{\varepsilon}=\frac{1}{\varepsilon}\left(-q_i^{\varepsilon}+\sum_{j=1}^N\alpha_{ij}\rho_i^{\varepsilon}\nabla\phi_j^{\varepsilon}\right)
\end{align*}
where $q^\eps = q^\eps_1+\cdots + q^\eps_N$ satisfies $q^\eps = \rho^\eps \na f'_m(\rho^\eps)$ in the following sense:
\begin{align*}
\int_0^{\infty}\int_{\Omega} q^{\varepsilon}  \cdot \xi\,dxdt = \int_0^{\infty}\int_{\Omega}\left(f_m(\rho^{\varepsilon}) - \rho^{\varepsilon} f_m'(\rho^{\varepsilon})\right)\operatorname{div} \xi\,dxdt,
\end{align*}
but this is not be used in the study of the limit $\eps\to0$. On the other hand, as stressed in the introduction, we do not prove that the individual flux are given by
$ j_i^{\varepsilon}=\frac{1}{\varepsilon}\left(-\rho^\eps_i \na f'_m(\rho^\eps) +\sum_{j=1}^N\alpha_{ij}\rho_i^{\varepsilon}\nabla\phi_j^{\varepsilon}\right)$.

\medskip

Theorem \ref{thm:well-posedness} is proved via a discrete time scheme of JKO type as in \cite{KMW23}. We omit detailed statements and proofs  since they can be derived from  \cite{KMW23}  with natural modifications. We simply recall here the main steps:

\medskip

For $i=1,\cdots,N$, we let $K_i$ be the set of all measures that are absolutely continuous with respect to the Lebesgue measure on $\Omega$ whose density is in $L^m(\Omega;\R_{\geq0})$ with mass $m_i$. Let $\mathcal{K}=K_1\times\cdots\times K_N$ and be equipped with the Wasserstein distance in vector form defined by
\begin{align*}
W_2^2(\overrightarrow{\rho}^1,\overrightarrow{\rho}^2):=W_2^2(\rho_1^1,\rho_1^2)+\cdots+W_2^2(\rho_N^1,\rho_N^2)
\end{align*}
for $\overrightarrow{\rho}^i=(\rho_1^i,\cdots,\rho_N^i)^T\in\mathcal{K}$, $i=1,2$. Here, each individual term is the usual Wasserstein distance defined by for $\mu,\nu\in K_i$, $i=1,\cdots,N$,
\begin{align*}
W_2^2(\mu,\nu) = \inf_{\pi \in \Pi(\mu,\nu)} \int_{\Omega \times \Omega} |x - y|^2 \, d\pi(x, y)
\end{align*}
where $\Pi(\mu,\nu)$ denotes the set of all nonnegative measures with marginals $\mu$ and $\nu$.

\medskip

For fixed $\varepsilon\in(0,1)$ and  given $\tau\in(0,1)$ (the time step size that is destined to go to zero), we  define recursively a sequence $\overrightarrow{\rho}^n\in\mathcal{K}$ by
\begin{align*}
\overrightarrow{\rho}^{\tau,0} = \overrightarrow{\rho}^{\varepsilon}_{in}, \quad \overrightarrow{\rho}^{\tau,n} \in \operatorname{argmin} \left\{ \frac{1}{2\tau} W_2^2(\overrightarrow{\rho}, \overrightarrow{\rho}^{\tau,n-1}) + \cE^{\varepsilon}(\overrightarrow{\rho})\, :\, \overrightarrow{\rho} \in \mathcal{K} \right\} \quad \text{for } n \geq 0.
\end{align*}
The velocity field is then built from $T_i$, the unique optimal transport map from $\rho_i^{n}$ to $\rho_i^{n-1}$, as follows:
\begin{align*}
v_i^{\tau,n}(x):=\frac{x-T_i(x)}{\tau}. 
\end{align*}
We can define the pressure $p_i^{\tau,n}$ satisfying $p^{\tau,n}\in\partial f(\rho^{\tau,n})$ where $\rho^{\tau,n} = \rho^{\tau,n}_1 + \cdots + \rho^{\tau,n}_N$ and 
\begin{align*}
\rho_i^{\tau,n}v_i^{\tau,n}=\frac{1}{\varepsilon}\left(\sum_{j=1}^N\alpha_{ij}\rho_i^{\tau,n}\nabla\phi_j^{\tau,n} - \rho_i^{\tau,n}\nabla p^{\tau,n}\right),
\end{align*}
where $\phi_j^{\tau,n}$ is the solution to \eqref{eq:epsilon-problem-2} with $\rho_j^{\tau,n}$ in place of $\rho_i^{\varepsilon}$. We refer to \cite[Section 2]{KMW23} for the pressure.

\medskip

For each $n\in\mathbb{N}_{\geq0}$, define the piecewise constant functions by
\begin{equation*}
\begin{cases}
\overrightarrow{\rho}^{\tau}(t) &:= \overrightarrow{\rho}^{\tau,n} \quad \text{for } t \in [n\tau, (n+1)\tau), \\
\overrightarrow{\phi}^{\tau}(t) &:= \overrightarrow{\phi}^{\tau,n} \quad \text{for } t \in [n\tau, (n+1)\tau), \\
\,p^{\tau}(t)    &:= p^{\tau,n}    \quad \,\text{ for } t \in [n\tau, (n+1)\tau),\text{ and}\\
\,v_i^{\tau}(t) &:= v^{\tau,n} \quad \,\text{ for } t \in [n\tau, (n+1)\tau),\,i=1,\cdots,N.
\end{cases}
\end{equation*}
We have the following convergence theorem.

\begin{theorem}\label{thm:time-discretization}
Given $T>0$, $\varepsilon\in(0,1)$ and initial data $\overrightarrow{\rho}_{in}$ with $\cE^{\varepsilon}(\overrightarrow{\rho}_{in})<+\infty$, there is a sequence $\tau\to0$ such that the following hold.
\begin{enumerate}[label=(\roman*)]
\item The vector $\overrightarrow{\rho}^{\tau}$ converges to $\overrightarrow{\rho}\in L^{\infty}(0,T;\mathcal{K})\cap C^{1/2}(0,T;W_0^{-1,\frac{2m}{m-1}}(\Omega)^N)$ locally uniformly in $t>0$ with respect to the Wasserstein distance.

\item The vector $\overrightarrow{\phi}^{\tau}$ converges to $\overrightarrow{\phi}$ in $L^{\infty}(0,T;W^{1,2}(\Omega)^N)$.

\item The flux $j_i^{\tau}:=\rho_i^{\tau}v_i^{\tau}$ converges to $j_i$ weakly$^*$ in $L^{\frac{2m}{m-1}}(\Omega\times(0,T))$ for each $i=1,\cdots,N$. Moreover, for each $i=1,\cdots,N$, there exists $v_i\in L^2(\Omega\times(0,T);d\rho_i)$ such that $j_i=\rho_iv_i$.

\item The field $\rho_i^{\tau}\nabla p^{\tau}$ converges to $\rho_i\overrightarrow{\varphi}$ weakly$^*$ in $L^{\frac{2m}{m-1}}(\Omega\times(0,T))$ for each $i=1,\cdots,N$. 

\item Finally, the limits 
solve the equations \eqref{eq:epsilon-problem-1}-\eqref{eq:epsilon-problem-2} up to time $T>0$ in the sense of Theorem \ref{thm:well-posedness}(i)-(iv). Also, the dissipation as stated in Theorem \ref{thm:well-posedness}(v) holds.
\end{enumerate}
\end{theorem}

\section{Transition from partial to full engulfment}\label{subsec:numerics}

While formula \eqref{eq:metric-coefficient-before-change-of-variables} does not allow us to explicitly compute the adhesion coefficients $\sigma_{ij}$ for given interaction coefficients $\alpha_{ij}$, we can compute numerical values easily. 
In this appendix, we fix 
$m=4$ and $A=\left( \begin{array}{ll} 4 & \gamma \\ \gamma & 1 \end{array}\right)$. We use the notations $A$, $B$, $C$ for the vertices $\overrightarrow{\alpha_0}$, $\overrightarrow{\alpha_1}$ and $\overrightarrow{\alpha_2}$.

\begin{table}[h]
\centering
\begin{tabular}{cccccc}
\toprule
$\gamma$ & $\sigma_{AB}$ & $\sigma_{AC}$ & $\sigma_{CB}$ & Detour & Gap \\
\midrule
0.50 & 1.0848 & 0.0678 & 1.0466 & 1.1144 & 0.0296 \\
0.60 & 1.0822 & 0.0678 & 1.0214 & 1.0892 & 0.0070 \\
0.70 & 1.0628 & 0.0678 & 0.9950 & 1.0628 & 0.0000 \\
0.80 & 1.0354 & 0.0678 & 0.9677 & 1.0355 & 0.0000 \\
0.90 & 1.0071 & 0.0678 & 0.9393 & 1.0071 & 0.0000 \\
1.00 & 0.9777 & 0.0678 & 0.9100 & 0.9778 & 0.0000 \\
1.10 & 0.9475 & 0.0678 & 0.8797 & 0.9475 & 0.0000 \\
\bottomrule
\end{tabular}
\caption{The table shows the vanishing of $\sigma_{AC} + \sigma_{CB} - \sigma_{AB}$ in the last column (Gap). The second-last column (Detour) represents $\sigma_{AC} + \sigma_{CB}$.} 
\label{tab:horizontal_geodesics}
\end{table}

Table \ref{tab:horizontal_geodesics} shows numerical values for the distances between the vertices. The last column gives the defect in the triangle inequality $\sigma_{AC} + \sigma_{CB} - \sigma_{AB}$,
and suggests that there exists a critical value $\gamma_c$ such that this defect vanishes when $\gamma\geq \gamma_c$ 
(so partial engulfment occurs when $\gamma <\gamma_c$ and full engulfment occurs when $\gamma\geq \gamma_c$). Note that when $\gamma \geq \beta (=1)$, this defect is zero, which is consistent with the result of Theorem \ref{thm:DAH}(iii) (is all compatible with \cite{FBC23,MT15}).

\begin{figure}[H]
	\begin{center}
            \includegraphics[height=6.2cm]{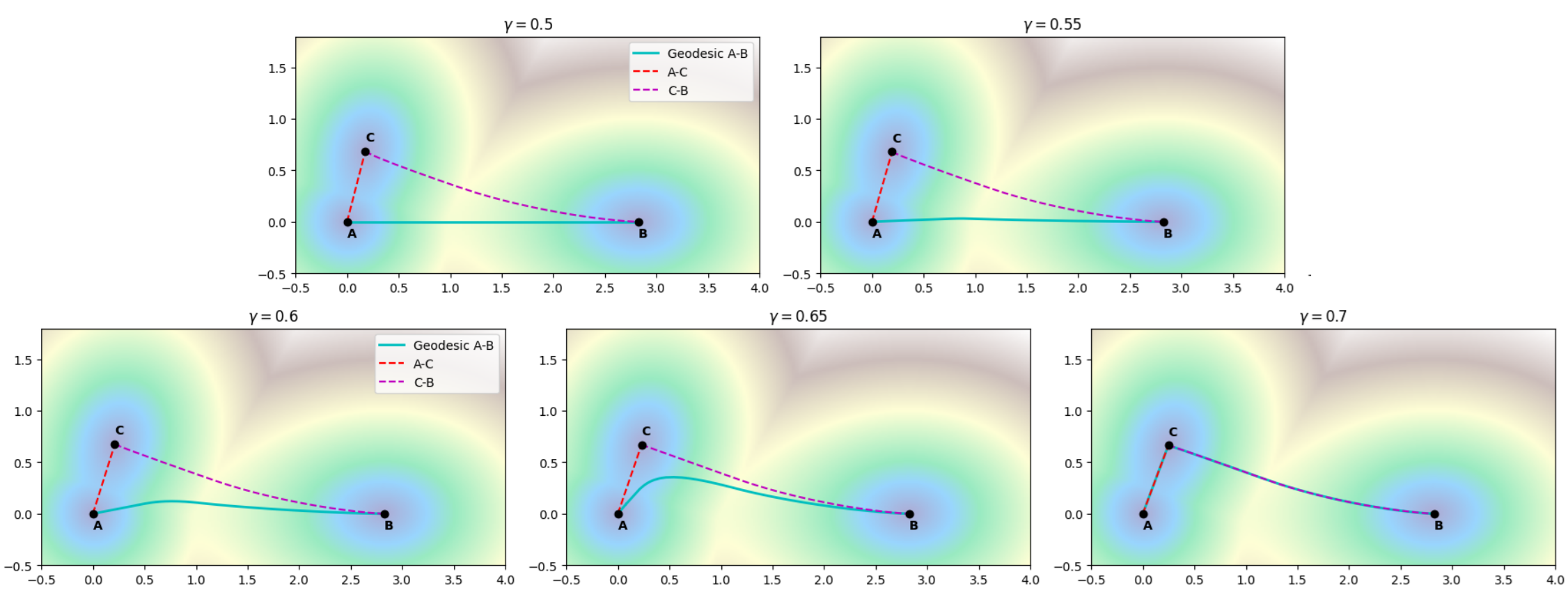}
            \captionsetup{width=0.8\textwidth} 
		\vskip 0pt
		\caption{The geodesics between the vertices are shown for $m=4$. When $\gamma=0.7$, the geodesic between $\overrightarrow{\alpha_0}$ and $\overrightarrow{\alpha_1}$ overlaps with the concatenation of those of $\overrightarrow{\alpha_0}$ and $\overrightarrow{\alpha_2}$, $\overrightarrow{\alpha_2}$ and $\overrightarrow{\alpha_1}$. The color map corresponds to the values of $h$.}
        \label{fig:geodesics}
	\end{center}
\end{figure}

Figure \ref{fig:geodesics} shows the geodesics  in the phase plane. When $\gamma=0.5$, the geodesic $\overrightarrow{\alpha_0}$ to $\overrightarrow{\alpha_1}$ is the straight line and this is pulled up to the other vertex $\overrightarrow{\alpha_2}$ as $\gamma$ increases. 
When $\gamma=0.7$,  the geodesic from $A=\overrightarrow{\alpha_0}$ to $B=\overrightarrow{\alpha_1}$ passes through the vertex $C=\overrightarrow{\alpha_2}$, which implies $\sigma_{AC} + \sigma_{CB} - \sigma_{AB} = 0$.

\begin{figure}[H]
    \centering
    
    \begin{minipage}{0.45\textwidth}
        \centering
        \small 
        \begin{tabular}{cccccc}
            \toprule
            $m$ & Lower bound & & $\gamma_c$ & & Upper bound \\ \midrule
            5   & 0.7500 & $\leqslant$ & $\gamma_c$ & $\leqslant$ & 0.8000 \\
            10  & 0.9200 & $\leqslant$ & $\gamma_c$ & $\leqslant$ & 0.9400 \\
            20  & 0.9600 & $\leqslant$ & $\gamma_c$ & $\leqslant$ & 0.9700 \\
            30  & 0.9700 & $\leqslant$ & $\gamma_c$ & $\leqslant$ & 0.9775 \\
            40  & 0.9850 & $\leqslant$ & $\gamma_c$ & $\leqslant$ & 0.9900 \\
            50  & 0.9900 & $\leqslant$ & $\gamma_c$ & $\leqslant$ & 0.9925 \\ \bottomrule
        \end{tabular}
        \vspace{1cm}
        \captionof{table}{Numerical bounds for $\gamma_c$ for various $m$.}
        \label{tab:m-higher-numerics}
    \end{minipage}
    \hfill 
    \begin{minipage}{0.5\textwidth}
        \centering
        \begin{tikzpicture}[scale=0.85] 
            \begin{axis}[
                xlabel={$m$},
                ylabel={$\gamma_c$},
                xmin=0, xmax=55,
                ymin=0.7, ymax=1.05,
                xtick={5,10,20,30,40,50},
                grid=both,
                grid style={line width=.1pt, draw=gray!10},
                major grid style={line width=.2pt,draw=gray!50},
                legend style={at={(0.5,-0.2)},anchor=north,legend columns=-1}
            ]

            \addplot[blue, mark=*, name path=lower] coordinates {
                (5, 0.7500) (10, 0.9200) (20, 0.9600) (30, 0.9700) (40, 0.9850) (50, 0.9900)
            };
            \addlegendentry{Lower bound }

            \addplot[red, mark=square*, name path=upper] coordinates {
                (5, 0.8000) (10, 0.9400) (20, 0.9700) (30, 0.9775) (40, 0.9900) (50, 0.9925)
            };
            \addlegendentry{Upper bound}

            \addplot[gray!20, opacity=0.5] fill between[of=lower and upper];

            \end{axis}
        \end{tikzpicture}
        \captionof{figure}{The graph of bounds $\gamma_c$ in $m$.}
        \label{fig:m-higher-numerics}
    \end{minipage}

\end{figure}

Finally Table \ref{tab:m-higher-numerics}  and Figure \ref{fig:m-higher-numerics} illustrate the dependence of this critical $\gamma_c$ with $m$. We only compute lower and upper bounds of the transition values $\gamma_c$ for various values of $m$ as these are enough to indicate the general trend: $\gamma_c$ depends on $m$, is increasing and approaches $1$ as $m\to\infty$. This is consistent with Theorem \ref{thm:DAH}(iii) which states that $\gamma_c\leq 1$ for all $m>2$.

\section{Some properties of BV functions}
We recall the following classical result (see for instance \cite{AFP00,KMW24}):
\begin{proposition}\label{prop:BV}
Let $f_k$ be a sequence of functions such that $f_k \to f$ in $L^1(\Omega)$ when $k\to\infty$.
Then
$$\liminf_{k\to\infty} \int_\Omega \zeta(x)  |\na f_k|
\geq  \int_\Omega \zeta(x) |\na f|  
$$
for all $\zeta \in C(\Omega)$ with $\zeta \geq 0$.
Furthermore, if 
$ |\na f_k| (\Omega)\to   |\na f|(\Omega)$, 
then
$$
\lim_{k\to\infty} \int_\Omega\zeta (x) |\na f_k|    
= \int_\Omega\zeta(x)  |\na f| \qquad 
\mbox{ for all $\zeta\in C(\Omega)$}.$$
\end{proposition}
A similar result (with a similar proof) holds for the supremum of measures:
\begin{proposition}\label{prop:BV-fact}
Let $f_k^1$ and $f_k^2$ be two sequences of functions such that $f_k^i \to f^i$ in $L^1(\Omega)$ as $k \to \infty$ for $i=1, 2$. Let $\mu_k = |\nabla f_k^1| \vee |\nabla f_k^2|$ and $\mu = |\nabla f^1| \vee |\nabla f^2|$ be the corresponding supremum measures. Then
\[
\liminf_{k \to \infty} \int_{\Omega} \zeta(x) \, d\mu_k \geq \int_{\Omega} \zeta(x) \, d\mu\qquad\text{for all $\zeta \in C(\Omega)$ with $\zeta \geq 0$. }
\]
Furthermore, if $\mu_k(\Omega) \to \mu(\Omega)$, then
\[
\lim_{k \to \infty} \int_{\Omega} \zeta(x) \, d\mu_k = \int_{\Omega} \zeta(x) \, d\mu \qquad \textit{for all } \zeta \in C(\Omega).
\]
\end{proposition}

\medskip

\medskip

\noindent{\bf Data availability statement:}
No data are associated with this article.
\medskip

\noindent{\bf Conflict of Interest Statement:}
The authors declare that they have no conflict of interest.

\bibliographystyle{plain}
\bibliography{Preprint}

\end{document}